\documentclass[a4paper,reqno]{amsart}

\usepackage{amsfonts}
\usepackage{amsmath,amssymb,amsthm,amsthm}
\usepackage{mathtools}
\usepackage{etoolbox}
\usepackage{mathrsfs}
\usepackage{hyperref}
\usepackage[english]{babel}
\usepackage[utf8]{inputenc}
\usepackage{fullpage}
\usepackage{verbatim}
\usepackage{sidecap}
\usepackage{xcolor}
\usepackage{accents}
\usepackage{multirow}
\usepackage{cases}
\usepackage{empheq}
\usepackage{microtype}
	
\usepackage{color, colortbl}

\DisableLigatures{encoding = *, family = * }

\newtheorem{theorem}{Theorem}[section]
\newtheorem{proposition}{Proposition}[section]

\newtheorem{remark}{Remark}[section]

\newcommand{\NN}{\mathbb{N}}
\newcommand{\ZZ}{\mathbb{Z}}
\newcommand{\RR}{\mathbb{R}}
\newcommand{\CC}{\mathbb{C}}
\newcommand{\norm}[2]{\left\|#1\right\|_{#2}}
\newcommand{\C}{\mathcal C}
\newcommand{\PP}{\mathcal P}

\newcommand{\fwd}{\partial_h^+}
\newcommand{\bwd}{\partial_h^-}

\newcommand{\Dh}{\partial_h}

\newcommand{\Dhc}{\Delta_{c,h}}
\newcommand{\RT}{\widetilde{\mathcal R}}
\newcommand{\BF}[1]{\boldsymbol{#1}}
\newcommand{\acc}[1]{\accentset{\circ}{#1}}
\newcommand{\bx}{{\boldsymbol{x}}}
\newcommand{\bxi}{{\boldsymbol{\xi}}}
\newcommand{\bj}{{\boldsymbol{j}}}

\definecolor{Gray}{gray}{0.75}

\numberwithin{equation}{section}

\makeatletter
\@namedef{subjclassname@2020}{
	\textup{2020} Mathematics Subject Classification}
\makeatother

\begin{document}
	
\title{Gaussian Beam ansatz for finite difference wave equations} 
	
\author[U. Biccari]{Umberto Biccari}
\address[U. Biccari]{Chair of Computational Mathematics, Fundaci\'on Deusto,	Avenida de las Universidades, 24, 48007 Bilbao, Basque Country, Spain.}
\email[]{umberto.biccari@deusto.es, u.biccari@gmail.com}

\thanks{E. Zuazua has been funded by the Alexander von Humboldt-Professorship program, the ModConFlex Marie Curie Action, HORIZON-MSCA-2021-DN-01, the COST Action MAT-DYN-NET, the Transregio 154 Project ``Mathematical Modelling, Simulation and Optimization Using the Example of Gas Networks'' of the DFG, grants PID2020-112617GB-C22 and TED2021-131390B-I00 of MINECO (Spain), and by the Madrid Goverment -- UAM Agreement for the Excellence of the University Research Staff in the context of the V PRICIT (Regional Programme of Research and Technological Innovation)}
	
\author[E. Zuazua]{Enrique Zuazua}
\address[E. Zuazua]{Friedrich-Alexander-Universit\"at N\"urnberg, Department of Mathematics, Chair for Dynamics, Control, Machine Learning and Numerics (Alexander von Humboldt Professorship), Cauerstr. 11, 91058 Erlangen, Germany.
	\newline \indent 
	Chair of Computational Mathematics, Fundaci\'on Deusto,	Avenida de las Universidades, 24, 48007 Bilbao, Basque Country, Spain. 
	\newline \indent
	Universidad Aut\'onoma de Madrid, Departamento de Matem\'aticas, Ciudad Universitaria de Cantoblanco, 28049 Madrid, Spain.}
\email{enrique.zuazua@fau.de}

\begin{abstract}

This work is concerned with the construction of Gaussian Beam (GB) solutions for the numerical approximation of wave equations, semi-discretized in space by finite difference schemes. GB are high-frequency solutions whose propagation can be described, both at the continuous and at the semi-discrete levels, by microlocal tools along the bi-characteristics of the corresponding Hamiltonian. Their dynamics differ in the continuous and the semi-discrete setting, because of the high-frequency gap between the Hamiltonians. In particular, numerical high-frequency solutions can exhibit spurious pathological behaviors, such as lack of propagation in space, contrary to the classical space-time propagation properties of continuous waves. This gap between the behavior of  continuous and numerical waves introduces also significant analytical difficulties, since classical GB constructions cannot be immediately extrapolated to the finite difference setting, and need to be properly tailored to accurately detect the propagation properties in discrete media. Our main objective in this paper is to present a general and rigorous construction of the GB ansatz for finite difference wave equations, and corroborate this construction through accurate numerical simulations.

\end{abstract}
	
\maketitle 

\tableofcontents

\section{Introduction}

This article deals with the construction of Gaussian Beam (GB) solutions for the numerical approximation of wave equations, semi-discretized in space by finite difference schemes.

GB are high-frequency quasi-solutions of wave-like equations concentrated on ray paths, trajectories of the underlying Hamiltonian system, whose amplitudes at any given time are nearly Gaussian distributions up to some small error. These waves propagate in a very simple fashion, and it is possible to construct them rather explicitly. Moreover, one can use them as fundamental building blocks of wave motion, to study general solutions of PDE and their propagation properties. 

Our model of reference in this paper will be the following constant coefficients wave equation defined on $\RR^d$, $d\geq 1$:
\begin{align}\label{eq:mainEqR}
	\begin{cases}
		u_{tt}(\bx,t) -c\Delta u(\bx,t) = 0, & (\bx,t)\in \RR^d\times (0,T)
		\\
		u(\bx,0) = u_0(\bx), \;\; u_t(\bx,0) = u_1(\bx), & \bx \in\RR^d.
	\end{cases}
\end{align}

It is well-known (see, e.g., \cite{ralston1982gaussian}) that there exist GB solutions of \eqref{eq:mainEqR} whose energy is localized near certain curves $\Gamma:(\bx(t),t)\in\RR^d\times\RR$ in space-time, the so-called (bi-characteristic) rays, solutions of a Hamiltonian system of ordinary differential equations. In fact, given a ray path $(\bx(t),t)$ it is possible to construct a sequence of quasi-solutions $(u^k)_{k\in\NN}$ of the wave equation \eqref{eq:mainEqR} such that the amount of their energy outside a small cylinder centered at $\bx(t)$ is exponentially small. 

In this work, we extend this analysis to semi-discrete approximations of \eqref{eq:mainEqR}, obtained by means of a finite difference scheme. To this end, let 
\begin{align}\label{eq:mesh}
	\mathcal G^h:= \Big\{\bx_\bj := \bj h,\,\bj\in\ZZ^d\,\Big\}
\end{align}
be a uniform mesh of size $h$, $u_\bj(t):=u(\bx_\bj,t)$ and 
\begin{align*}
	\Dhc u_\bj(t):= \frac{c}{h^2}\sum_{i=1}^d \Big(u_{\bj+\BF{e}_i}-2u_\bj+u_{\bj-\BF{e}_i}\Big)
\end{align*}
be the finite difference approximation on $\mathcal G^h$ of the Laplacian $c\Delta$, where $(\BF{e}_i)_{i=1}^d$ denotes the canonical basis in $\RR^d$. With this notation, let us consider the following semi-discrete approximation of the wave equation \eqref{eq:mainEqR}
\begin{align}\label{eq:mainEqFDintro}
	\begin{cases}
		u_\bj''(t)-\Dhc u_\bj(t) = 0, & \bj\in\ZZ^d,\;\;t\in(0,T)
		\\
		u_\bj(0) = u_\bj^0, \;\;\; u_\bj'(0) = u_\bj^1, & \bj\in\ZZ^d.
	\end{cases}
\end{align}

As illustrated through numerical simulations in \cite{biccari2020propagation,marica2015propagation}, also for \eqref{eq:mainEqFDintro} - at least in space dimension $d=1,2$ - there exist quasi-solutions concentrated along the corresponding bi-characteristic rays. However, the propagation properties of these solutions change substantially with respect to their continuous counterpart, due to relevant changes in the dynamical behavior of the rays. In this paper, we are going to show that these propagation properties can be completely understood by properly constructing a GB ansatz for \eqref{eq:mainEqFDintro}. 

This paper is organized as follows. In Section \ref{sec:motivation}, we motivate our study and discuss some existing bibliography related with our work. In Section \ref{sec:Hamiltonian}, we introduce the Hamiltonian systems for the continuous and finite difference wave equations, and discuss the main differences between the corresponding rays of geometric optics. In Section \ref{sec:GBcontinuous}, we briefly recall the construction of GB solutions for the continuous wave equation, for which we will follow the nowadays classical approach of \cite{ralston1982gaussian} (see also \cite{macia2002lack}). In Section \ref{sec:GBdiscrete}, we adapt this construction to produce a GB ansatz for the semi-discrete problem \eqref{eq:mainEqFDintro}. In Section \ref{sec:numerics}, we present and discuss some numerical simulations corroborating our theoretical results. Finally, in Section \ref{sec:conclusions}, we gather our conclusions and present some open problems related to our work.

\section{Motivations and bibliographical discussion}\label{sec:motivation}
The computation of GB for wave-like equations is a very classical technique, dating back at least to the early 1970s, when it was originally employed to study resonances in lasers \cite{arnaud1973hamiltonian,babich1972asymptotic}. Since then, this approach has branched out to an ample fan of different fields.

In \cite{hormander1971existence,ralston1982gaussian}, these constructions have been used to understand the propagation of singularities in PDE. Later on, GB have been employed for the resolution of high frequency waves near caustics, with geophysical applications, for instance to model the seismic wave field \cite{cerveny1982computation} or to study seismic migration \cite{hill2001prestack}. More recent works in this direction include studies of gravity waves \cite{tanushev2007mountain}, the semi-classical Schr\"odinger equation \cite{jin2008gaussian,leung2009eulerian}, or acoustic wave equations \cite{motamed2010taylor,tanushev2008superpositions}. Finally, GB have also been widely employed by the control theory community, to describe the observability and controllability properties of wave-like equations. An incomplete literature on applications of GB to control includes \cite{bardos1992sharp,burq1997condition,burq1998controle,macia2002lack}.

GB are closely related to geometric optics, also known as the Wentzel-Kramers-Brillouin (WKB) method or ray-tracing \cite{brillouin1926mecanique,engquist2003computational,keller1962geometrical,runborg2007mathematical}. In both approaches, the solution of the PDE is assumed to be of the form
\begin{align*}
	u^k(x,t) = k^{-\frac 34}a(\bx,t)e^{ik\phi(\bx,t)}, \quad (\bx,t)\in\RR^d\times\RR,
\end{align*}
where $k\in\NN^\star = \NN\setminus\{0\}$ is the high-frequency parameter, $a$ is the amplitude function, and $\phi$ is the phase function. What differentiates geometric optics and GB are the assumptions one makes on the phase. 

In the geometric optics method, $\phi$ is assumed to be a real valued function. This, however, has a main drawback, since solving the equation for the phase using the method of characteristics may lead to singularities which invalidate the approximation. Generally speaking, this breakdown occurs at the intersection of nearby rays, resulting in a caustic where geometric optics incorrectly predicts that the solution's amplitude is infinite. 

GB, on the contrary, are built with a complex valued phase and do not develop caustics. Intuitively speaking, this is because these solutions are concentrated on a single ray that cannot self-intersect. Mathematically, this stems from the fact that the standard symplectic form and its complexification are preserved along the flow defined by the Hessian matrix of the complex valued phase $\phi$ \cite{ralston1982gaussian,tanushev2008superpositions}. Thus, GB are global solutions of the PDE, and represent a very effective tool to understand wave propagation. 

The construction of GB is nowadays well-understood for a large class of PDE. Starting from the previously mentioned works \cite{babich1972asymptotic,hormander1971existence}, these techniques have been later extended by several authors, and complemented with the developments of tools like microlocal defect measures (introduced independently by G\'erard in \cite{gerard1991microlocal} and by Tartar in \cite{tartar1990h}, in the context of nonlinear partial differential equations and of homogenization, respectively) or Wigner measures \cite{lions1993mesures,markowich1994wigner,wigner1932quantum}.

Nevertheless, these kinds of approaches are still only partially developed at the numerical level. If, on the one hand, we can mention some  works on the extension of microlocal techniques to the study of the propagation properties for discrete waves \cite{macia2002propagacion,marica2015propagation}, on the other hand, we are not aware of any contribution on the GB construction for discretized PDE.

This paper aims at filling this gap, by analyzing the GB approximation of finite difference wave equations and using them to describe their propagation properties. 

The analysis of propagation properties of numerical waves obtained through a finite difference discretization on uniform or non-uniform meshes is a topic which has been extensively investigated in the literature. Among other contributions, we mention works of Trefethen \cite{trefethen1982group,trefethen1982wave} and Vichnevetsky \cite{vichnevetsky1980propagation,vichnevetsky1981energy,vichnevetsky1981propagation,vichnevetsky1987wave,vichnevetsky1982fourier}, as well as the survey paper \cite{zuazua2005propagation}. In particular, it is by now well-known that the finite difference discretization of hyperbolic equations introduces spurious high-frequency solutions with pathological propagation behaviors that are not detected in their continuous counterpart. 

The employment of GB helps understanding the reason of this discrepancy. In fact, knowing that GB solutions remain concentrated along bi-characteristic rays allows to completely describe the propagation properties of numerical solutions and detect the pathologies that the discretization introduces. At the continuous level, if the coefficients of the equation are constants, the bi-characteristic rays are straight lines and travel with a uniform velocity. In the case of variable coefficients, instead, the heterogeneity of the medium where waves propagate produces the bending of the rays and, consequently, an increase or decrease in their velocity.

On the other hand, the finite difference space semi-discretization of the equation may introduce different dynamics, with a series of unexpected propagation properties at high frequencies, that substantially differ from the expected behavior of the continuous equation. For instance, one can generate spurious solutions traveling at arbitrarily small velocities \cite{trefethen1982group} which, therefore, show lack of propagation in space. 

As we shall see, this phenomenon is related to the particular nature of the discrete group velocity which, in contrast with the continuous equation, may vanish at certain frequencies. In addition, the introduction of a non-uniform mesh for the discretization of the equation makes the situation even more intricate. For instance, as indicated in \cite{biccari2020propagation,marica2015propagation,vichnevetsky1987wave}, for some numerical grids the rays of geometric optics may present internal reflections, meaning that the waves change direction without hitting the boundary of the domain where they propagate.

All these pathologies are purely numerical, and they are related to changes in the Hamiltonian system giving the equations of the rays. In  \cite{biccari2020propagation,macia2002propagacion,marica2015propagation}, a complete discussion of these spurious dynamics has been carried out by means of microlocal techniques, supported by sharp numerical simulations. This work complements the aforementioned contributions, by providing a simple yet accurate GB ansatz to understand wave propagation in the finite difference setting. We anticipate that our analysis will be conducted mostly in the frequency regime $k\sim h^{-1}$ (that is the one at which the interesting pathologies discussed in \cite{biccari2020propagation,macia2002propagacion,marica2015propagation} appear), although we will also present an heuristic study of other relevant ranges of frequency.  

\section{Hamiltonian system and rays of geometric optics}\label{sec:Hamiltonian}

In this section, we discuss briefly the behavior of the rays of geometric optics associated with the continuous and semi-discrete wave equations \eqref{eq:mainEqR} and \eqref{eq:mainEqFDintro}. 

\subsection{Rays of geometric optics for the continuous wave equation}

The rays of geometric optics are defined as the projections on the physical space $(\bx,t)$ of the bi-characteristic rays given by the Hamiltonian system associated with the principal symbol of the wave operator. In the case of the wave equation \eqref{eq:mainEqR}, this principal symbol is given by 
\begin{align}\label{eq:PScontinuous}
	\PP(\bx,t,\bxi,\tau): = -\tau^2 + c|\bxi|^2.
\end{align}
and the bi-characteristic rays are the curves 
\begin{align*}
	s\mapsto (\bx(s),t(s),\bxi(s),\tau(s))\in\RR^d\times\RR\times\RR^d\times\RR
\end{align*}
solving the first-order ODE system:
\begin{align}\label{eq:HamiltonianSyst}
	\begin{cases}
		\dot\bx(s) = \nabla_\bxi\PP(\bx(s),t(s),\bxi(s),\tau(s)) & \quad \bx(0) = \bx_0
		\\
		\dot t(s) \;= \PP_\tau(\bx(s),t(s),\bxi(s),\tau(s)) & \quad t(0) \;= t_0
		\\
		\dot\bxi(s) \,= -\nabla_\bx\PP(\bx(s),t(s),\bxi(s),\tau(s)) & \quad \bxi(0)\,=\bxi_0\neq \BF{0}
		\\
		\dot\tau(s) \,= -\PP_t(\bx(s),t(s),\bxi(s),\tau(s)) & \quad \tau(0) \,= \tau_0 
	\end{cases}
\end{align}
with initial data $(\bx_0,t_0,\bxi_0,\tau_0)\in\RR^d\times\RR\times\RR^d\times\RR$ such that 
\begin{align}\label{eq:initDataContinuous}
	\PP(\bx_0,t_0,\bxi_0,\tau_0) = 0. 
\end{align}

In what follows, without losing generality, we will assume that $t_0=0$. Then, we immediately obtain from \eqref{eq:PScontinuous} and \eqref{eq:HamiltonianSyst} the new system
\begin{align}\label{eq:HamiltonianSystNew}
	\begin{cases}
		\dot\bx(s) = 2c\bxi_0 & \quad \bx(0) = \bx_0
		\\
		t(s) \;= -2\tau_0s 
		\\
		\bxi(s) \,= \bxi_0
		\\
		\tau(s) \,= \tau_0 
	\end{cases}
\end{align}
from which, inverting the variables $s$ and $t$ in the second equation, we find the following expression for the ray $\bx(t)$
\begin{align}\label{eq:ray_prel}
	\bx(t) = \bx_0 - \frac{c\bxi_0}{\tau_0}t.
\end{align}
Moreover, from \eqref{eq:initDataContinuous}, we have that the initial value $\tau_0$ has to be chosen such that
\begin{align}\label{eq:tau0cond}
	\tau_0^2 = c|\bxi_0|^2.
\end{align}
Using this in \eqref{eq:ray_prel}, we finally obtain 
\begin{align}\label{eq:ray}
	\bx^\pm(t) = \bx_0 \pm \sqrt{c}\frac{\bxi_0}{|\bxi_0|}t,
\end{align}
i.e., the rays $\bx^\pm(t)$ are straight lines which propagate from $\bx_0$ with constant velocity $\sqrt{c}$ and in the direction prescribed by the unitary vector $\bxi_0/|\bxi_0|\in\RR^d$. Notice that, in space dimension $d=1$, this is consistent with the D'Alambert's formula, according to which, given the initial data $(u_0,u_1)$, the corresponding solution of \eqref{eq:mainEqR} can be uniquely decomposed into two components - each one propagating along one of the characteristics $x^\pm(t)$ - and is given by 
\begin{align*}
	u(x,t) = \frac 12 \Big[u_0(x+\sqrt{c}t) + u_0(x-\sqrt{c}t)\Big] + \frac 12 \int_{x-\sqrt{c}t}^{x+\sqrt{c}t} u_1(z)\,dz.
\end{align*}

\subsection{Rays of geometric optics for the semi-discrete wave equation}
When considering finite difference approximations of \eqref{eq:mainEqR}, namely \eqref{eq:mainEqFDintro}, the principal symbol of the wave operator becomes
\begin{align}\label{eq:PSFD}
	\PP_{fd}(\bx,t,\bxi,\tau): = -\tau^2 + 4c\left|\sin\left(\frac{\bxi}{2}\right)\right|^2 = -\tau^2 + 4c\sum_{i=1}^d \sin^2\left(\frac{\xi_i}{2}\right),
\end{align}
where we have denoted by $\xi_i$, $i\in\{1,\ldots,d\}$, the $i$-th component of the vector $\bxi\in\RR^d$, while with the notation $|\cdot|$ we refer to the classical Euclidean norm. 

This trigonometric symbol \eqref{eq:PSFD} can be easily inferred by taking in \eqref{eq:mainEqFDintro} plane wave solutions of the form 
\begin{align}\label{eq:planWaveFD}
	u_\bj(t) = e^{i\left(\frac{\tau t}{h} + \bxi\cdot\bj\right)} = e^{\frac ih\left(\tau t + \bxi\cdot\bx_\bj\right)},
\end{align}
where $\tau$ is the temporal frequency and $\bxi$ the spatial frequency (also known as the wave number).

Notice that these solutions \eqref{eq:planWaveFD} are in a high-frequency regime of order $h^{-1}$. This results in a principal symbol \eqref{eq:PSFD} which is independent of the mesh parameter $h$. 

At this regard, we shall stress that in some classical references (see, e.g., \cite{trefethen1982wave}) the authors consider plane waves in a uniform (independent of $h$) frequency regime, i.e.
\begin{align*}
	u_\bj(t) = e^{i\left(\tau t + \bxi\cdot\bx_\bj\right)},
\end{align*}
whose associate principal symbol 
\begin{align}\label{eq:FDsymbol_full}
	\PP_{fd,h}(\bx,t,\bxi,\tau): = -\tau^2 + \frac{4c}{h^2}\left|\sin\left(\frac{h\bxi}{2}\right)\right|^2 = -\tau^2 + \frac{4c}{h^2}\sum_{i=1}^d \sin^2\left(\frac{h\xi_i}{2}\right)
\end{align}
depends explicitly on $h$. But since one of the motivations of our study is to provide a deeper understanding of the high-frequency pathologies of finite-difference wave propagation, it is more natural to work in the high-frequency regime of \eqref{eq:planWaveFD}.

There is a clear substantial difference in this symbol \eqref{eq:PSFD} with respect to its continuous counterpart \eqref{eq:PScontinuous}. 
In \eqref{eq:PSFD}, the Fourier symbol $|\bxi|^2$ of the Laplace operator is replaced by 
\begin{align*}
	4c\sum_{i=1}^d \sin^2\left(\frac{\xi_i}{2}\right),
\end{align*}
corresponding to the finite difference approximation of the second-order space derivative. This affects also the behavior of the bi-characteristic rays, that are now given by the curves 
\begin{align*}
	s\mapsto (\bx_{fd}(s),t(s),\bxi_{fd}(s),\tau(s))\in\RR^d\times\RR\times\RR^d\times\RR
\end{align*}
solving the first-order ODE system
\begin{align}\label{eq:HamiltonianSystFD}
	\begin{cases}
		\dot\bx_{fd}(s) = \nabla_\bxi\PP_{fd}(\bx_{fd}(s),t(s),\bxi_{fd}(s),\tau(s)) & \quad \bx_{fd}(0) = \bx_0
		\\ 
		\dot t(s) = \partial_\tau\PP_{fd}(\bx_{fd}(s),t(s),\bxi_{fd}(s),\tau(s)) & \quad t(0) = t_0
		\\
		\dot\bxi_{fd}(s) = -\nabla_\bx\PP_{fd}(\bx_{fd}(s),t(s),\bxi_{fd}(s),\tau(s)) & \quad \bxi_{fd}(0)=\bxi_0\neq \BF{0}
		\\
		\dot\tau(s) = -\partial_t\PP_{fd}(\bx_{fd}(s),t(s),\bxi_{fd}(s),\tau(s)) = 0 & \quad \tau(0) = \tau_0 
	\end{cases}
\end{align}
with initial data $(\bx_0,t_0,\bxi_0,\tau_0)\in\RR^d\times\RR\times\RR^d\times\RR$ such that 
\begin{align}\label{eq:initDataFD}
	\PP_{fd}(\bx_0,t_0,\bxi_0,\tau_0) = 0. 
\end{align}

Once again, without losing generality, we will assume that $t_0=0$. Then, we immediately obtain from \eqref{eq:PSFD} and \eqref{eq:HamiltonianSystFD} the new system
\begin{align}\label{eq:HamiltonianSystNewFD}
	\begin{cases}
		\dot\bx_{fd}(s) = 2c\sin(\bxi_0) & \quad \bx(0) = \bx_0
		\\
		t(s) = -2\tau_0s 
		\\
		\bxi_{fd}(s) = \bxi_0
		\\
		\tau(s) = \tau_0 
	\end{cases}
\end{align}
from which, inverting the variables $s$ and $t$ in the second equation, we find the following expression for the ray $\bx_{fd}(t)$
\begin{align}\label{eq:rayFD_prel}
	\bx_{fd}(t) = \bx_0 - \frac{c\sin(\bxi_0)}{\tau_0}t.
\end{align}
Moreover, from \eqref{eq:initDataFD}, we have that the initial value $\tau_0$ has to be chosen such that 
\begin{align*}
	\tau_0^2 = 4c\left|\sin\left(\frac{\bxi_0}{2}\right)\right|^2.
\end{align*}
Using this in \eqref{eq:rayFD_prel}, we then obtain 
\begin{align*}
	\bx^\pm_{fd}(t) = \bx_0 \pm \frac{\sqrt{c}\sin(\bxi_0)}{2\left|\sin\left(\frac{\bxi_0}{2}\right)\right|}t.
\end{align*}
Finally, by observing that
\begin{align*}
	\sin(\bxi_0) = 2\sin\left(\frac{\bxi_0}{2}\right)\odot \cos\left(\frac{\bxi_0}{2}\right),
\end{align*}
where $\odot$ denotes the standard Hadamard product \footnote{We recall that, given two vectors $\BF{v} = (v_1,v_2,\ldots,v_d)\in\RR^d$ and $\BF{w} = (w_1,w_2,\ldots,w_d)\in\RR^d$, their Hadamard product is the vector $\BF{v}\odot\BF{w} = \BF{z} = (z_1,z_2,\ldots,z_d)\in\RR^d$ with $z_i=v_iw_i$ for all $i\in\{1,\ldots,d\}$.}, we get the following expression for the characteristic rays
\begin{align}\label{eq:rayFD}
	\bx^\pm_{fd}(t) = \bx_0 \pm \sqrt{c}\cos\left(\frac{\bxi_0}{2}\right)\odot \frac{\sin\left(\frac{\bxi_0}{2}\right)}{\left|\sin\left(\frac{\bxi_0}{2}\right)\right|}t.
\end{align}

We then see that the rays $\bx^\pm_{fd}(t)$ are still straight lines originating from $\bx_0$ and propagating in the direction of the unitary vector
\begin{align*}
	\frac{\sin\left(\frac{\bxi_0}{2}\right)}{\left|\sin\left(\frac{\bxi_0}{2}\right)\right|},
\end{align*}
but this time with a velocity of propagation (also denoted \textbf{group velocity} in some classical references - see \cite{trefethen1982group,vichnevetsky1981energy})
\begin{align*}
	v = \sqrt{c}\left|\cos\left(\frac{\bxi_0}{2}\right)\right| 
\end{align*}
which is not constant anymore. Instead, it depends on the frequency $\bxi_0$ and vanishes whenever $\big|\cos(\bxi_0/2)\big| = 0$. This happens, for instance, if
\begin{align*}
	\bxi_0 = (2\BF{k}+\BF{1})\pi,\quad\text{ with } \BF{k}\in\ZZ^d \text{ and }\BF{1} = (1,1,\ldots,1)\in\RR^d.
\end{align*} 
In Figure \ref{fig:rayVel}, we display this phenomenon in the one-dimensional case $d=1$.
\begin{figure}[h]
	\includegraphics[width=0.85\textwidth]{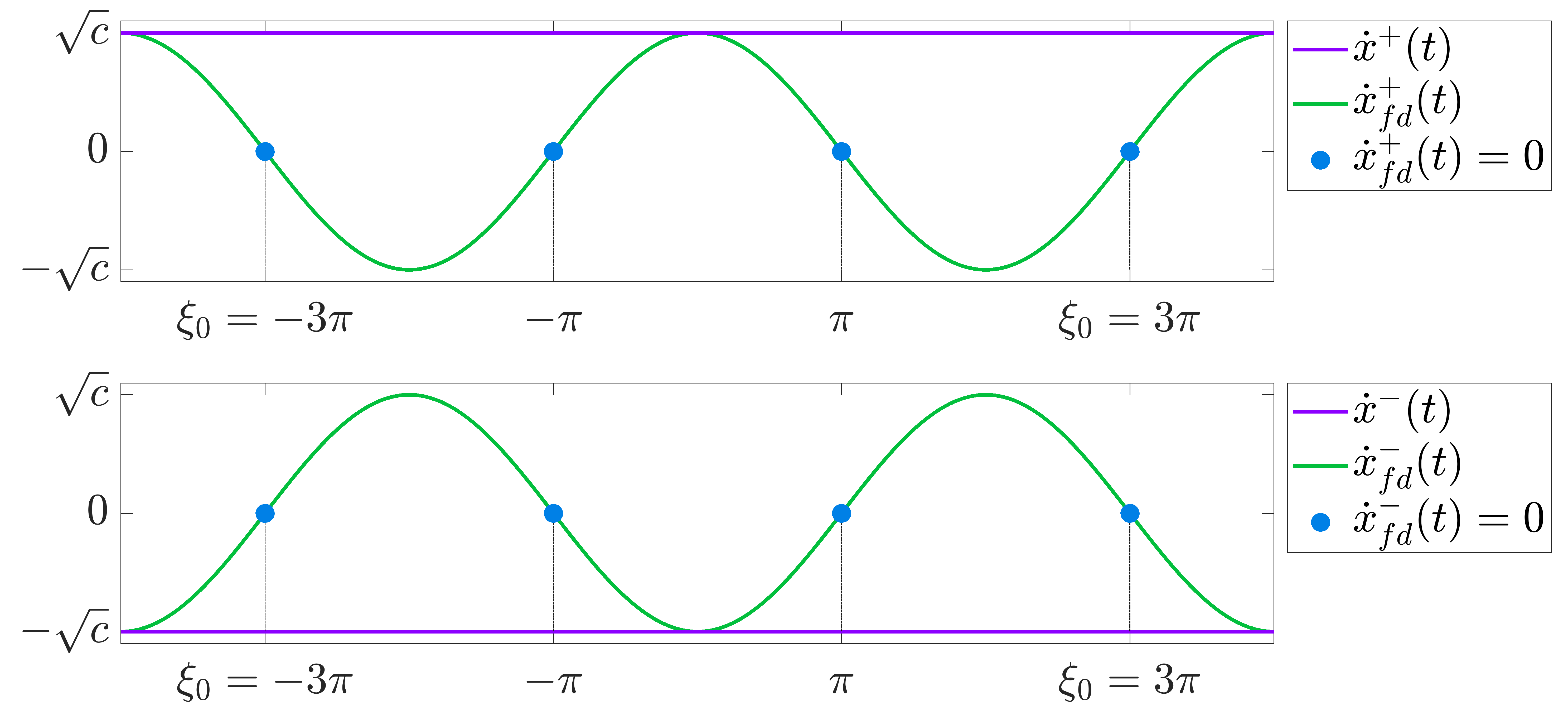}
	\caption{Velocity of propagation of $x^\pm(t)$ and $x^\pm_{fd}(t)$ for $\xi_0\in[-4\pi,4\pi]$ in space dimension $d=1$. The blue dots indicate the values of $\xi_0$ for which $\dot x^\pm_{fd}(t)=0$.}\label{fig:rayVel}
\end{figure}

This possibility of a zero velocity of propagation generates unexpected dynamical behaviors and spurious high-frequency solutions of \eqref{eq:mainEqFDintro}, that are not observed at the continuous level and shall be duly taken into account when addressing the construction of GB.

\section{Gaussian Beams for the continuous wave equation}\label{sec:GBcontinuous}

In this section, we give an abridged presentation of the construction of GB solutions for the wave equation \eqref{eq:mainEqR}. This construction being nowadays very classical, in what follows we shall only recall its main steps. Complete details can be found, e.g., in \cite{macia2002lack,ralston1982gaussian}.

\subsection{The GB ansatz} 
Given a ray $\bx(t)$ as in \eqref{eq:ray}, our objective is to generate approximate solutions of equation \eqref{eq:mainEqR} with energy
\begin{align}\label{eq:energy}
	E_c(u(\cdot,t)) = \frac 12 \int_{\RR^d} \Big[|u_t(\bx,t)|^2 + c|\nabla u(\bx,t)|^2\Big]\,d\bx,
\end{align}
concentrated on $\bx(t)$ for every $t\in(0,T)$. These solutions will have the structure 
\begin{align}\label{eq:ansatz}
	u^k(\bx,t) = k^{\frac d4-1}a(\bx,t)e^{ik\phi(\bx,t)}, \quad k\in\NN^\star = \NN\setminus\{0\},
\end{align}
with an \textbf{amplitude function} $a$ given by 
\begin{align*}
	a(\bx,t):= e^{-|\bx-\bx(t)|^2}
\end{align*}
and a \textbf{phase function} $\phi$ of the form
\begin{align}\label{eq:phi}
	\phi(\bx,t) = \bxi_0\cdot(\bx-\bx(t)) + \frac 12(\bx-\bx(t))\cdot\Big[M_0(\bx-\bx(t))\Big],
\end{align}
where $M_0\in\CC^{d\times d}$ is a $d\times d$ complex symmetric matrix \textbf{with strictly positive imaginary part} to be determined. Let us stress that taking $M_0$ with strictly positive imaginary part is fundamental for the construction of GB. In fact, if we replace \eqref{eq:phi} in \eqref{eq:ansatz}, we can easily see that
\begin{align*}
	u^k(\bx,t) = k^{\frac d4-1}a(\bx,t)e^{-\frac k2 (\bx-\bx(t))\cdot\big[\Im(M_0)(\bx-\bx(t))\big]} e^{ik\bxi_0\cdot(\bx-\bx(t))}e^{\frac{ik}{2}(\bx-\bx(t))\cdot\big[\Re(M_0)(\bx-\bx(t))\big]},
\end{align*}
so that 
\begin{align*}
	|u^k(\bx,t)|^2 = k^{\frac d2-2}|a(\bx,t)|^2 e^{-k(\bx-\bx(t))\cdot\big[\Im(M_0)(\bx-x(t))\big]}
\end{align*}
and $\Im(M_0) > 0$ implies that $u^k$ is essentially a Gaussian profile translated along $\bx(t)$.

The main result that we recall in this section is nowadays classical (see \cite{macia2002lack,ralston1982gaussian}), and establishes the existence of functions of the form \eqref{eq:ansatz}-\eqref{eq:phi} that are approximate solutions of \eqref{eq:mainEqR}. 

\begin{theorem}\label{thm:ApproxSol}
Let $M_0\in\CC^{d\times d}$ with $\Im(M_0)>0$, $\BF{0}\neq\bxi_0\in\RR^d$ and $0<c\in\RR$. Let $\square_c:=\partial_t^2 - c\Delta$ denote the standard D'Alambert operator, and let $\bx(t)$ be a ray for $\square_c$ given by \eqref{eq:ray}. Given any $0<k\in\RR$, define 
\begin{align}\label{eq:ansatzConstThm}
	u^k(\bx,t):= k^{\frac d4-1}a(\bx,t)e^{ik\phi(\bx,t)},
\end{align}
with
\begin{align}\label{eq:aThm}
	a(\bx,t):= e^{-|\bx-\bx(t)|^2}
\end{align}
and 
\begin{align}\label{eq:phiThm}
	\phi(\bx,t):= \bxi_0(\bx-\bx(t)) + \frac 12(\bx-\bx(t))\cdot\Big[M_0(\bx-\bx(t))\Big].
\end{align}
Then, the following facts hold:
\begin{itemize}
	\item[1.] $u^k$ is an approximate solution of the wave equation \eqref{eq:mainEqR}:
	\begin{align}\label{eq:approxSol}
		\sup_{t\in(0,T)}\norm{\square_c u^k(\cdot,t)}{L^2(\RR^d)}\leq \C k^{-\frac 12}
	\end{align}
	for some constant $\C = \C(a,\phi)>0$ not depending on $k$.
	\item[2.] The energy of $u^k$ is of the order of a positive constant when $k\to +\infty$: more precisely, for $t\in(0,T)$ we have
	\begin{align}\label{eq:approxEnergy}
		\lim_{k\to +\infty} E_c(u^k(\cdot,t)) = \C\Big(d,\bxi_0,M_0\Big)\left(\frac{\pi}{\text{det}\big(\Im (M_0)\big)}\right)^{\frac d2}.
	\end{align}
	\item[3.] The energy of $u^k$ is exponentially small off $\bx(t)$ as $k\to +\infty$:
	\begin{align}\label{eq:approxRay}
		\sup_{t\in(0,T)}\int_{\RR^d\setminus B_k(t)} \Big(|u^k_t(\cdot,t)|^2 + c|\nabla u^k(\cdot,t)|^2\Big)\,d\bx \leq \C(a,\phi,d,M_0) e^{-\frac 12\text{det}\big(\Im (M_0)\big)k^{\frac 12}}.
	\end{align}
	Here $B_k(t)$ denotes the $d$-dimensional ball centered at $\bx(t)$ of radius $k^{-1/4}$ and $\C(a,\phi,d,M_0)>0$ is a positive constant not depending on $k$. 
\end{itemize}
\end{theorem}

\begin{remark}[High-order Gaussian Beams]\label{rem:GBhigh}
As shown e.g. in \cite{weiss1985reflection}, it is possible to find correcting terms $\tilde{\phi}, a_1,a_2,\ldots,a_N$ and a cut-off function $\chi\in C_0^\infty(\RR^d\times\RR)$ identically equal to one in a neighborhood of the ray $\bx(t)$ such that the function
\begin{align}\label{eq:GBhigh}
	\tilde u^k_N(\bx,t) = k^{\frac d4-1}\chi(\bx,t)\left(a(\bx,t)+\sum_{j=1}^N k^{-j}a_j(\bx,t)\right)e^{ik\left(\phi(\bx,t)+\tilde\phi(\bx,t)\right)}
\end{align}
still satisfies the conclusions of Theorem \ref{thm:ApproxSol} and moreover
\begin{align}\label{eq:approxSol_N}
	\sup_{t\in(0,T)}\norm{\square_c u^k_N(\cdot,t)}{L^2(\RR^d)}\leq \C k^{-\frac 12-N}.
\end{align}
We stress that the introduction of a cut-off function in \eqref{eq:GBhigh} is necessary to avoid spurious growth away from the center ray. Besides, we see from \eqref{eq:approxSol_N} that the approximation rate of $u^k_N$ is improved by a factor $k^{-N}$. Actually, with this procedure, we can build quasi-solutions of \eqref{eq:mainEqR} approximating the real solution up to an arbitrary order. These quasi-solutions of \eqref{eq:mainEqR} are usually called $N$\textbf{-th order Gaussian Beams}. Assuming this terminology, the ansatz \eqref{eq:ansatz} will then define a $0$-th order Gaussian Beam.
\end{remark}

\begin{remark}[Variable-coefficients wave equation]\label{rem:GBvaraible}
When considering a variable-coefficients wave equation, i.e. when taking $c = c(x) \in C^\infty(\RR)$, the GB construction of Theorem \ref{thm:ApproxSol} still applies. Nevertheless, some small changes need to be introduced in the phase function. In particular, $\phi$ has to be chosen in the form 
\begin{align*}	
	\phi(\bx,t):= \bxi_0(\bx-\bx(t)) + \frac 12(\bx-\bx(t))\cdot\Big[M(t)(\bx-\bx(t))\Big],
\end{align*}
with $M(t)\in\CC^{d\times d}$ solution of the nonlinear ODE
\begin{align}\label{eq:Riccati}
	\begin{cases}
		\dot M(t) = M(t)C(t)M(t) + B(t)M(t) + M(t)B^\top(t) + A(t), & t\in (0,T)
		\\
		M(0) = M_0
	\end{cases}
\end{align}
and where $A(t)$, $B(t)$ and $C(t)$ are $d\times d$ matrices whose coefficients depend on the first and second derivatives of the principal symbol $\mathcal P$ evaluated along the characteristics. This is a Riccati equation and it can be shown (\cite{babich1999higher,ralston1982gaussian}) that, given a symmetric matrix $M_0\in\CC^{d\times d}$ with $\Im(M_0) > 0$, there exist a global solution $M(t)$ of \eqref{eq:Riccati} that satisfies $M(0) = M_0$, $M(t) = M(t)^\top$ and $\Im(M(t)) > 0$ for all $t$.
\end{remark}

We are postponing the proof of Theorem \ref{thm:ApproxSol} to Appendix \ref{appendixCont}. Here we shall just highlight the main ingredients for the explicit construction of the ansatz \eqref{eq:ansatzConstThm}-\eqref{eq:aThm}-\eqref{eq:phiThm}, that shall be later adapted to the finite difference setting. 

\subsection{Asymptotic expansion and explicit construction of the GB ansatz}
We start by substituting the function $u^k$ into \eqref{eq:mainEqR} and, after having gathered the terms with equal power of $k$, we get
\begin{align}\label{eq:ansatz2}
	\square_c u^k =&\, k^{\frac d4-1}e^{ik\phi}\square_c a \notag 
	\\
	&+ k^{\frac d4}e^{ik\phi} i\Big(a\square_c\phi + 2 a_t\phi_t - 2c \nabla a\cdot\nabla\phi\Big) 
	\\
	&+ k^{\frac d4+1}e^{ik\phi} \Big(c\,|\nabla\phi|^2-\phi_t^2\Big)a. \notag 
\end{align}
Let us write the expression \eqref{eq:ansatz2} as 
\begin{align}\label{eq:ansatz3}
	\square_c u^k = k^{\frac d4-1}e^{ik\phi}r_0 + ik^{\frac d4}e^{ik\phi} r_1 + k^{\frac d4+1}e^{ik\phi} r_2, 
\end{align}
where we have denoted 
\begin{subequations}
	\begin{align}
		&r_0:= \square_c a \label{eq:r0}
		\\
		&r_1:= a\square_c\phi + 2 a_t\phi_t - 2c \nabla a\cdot\nabla\phi \label{eq:r1}
		\\
		&r_2:= \Big(c\,|\nabla\phi|^2-\phi_t^2\Big)a \label{eq:r2}
	\end{align}
\end{subequations}

We are going to construct $a$ and $\phi$ in such a way that the terms of higher order in $k$, namely $r_1$ and $r_2$, vanish on $\bx(t)$ up to order $0$ and $2$, respectively.

\subsubsection{Analysis of the $r_2$ term: computation of the phase $\phi$.} We want to construct $\phi$ such that 
\begin{align}\label{eq:r2zero}
	D_\bx^\alpha r_2(\bx(t),t) = 0 \text{ for all } t\in\RR \text{ and } \alpha\in\NN^d\text{ with }|\alpha|\in\{0,1,2\}. 
\end{align}
with $r_2$ given by \eqref{eq:r2}. For this, it is enough to solve the eikonal equation
\begin{align}\label{eq:eikonal}
	c|\nabla\phi|^2 - \phi_t^2 = 0
\end{align}
up to order $2$ on $(\bx(t),t)$. Next we show that this can be done if $\phi$ is of the form \eqref{eq:phiThm}. To this end, let us first notice that, from the definition \eqref{eq:phiThm} of the phase $\phi$ we get 
\begin{align*}
	\nabla\phi(\bx(t),t) &= \Big[\bxi_0 + M_0(\bx-\bx(t))\Big]\bigg|_{\bx = \bx(t)} = \bxi_0 
	\\
	\phi_t(\bx(t),t) &= \bigg[- \bxi_0\cdot\dot \bx(t) - \dot \bx(t)\cdot\Big[M_0(\bx-\bx(t))\Big]\bigg]\Bigg|_{\bx = \bx(t)} = -\bxi_0\cdot\dot \bx(t). 
\end{align*}
Plugging this into \eqref{eq:eikonal}, and using \eqref{eq:ray_prel} and \eqref{eq:tau0cond}, we then obtain that
\begin{align*}
	c|\nabla\phi(\bx(t),t)|^2 - \phi_t^2(\bx(t),t) = c|\bxi_0|^2 - (\bxi_0\cdot\dot \bx(t))^2 = c|\bxi_0|^2 -\frac{c^2|\bxi_0|^4}{\tau_0^2} = c|\bxi_0|^2 -\frac{c^2|\bxi_0|^4}{c|\bxi_0|^2} = 0.
\end{align*}
Secondly, we have from \eqref{eq:eikonal} that
\begin{align}\label{eq:Rder}
	\nabla\Big(c|\nabla\phi|^2 - \phi_t^2\Big) = 2\Big(cH(\phi)\nabla\phi - \phi_t\nabla\phi_t\Big),
\end{align}
where $H(\phi)$ denotes the Hessian matrix of $\phi$. Moreover, 
\begin{align*}
	H(\phi(\bx(t),t)) = M_0 \quad\text{ and }\quad \nabla\phi_t(\bx(t),t) = -M_0\dot \bx(t). 
\end{align*}
Hence, we get from \eqref{eq:Rder}, \eqref{eq:ray_prel} and \eqref{eq:tau0cond} that
\begin{align*}
	\nabla\Big(c|\nabla\phi(\bx(t),t)|^2 - \phi_t^2(\bx(t),t)\Big) = 2cM_0\Big(\bxi_0 - \big(\bxi_0\cdot\dot\bx(t)\big)\dot\bx(t)\Big) = 2cM_0\bxi_0\left(1-\frac{c^2|\bxi_0|^2}{\tau_0^2}\right) = 0.
\end{align*}
Finally, taking into account that $D_\bx^3\phi = 0$, we can compute
\begin{align}\label{eq:Rder2}
	H\Big(c|\nabla\phi|^2 - \phi_t^2\Big) = 2\Big(cH(\phi)^2 - \phi_tH(\phi_t)-\nabla\phi_t\otimes\nabla\phi_t\Big),
\end{align}
with 
\begin{align*}
	\nabla\phi_t\otimes\nabla\phi_t := 
		\begin{pmatrix} 
		\phi_{x_1,t}^2 & \phi_{x_1,t}\phi_{x_2,t} & \ldots & \phi_{x_1,t}\phi_{x_d,t} 
		\\[5pt]
		\phi_{x_1,t}\phi_{x_2,t} & \phi_{x_2,t}^2 & \ldots & \phi_{x_2,t}\phi_{x_d,t} 
		\\[5pt]
		\vdots & \vdots & & \vdots 
		\\[5pt]
		\phi_{x_d,t}\phi_{x_2,t} & \phi_{x_d,t}\phi_{x_2,t} & \ldots & \phi_{x_d,t}^2
		\end{pmatrix}\in\RR^{d\times d}.
\end{align*}
Hence, since $H(\phi_t)(\bx(t),t) = 0$, we obtain from \eqref{eq:Rder2} that
\begin{align*}
	H\Big(c|\nabla\phi(\bx(t),t)|^2 - \phi_t^2(\bx(t),t)\Big) =  2M_0^2\Big(c - |\dot\bx(t)|^2\Big) = 0.
\end{align*}
Therefore, with our choice \eqref{eq:phiThm} of the phase function $\phi$ \eqref{eq:r2zero} is satisfied.  

\subsubsection{Analysis of the $r_1$ term: computation of the amplitude $a$.} To complete the construction of our ansatz, we now have to determine a suitable amplitude $a$. To this end, we shall start by computing $a$ on the bi-characteristic rays, which is done by asking that $r_1$ in \eqref{eq:r1} vanishes on $(\bx(t),t)$, that is, 
\begin{align}\label{eq:r1char}
	2c \nabla a(\bx(t),t)\cdot\nabla\phi(\bx(t),t) - 2a_t(\bx(t),t)\phi_t(\bx(t),t) - a(\bx(t),t)\square_c \phi(\bx(t),t) = 0.
\end{align}

On the other hand, we can readily check from the definition \eqref{eq:phiThm} that the D'Alambertian of the phase $\phi$ vanishes on the characteristics, that is,
\begin{align*}
	\square_c\phi(\bx(t),t) = 0.
\end{align*}
In view of this, \eqref{eq:r1char} simply becomes
\begin{align}\label{eq:r1char1}
	2c \nabla a(\bx(t),t)\cdot\nabla\phi(\bx(t),t) - 2a_t(\bx(t),t)\phi_t(\bx(t),t) = 0.
\end{align}

Substituting $\nabla\phi$ and $\phi_t$, and evaluating on the ray $(\bx(t),t)$ using the fact that $c = |\dot \bx(t)|^2=\bx(t)\cdot\bx(t)$, we obtain that  
\begin{equation}\label{eq:r1char2}
	\begin{array}{ll}
		\displaystyle 2c \nabla a(\bx(t),t)\cdot\nabla\phi(\bx(t),t) - 2a_t(\bx(t),t)\phi_t(\bx(t),t) &= \displaystyle 2\Big(c\bxi_0\cdot\nabla a(\bx(t),t) + \big(\bxi_0\cdot\dot\bx(t)\big) a_t(\bx(t),t)\Big)
		\\[7pt]
		&= \displaystyle 2\bxi_0\cdot\Big(c\nabla a(\bx(t),t) + \dot \bx(t) a_t(\bx(t),t)\Big)
		\\[7pt]
		&= \displaystyle 2\big(\bxi_0\cdot\dot \bx(t)\big)\Big(\dot \bx(t)\cdot\nabla a(\bx(t),t) + a_t(\bx(t),t)\Big)
		\\[7pt]
		&= \displaystyle -2\tau_0\frac{d}{dt}a(\bx(t),t),
	\end{array} 
\end{equation}
where we have used \eqref{eq:tau0cond} and \eqref{eq:ray} to compute
\begin{align*}
	2\bxi_0\cdot\dot \bx(t) = -\frac{2c|\bxi_0|^2}{\tau_0} = -\frac{2\tau_0^2}{\tau_0} = -2\tau_0.
\end{align*}
Hence, we obtain from \eqref{eq:r1char1} and \eqref{eq:r1char2} that $a(\bx(t),t)$ is determined by solving the equation
\begin{align}\label{eq:a}
	\frac{d}{dt} a(\bx(t),t) = 0.
\end{align}
i.e. 
\begin{align*}
	a(\bx(t),t) = a(\bx_0,0) \quad\text{ for all } t\in [0,T].
\end{align*}
In what follows, for simplicity, we will take $a(\bx_0,0)=1$, so that 
\begin{align}\label{eq:aConstRay}
	a(\bx(t),t) = 1 \quad\text{ for all } t\in [0,T].
\end{align}

Notice that we have many possible choices of a function $a(\bx,t)$ satisfying both \eqref{eq:a} and \eqref{eq:aConstRay}. Here, we will take 
\begin{align*}
	a(\bx,t) = e^{-|\bx-\bx(t)|^2},
\end{align*}
so that \eqref{eq:ansatzConstThm} is really a Gaussian profile propagating along the characteristic $\bx(t)$.

\section{Finite-difference approximation}\label{sec:GBdiscrete}

In this section, we adapt the continuous construction of GB described in Section \ref{sec:GBcontinuous} to the finite difference wave equation \eqref{eq:mainEqFD}. Our aim is to provide a GB ansatz yielding to approximate solutions of \eqref{eq:mainEqFD} concentrated on the rays $\bx_{fd}^\pm(t)$ in \eqref{eq:rayFD}, generated by the finite difference principal symbol \eqref{eq:PSFD}. As we shall see, two main difficulties raise when attempting this construction:
\begin{itemize}
	\item[1.] The discrete operators that we shall employ depend on the mesh size $h$. Because of that, we will need to limit the range of the high-frequency parameter $k$ in the GB ansatz according to $h$. We will see that the correct scale is $k = h^{-1}$.
	
	\item[2.] As mentioned before, the finite difference equation \eqref{eq:mainEqFD} admits some spurious solution with zero velocity of propagation. This shall be taken duly into account when constructing the ansatz. 
\end{itemize} 

\subsection{Numerical scheme}

Let us start by introducing in more detail the numerical scheme we shall employ. Given a mesh size $h>0$, we consider an uniform grid on the whole $\RR^d$
\begin{align*}
	\mathcal G^h:= \Big\{\bx_\bj := \bj h,\,\bj\in\ZZ^d\,\Big\}.
\end{align*}

Moreover, for a function $f:\RR^d\to\RR$, we denote $f_\bj:=f(\bx_\bj)$ its evaluation on the grid points, and we define the following finite difference operators: 
\begin{subequations}
	\begin{align}
		&\displaystyle\nabla_h^+ f_\bj: = \big(\partial_{h,i}^+\,f_\bj\big)_{i=1}^d \quad\text{ with }\quad \partial_{h,i}^+\,f_\bj:= \frac 1h \Big(f_{\bj+\BF{e}_i}-f_\bj\Big) & & \text{forward difference} \label{eq:FDoperators1d}
		\\[10pt]
		&\displaystyle\nabla_h^- f_\bj: = \big(\partial_{h,i}^-\,f_\bj\big)_{i=1}^d \quad\text{ with }\quad \partial_{h,i}^-\,f_\bj:= \frac 1h \Big(f_\bj-f_{\bj-\BF{e}_i}\Big) & &\text{backward difference} \label{eq:FDoperators2d}
		\\[10pt]
		&\displaystyle\nabla_h f_\bj: = \big(\partial_{h,i}\,f_\bj\big)_{i=1}^d \quad\;\text{ with }\quad \partial_{h,i}\,f_\bj:= \frac{1}{2h} \Big(f_{\bj+\BF{e}_i}-f_{\bj-\BF{e}_i}\Big) & &\text{centered difference} \label{eq:FDoperators3d}
		\\[10pt]
		&\displaystyle\Dhc f_\bj: = \frac{c}{h^2} \sum_{i=1}^d\Big(f_{\bj+\BF{e}_i} - 2f_\bj + f_{\bj-\BF{e}_i}\Big) & &\text{finite difference Laplacian}, \label{eq:FDoperators6d}
	\end{align}
\end{subequations}
with $(\BF{e}_i)_{i=1}^d$ denoting the canonical basis in $\RR^d$. 

With the notations just introduced, we consider the following semi-discrete finite difference wave equation on $\mathcal G^h$
\begin{align}\label{eq:mainEqFD}
	\begin{cases}
		\square_{c,h} u_\bj(t) = 0, & \bj\in\ZZ^d\;\;\;t\in(0,T)
		\\
		u_\bj(0) = u_\bj^0, \;\;\; u_\bj'(0) = u_\bj^1, & \bj\in\ZZ^d
	\end{cases}
\end{align}
where, for simplicity of notation, we have denoted
\begin{align*}
	\square_{c,h}:=\partial_t^2 - \Dhc
\end{align*}
the discrete D'Alambertian operator. Moreover, we define the semi-discrete energy associated with the solutions of \eqref{eq:mainEqFD} as
\begin{align}\label{eq:energyFD}
	\mathcal E_h[u](t):=\frac 12 \left(\norm{\partial_t u(t)}{\ell^2(h\ZZ^d)}^2 + c\norm{u(t)}{\acc{h}^1(h\ZZ^d)}^2\right),
\end{align}
where $\ell^2(h\ZZ^d)$ and $\acc{h}^1(h\ZZ^d)$ are discrete Lebesgue and Sobolev spaces on the mesh $\mathcal G^h$ defined as
\begin{align}\label{eq:discreteSob}
	&\displaystyle \ell^2(h\ZZ^d) := \left\{u\;\text{ s.t. }\norm{u}{\ell^2(h\ZZ^d)}:= \left(h^d\sum_{\bj\in\ZZ^d} |u_\bj|^2\right)^{\frac 12} <+\infty \right\}
	\\
	&\displaystyle \acc{h}^1(h\ZZ^d) := \left\{u\;\text{ s.t. }\norm{u}{\acc{h}^1(h\ZZ^d)}:= \left(\sum_{i=1}^d\norm{\partial_{h,i}^+\, u_\bj}{\ell^2(h\ZZ^d)}^2\right)^{\frac 12} = \left(h^d\sum_{i=1}^d\sum_{\bj\in\ZZ^d}|\partial_{h,i}^+\, u_\bj|^2\right)^{\frac 12} <+\infty \right\}. \notag
\end{align}

As illustrated numerically in \cite{biccari2020propagation}, the semi-discrete wave equation \eqref{eq:mainEqFD} admits highly concentrated and oscillating solutions that propagate along the characteristics $\bx_{fd}(t)$ given by \eqref{eq:rayFD}. The aim of this section is to justify these numerical observation through the definition of a GB ansatz.

In what follows, for the sake of simplicity, we will first consider the one-dimensional case $d=1$, in which we will give complete detail of the GB construction. In a second moment, we will comment about the extension of this construction to the general multi-dimensional case. 

\subsection{One-dimensional semi-discrete GB ansatz} We start by introducing the one-dimensional version of the finite difference operators we defined in \eqref{eq:FDoperators1d}, \eqref{eq:FDoperators2d}, \eqref{eq:FDoperators3d} and \eqref{eq:FDoperators6d}:
\begin{subequations}
	\begin{align}
		&\displaystyle \fwd f_j:= \frac 1h \Big(f_{j+1}-f_j\Big) \quad\text{ for all }j\in\ZZ & & \text{forward difference} \label{eq:FDoperators1}
		\\[10pt]
		&\displaystyle \bwd f_j:= \frac 1h \Big(f_j-f_{j-1}\Big) \quad\text{ for all }j\in\ZZ & &\text{backward difference} \label{eq:FDoperators2}
		\\[10pt]
		&\displaystyle \Dh f_j:= \frac{1}{2h} \Big(f_{j+1}-f_{j-1}\Big) \quad\text{ for all }j\in\ZZ & &\text{centered difference} \label{eq:FDoperators3}
		\\[10pt]
		&\displaystyle\Dhc f_j: = \frac{c}{h^2} \Big(f_{j+1} - 2f_j + f_{j-1}\Big) \quad\text{ for all }j\in\ZZ & &\text{finite difference Laplacian}, \label{eq:FDoperators6}
	\end{align}
\end{subequations}

These operators fulfill some useful properties, that can be easily shown through the definitions: for all $j\in\ZZ$, we have
\begin{subequations}\label{eq:FDcomposition}
	\begin{align}
		&\big(\fwd + \bwd\big)f_j = 2\Dh f_j \label{eq:FDcomposition4}
		\\[8pt]
		&\big(\fwd - \bwd\big)f_j = \frac hc\Dhc f_j \label{eq:FDcomposition5}
		\\[7pt]
		&\Dhc (fg)_j = f_j\Dhc g_j + g_j\Dhc f_j + c\Big(\fwd f_j\fwd g_j + \bwd f_j\bwd g_j\Big) \label{eq:FDproduct4}
	\end{align}
\end{subequations}

Finally, let us state the main result of the present paper, whose proof will be provided in the next section and whose validation will be given 
in Section \ref{sec:numerics} through sharp numerical simulations.

\begin{theorem}\label{thm:ApproxSolFD}
Let $M_0\in\CC$ with $\Re(M_0) = 0$ and $\Im(M_0)>0$, $0\neq\xi_0\in\RR$, $0<c\in\RR$ and $x_{fd}(t)$ be a ray for $\square_{c,h}$ given by \eqref{eq:rayFD}. Given $h\in(0,1)$, define 
\begin{align}\label{eq:ansatzFD_Thm}
	u^h_{fd}(x,t):= h^{\frac 34} A_j(x,t)e^{\frac ih\Phi_j(x,t)},
\end{align}
with
\begin{align}\label{eq:AThm}
	A(x,t):= e^{-(x-x_{fd}(t))^2} e^{\mp\frac 12 \ln\left(1-\frac{M_0\sqrt{c}}{2}\sin\left(\frac{\xi_0}{2}\right)t\right)}
\end{align}
and 
\begin{align}\label{eq:PhiThm}
	\Phi(x,t):= \pm\sqrt{c}\left(\xi_0\cos\left(\frac{\xi_0}{2}\right) - 2\sin\left(\frac{\xi_0}{2}\right)\right)t + \xi_0(x-x_{fd}(t)) + \frac{M_0}{2\mp M_0\sqrt{c}\sin\left(\frac{\xi_0}{2}\right)t}(x-x_{fd}(t))^2.
\end{align}
Then, the following facts hold:
\begin{itemize}
	\item[1.] The $u^h_{fd}$ are approximate solutions of the finite difference wave equation \eqref{eq:mainEqFD}:
	\begin{align}\label{eq:approxSolFD}
		\mathcal S_h[u_{fd}^h]:= \sup_{t\in(0,T)}\norm{\square_{c,h}u^h_{fd}(\cdot,t)}{\ell^2(h\ZZ)} = \mathcal O(h^{\frac 12}), \quad\text{ as } h\to 0^+.		
	\end{align}
	\item[2.] The energy of $u^h_{fd}$ satisfies
	\begin{align}\label{eq:approxEnergyFD}
		\mathcal E_h[u_{fd}^h](t) = \mathcal O(1), \quad\text{ as } h\to 0^+.
	\end{align}
	\item[3.] The energy of $u^h_{fd}$ is exponentially small off $x_{fd}(t)$ as $h\to 0^+$:
	\begin{align}\label{eq:approxRayFD}
		\sup_{t\in(0,T)} \frac h2\sum_{j\in\ZZ^\dagger(t)} \Big(|\partial_t u_{fd,j}^h|^2 + c|\fwd u_{fd,j}^h|^2\Big)\leq \C_1(A,\Phi) \Big(1 + h + h^2\Big) e^{-\C_2(M_0)h^{-\frac 12}}, 
	\end{align}
	where $\C_1(A,\Phi)>0$ and $\C_2(M_0)>0$ are two positive constants independent of $h$ and we have denoted 
	\begin{align*}
		\ZZ^\dagger(t) := \Big\{ j\in\ZZ\,:\, |x_j-x_{fd,j}(t)|>h^{\frac 14}\Big\}
	\end{align*}
\end{itemize}
\end{theorem}

We mention that, in Theorem \ref{thm:ApproxSolFD}, $x\in\RR$ has to be considered as a \textit{dummy} variable, that we inherited from the continuous construction in Section \ref{sec:GBcontinuous} and we have kept in order to slightly simplify our notation in the forthcoming computations.
 
\begin{remark}
We anticipate that the mesh-size parameter $h$ in the finite difference ansatz \eqref{eq:ansatzFD_Thm} will be related with the high-frequency parameter $k$ in the continuous ansatz \eqref{eq:ansatzConstThm} through the choice $h = k^{-1}$. In this way, our GB construction in Theorem \ref{thm:ApproxSolFD} is consistent with the one of Theorem \ref{thm:ApproxSol} in space dimension $d=1$ for what concerns the approximation rate of the obtained quasi-solutions. 
\end{remark}

\begin{remark}
We highlight that there are some differences between the continuous ansatz provided in Theorem \ref{thm:ApproxSol} and the semi-discrete one of Theorem \ref{thm:ApproxSolFD}. In particular, $u_{fd}^h$ defined in \eqref{eq:ansatzFD_Thm} is not simply the projection on the mesh $\mathcal G_h$ of its continuous counterpart \eqref{eq:ansatzConstThm}. Instead, some corrector terms have been introduced both in the amplitude $A$ and in the phase $\Phi$. As we shall see with more detail in Section \ref{sec:FDansatz}, the introduction of these corrector terms is required to compensate the non-uniform velocity of propagation of the finite difference characteristics $x_{fd}(t)$ in \eqref{eq:rayFD}, which makes $a$ and $\phi$ in \eqref{eq:ansatzConstThm} not suitable choices for our semi-discrete GB ansatz. 
\end{remark}

As we did for Theorem \ref{thm:ApproxSol} before, we are postponing the proof of Theorem \ref{thm:ApproxSolFD} to Appendix \ref{appendixFD}. Here we shall just highlight the main ingredients for the explicit construction of the ansatz \eqref{eq:ansatzFD_Thm}-\eqref{eq:AThm}-\eqref{eq:PhiThm}. 

\subsection{Asymptotic expansion and explicit construction of the GB ansatz}\label{sec:FDansatz} In our forthcoming computations, we shall use the following well-known identities: for all $\alpha,\beta\in\RR$
\begin{subequations}\label{eq:identitiesC}
	\begin{align}
		&e^{i\alpha} - e^{-i\beta} = 2i\sin\left(\frac{\alpha+\beta}{2}\right)e^{\frac i2(\alpha-\beta)} \label{eq:identitiesC2}
		\\[10pt]
		&2 - e^{i\alpha} - e^{-i\beta} = 4\sin^2\left(\frac{\alpha+\beta}{4}\right) - 4i\cos\left(\frac{\alpha+\beta}{2}\right)\sin\left(\frac{\alpha-\beta}{4}\right)e^{\frac i4(\alpha-\beta)} \label{eq:identitiesC3} 
	\end{align}
\end{subequations}

Moreover, similarly to the continuous case of Section \ref{sec:GBcontinuous}, we shall consider the following ansatz for approximated solutions of \eqref{eq:mainEqFD}
\begin{align}\label{eq:ansatzFD}
	u^k_{fd}(x,t) = k^{-\frac 34}A(x,t) e^{ik\Phi(x,t)},  
\end{align}
with suitable phase $\Phi$ and amplitude $A$. 

To properly identify these phase and amplitude functions, the starting point is once again to compute $\square_{c,h} u^k_{fd,j}$ and gather the terms with equal power of the high-frequency parameter $k$. First of all, we have
\begin{align}\label{eq:FDtimeDer}
	\partial_t^2 u^k_{fd,j} = e^{ik\Phi_j}\left[k^{-\frac 34}\partial_t^2A_j + ik^{\frac 14}\Big(2\partial_t A_j\partial_t \Phi_j + A_j\partial_t^2\Phi_j\Big)-k^{\frac 54}A_j(\partial_t\Phi_j)^2\right].
\end{align}
Secondly, using \eqref{eq:FDproduct4} we can compute
\begin{align}\label{eq:FDlaplacianPrel}
	\Dhc u^k_{fd,j} = k^{-\frac 34}\bigg[A_j\Dhc e^{ik\Phi_j} + e^{ik\Phi_j}\Dhc A_j + c\Big(\fwd A_j\fwd e^{ik\Phi_j} + \bwd A_j\bwd e^{ik\Phi_j}\Big)\bigg]. 
\end{align}
Now, using \eqref{eq:FDoperators1}, \eqref{eq:FDoperators2}, \eqref{eq:FDcomposition4}, \eqref{eq:FDcomposition5} and \eqref{eq:identitiesC2}, we can show that
\begin{align*}
	\fwd A_j\fwd e^{ik\Phi_j} + \bwd A_j\bwd e^{ik\Phi_j} = 2ike^{ik\Phi_j}\Dh A_j\frac{\sin(hk\Dh\Phi_j)}{hk}e^{i\frac{h^2 k}{2c}\Dhc\Phi_j}.
\end{align*}
We then get from \eqref{eq:FDlaplacianPrel} that
\begin{align}\label{eq:FDlaplacianPrel2}
	\Dhc u^k_{fd,j} = k^{-\frac 34}\Big(A_j\Dhc e^{ik\Phi_j} + e^{ik\Phi_j}\Dhc A_j\Big) + k^{\frac 14}e^{ik\Phi_j} 2ic \Dh A_j \frac{\sin(hk\Dh\Phi_j)}{hk}e^{i\frac{h^2 k}{2c}\Dhc\Phi_j}. 
\end{align}
Moreover, by means of \eqref{eq:FDoperators6}, \eqref{eq:FDcomposition4}, \eqref{eq:FDcomposition5} and \eqref{eq:identitiesC3}, we get
\begin{align*}
	\Dhc e^{ik\Phi_j} =&\; \frac{c}{h^2}\bigg(e^{ik\Phi_{j+1}} - 2e^{ik\Phi_j} + e^{ik\Phi_{j-1}}\bigg) = -\frac{c}{h^2}e^{ik\Phi_j}\bigg(2-e^{ihk\fwd\Phi_j} - e^{-ihk\bwd\Phi_j}\bigg)
	\\
	=&\,e^{ik\Phi_j}\left(-k^2\frac{4c\sin^2\left(\frac{hk}{2}\Dh\Phi_j\right)}{(hk)^2} + ik \cos\left(hk\Dh\Phi_j\right)\frac{4c\sin\left(\frac{h^2k}{4c}\Dhc\Phi_j\right)}{h^2k}e^{i\frac{h^2k}{4c}\Dhc\Phi_j}\right). 
\end{align*}
Hence, we obtain from \eqref{eq:FDlaplacianPrel2} that
\begin{equation}\label{eq:FDlaplacian}
	\begin{array}{lll}
		\Dhc u^k_{fd,j} =
		\\[12pt]
		\quad k^{-\frac 34} e^{ik\Phi_j}\Dhc A_j
		\\
		\quad + ik^{\frac 14}e^{ik\Phi_j} \displaystyle \Bigg(2c\Dh A_j\frac{\sin\left(hk\Dh\Phi_j\right)}{hk}e^{\frac{ih^2k}{2c}\Dhc\Phi_j} + A_j\cos\left(hk\Dh\Phi_j\right)\frac{4c\sin\left(\frac{h^2k}{4c}\Dhc\Phi_j\right)}{h^2k}e^{i\frac{h^2k}{4c}\Dhc\Phi_j}\Bigg) 
		\\[15pt]
		\quad\displaystyle -k^{\frac 54}e^{ik\Phi_j} A_j\frac{4c\sin^2\left(\frac{hk}{2}\Dh\Phi_j\right)}{(hk)^2}. 
	\end{array} 
\end{equation}
Therefore, joining \eqref{eq:FDtimeDer} and \eqref{eq:FDlaplacian} we finally get
\begin{align}\label{eq:FDdAlambertian}
	\square_{c,h} u^k_{fd,j} = k^{-\frac 34}e^{ik\Phi_j}\mathcal R_0 + ik^{\frac 14}e^{ik\Phi_j}\mathcal R_1 + k^{\frac 54}e^{ik\Phi_j}A_j\mathcal R_2,
\end{align}
with
\begin{subequations}
	\begin{align}
		\mathcal R_0:=&\, \square_{c,h} A_j  \label{eq:R0}
		\\[7pt]
		\mathcal R_1:= &\, 2\partial_t A_j\partial_t\Phi - 2c\Dh A_j \frac{\sin\left(hk\Dh\Phi_j\right)}{hk}e^{\frac{ih^2k}{2c}\Dhc\Phi_j}  \label{eq:R1}
		\\
		&+ A_j\left(\partial_t^2\Phi_j - \cos\left(hk\Dh\Phi_j\right)\frac{4c\sin\left(\frac{h^2k}{4c}\Dhc\Phi_j\right)}{h^2k}e^{i\frac{h^2k}{4c}\Dhc\Phi_j}\right) \notag
		\\[7pt]
		\mathcal R_2:=&\, \frac{4c\sin^2\left(\frac{hk}{2}\Dh\Phi_j\right)}{(hk)^2} - (\partial_t\Phi_j)^2 \label{eq:R2}
	\end{align}
\end{subequations}

Starting from \eqref{eq:FDdAlambertian}, we shall now determine the phase $\Phi$ and the amplitude $A$ of the ansatz by annulling the terms $\mathcal R_1$ and $\mathcal R_2$ on the semi-discrete characteristics. 

To do so, it will be fundamental a correct selection of the parameter $k$, which shall be taken as a power of the step-size $h$: 
\begin{align*}
	k=h^q, \quad q\in\RR. 
\end{align*}
Nevertheless, when doing this, we have to choose carefully the exponent $q\in\RR$. In fact:
\begin{itemize}
	\item If $q<-1$, then 
	\begin{align*}
		\lim_{h\to 0^+} hk = \lim_{h\to 0^+} h^{q+1} = +\infty
	\end{align*}
	and we have from \eqref{eq:R2} that 
	\begin{align*}
		\lim_{h\to 0^+} \mathcal R_2 = -(\partial_t\Phi)^2.
	\end{align*}
	We then obtain a degenerate eikonal equation for $\Phi$ in which the space derivative $\partial_x\Phi$ does not appear, which  suggests that this choice of $q$ is not suitable.
	
	\item If $q>-1$, then 
	\begin{align*}
		\lim_{h\to 0^+} hk = \lim_{h\to 0^+} h^{q+1} = 0
	\end{align*}
	and we have from \eqref{eq:R2} that 
	\begin{align*}
		\lim_{h\to 0^+} \mathcal R_2 = c(\partial_x\Phi)^2 - (\partial_t\Phi)^2.
	\end{align*}
	This is the eikonal equation corresponding to the continuous wave equation \eqref{eq:mainEqR}. Nevertheless, the construction of GB for the finite difference wave equation \eqref{eq:mainEqFD} should be based on solving the eikonal equation corresponding to the principal symbol \eqref{eq:PSFD}, that is  
	\begin{align}\label{eq:eikonalFD}
		\mathcal R_{fd}:= 4c\sin^2\left(\frac{\partial_x\Phi}{2}\right) - (\partial_t\Phi)^2 = 0.
	\end{align}
	Then, also this second choice of $q$ is not appropriate for our construction.	
	
	\item If $q=-1$, then 
	\begin{align*}
		\lim_{h\to 0^+} hk = \lim_{h\to 0^+} h^{q+1} = 1
	\end{align*}
	and we have from \eqref{eq:R2} that 
	\begin{align*}
		\lim_{h\to 0^+} \mathcal R_2 = 4c\sin^2\left(\frac{\partial_x\Phi}{2}\right) - (\partial_t\Phi)^2 = \mathcal R_{fd}.
	\end{align*}
	Therefore, $k = h^{-1}$ is the correct choice for the high-frequency parameter.	
\end{itemize}

\begin{remark}
We stress that more general choices of the high-frequency parameter as a function of the mesh size $h$ would also be possible. In fact, we could take any $k=\zeta(h)$ with
\begin{align}\label{eq:kSel}
	\lim_{h\to 0^+} h\zeta(h) = 1.
\end{align}
The choice $k=h^{-1}$, for the sake of simplicity, is the most natural situation in which \eqref{eq:kSel} holds.
\end{remark}

In view of the above discussion, in the sequel, we will consider the following ansatz for approximated solutions of \eqref{eq:mainEqFD} 
\begin{align}\label{eq:ansatzFD_h}
	u^h_{fd}(x,t) = h^{\frac 34}A(x,t) e^{\frac ih\Phi(x,t)}.
\end{align}
Then, from \eqref{eq:FDdAlambertian} we get
\begin{align}\label{eq:FDdAlambertian_h}
	\square_{c,h} u^h_{fd,j} & = e^{ik\Phi_j}\Big[h^{\frac 34}\mathcal R_0 + ih^{-\frac 14}\mathcal R_1 + h^{-\frac 54}A_j\mathcal R_2\Big] \notag  
	\\
	&= e^{\frac ih\Phi_j}\left[h^{\frac 34}\left(\mathcal R_0 + A_j \frac{\mathcal R_2-\mathcal R_{fd}}{h^2}\right) + ih^{-\frac 14}\mathcal R_1 + h^{-\frac 54}A_j\mathcal R_{fd}\right].
\end{align}

\subsubsection{Frequency ranges: discrete versus continuous}
Before continuing further with the technical details about the construction of GB solutions for the discrete wave equation \eqref{eq:mainEqFD}, let us devote some words to a heuristic discussion showing how our asymptotic analysis allows building a bridge to connect the GB theory for the continuous model \eqref{eq:mainEqR} with the FD regime studied in this paper. 

To this end, let us consider the FD symbol \eqref{eq:FDsymbol_full} that, in the one-dimensional case that we are addressing in this section, reads as
\begin{align}\label{eq:FDsymbol_1D}
	\mathcal P_{fd,h}(\xi,\tau) = -\tau^2 + \frac{4c}{h^2}\sin^2\left(\frac{h\xi}{2}\right).
\end{align}
Taking into account that the sinus is an analytic function, we can replace it with its Taylor expansion 
\begin{align}\label{eq:taylor}
	\sin\left(\frac{h\xi}{2}\right) = \sum_{n\geq 0} \beta_n(h\xi)^{2n+1}, \quad \beta_n = \frac{(-1)^n}{2^{2n+1}(2n+1)!}\text{ for all }n\geq 0,
\end{align}
thus obtaining an equivalent FD symbol in the form 
\begin{align*}
	\mathcal P_{fd,h}(\xi,\tau) = -\tau^2 + 4c\sum_{n\geq 0}\gamma_n h^{2n}\xi^{2n+2}, 
\end{align*}
with 
\begin{align*}
	\gamma_n = \sum_{m=0}^n \beta_n\beta_m = \sum_{m=0}^n\frac{(-1)^{m+n}}{2^{2m+2n+2}(2m+1)!(2n+1)!}, \quad\text{ for all }n\geq 0.
\end{align*}
Moreover, observing that $\gamma_0 = 1/4$, we can easily obtain
\begin{align}\label{eq:FDsymbol_full_1D}
	\mathcal P_{fd,h}(\xi,\tau) = -\tau^2 + c\xi^2 + 4c\sum_{n\geq 1}\gamma_n h^{2n}\xi^{2n+2}.
\end{align}

From the above expression, we can immediately see how the symbol $-\tau^2 + c\xi^2$ of the continuous one-dimensional wave equation is obtained simply by truncating the Taylor expansion \eqref{eq:taylor} at the first term $n=0$. But, actually, \eqref{eq:FDsymbol_full_1D} hides more information. 

As a matter of fact, the series in \eqref{eq:FDsymbol_full_1D} produces different types of effects on the symbol $\mathcal P_{fd,h}(\xi,\tau)$, depending on the range of frequencies at which we observe it. 

\medskip 
\paragraph*{\textbf{Case 1:} $|\xi|\sim h^{-1}$} We start by analyzing the frequency regime $|\xi|\sim h^{-1}$ that, we recall, is the one at which we are going to construct our GB solution. Consider the partial sums
\begin{align*}
	s_N:= 4c\sum_{n=1}^N\gamma_n h^{2n}\xi^{2n+2}, \quad N\in\NN^\ast = \NN\setminus\{0\}
\end{align*}
and the associated partial symbol 
\begin{align}\label{eq:FDsymbol_partial}
	\mathcal P_{fd,h,N}(\xi,\tau)\coloneqq -\tau^2 + c\xi^2 + s_N,
\end{align}
and observe that, when $|\xi|\sim h^{-1}$, for all $N\in\NN^\ast$ we can approximate $s_N$ as
\begin{align*}
	s_N = 4c\sum_{n=1}^N\gamma_n h^{2n}\xi^{2n+2} = 4c\xi^2\sum_{n=1}^N\gamma_n h^{2n}\xi^{2n} \sim 4c\xi^2\sum_{n=1}^N\gamma_n. 
\end{align*}
When replacing the above expression into \eqref{eq:FDsymbol_partial}, we then obtain that
\begin{align*}
	\mathcal P_{fd,h,N}(\xi,\tau) \sim -\tau^2 + c_N\xi^2,
\end{align*}
with 
\begin{align*}
	c_N\coloneqq c\left(1 + 4\sum_{n=1}^N\gamma_n\right), \quad\text{ for all }N\in\NN^\ast.
\end{align*}

In other words, the contribution of the partial sums $s_N$ is that of introducing correction terms on the velocity of propagation of the waves, making it deviating from its usual value $c$. This phenomenon is appreciated in Figure \ref{fig:symbol}, where we show the function $\xi^2 + s_N$ for different values of $N$, ranging from $N=0$ (corresponding to $c\xi^2$) up to the finite-difference symbol \eqref{eq:FDsymbol_1D} as $N\to +\infty$. We can see in the plot how the successive approximations \eqref{eq:FDsymbol_partial} fill the gap between the continuous and finite-difference setting, reducing their slope during the process and, therefore, inducing the aforementioned adjustments in the waves' propagation velocity.

\begin{SCfigure}[1][!h]
	\centering
	\includegraphics[width=0.62\textwidth]{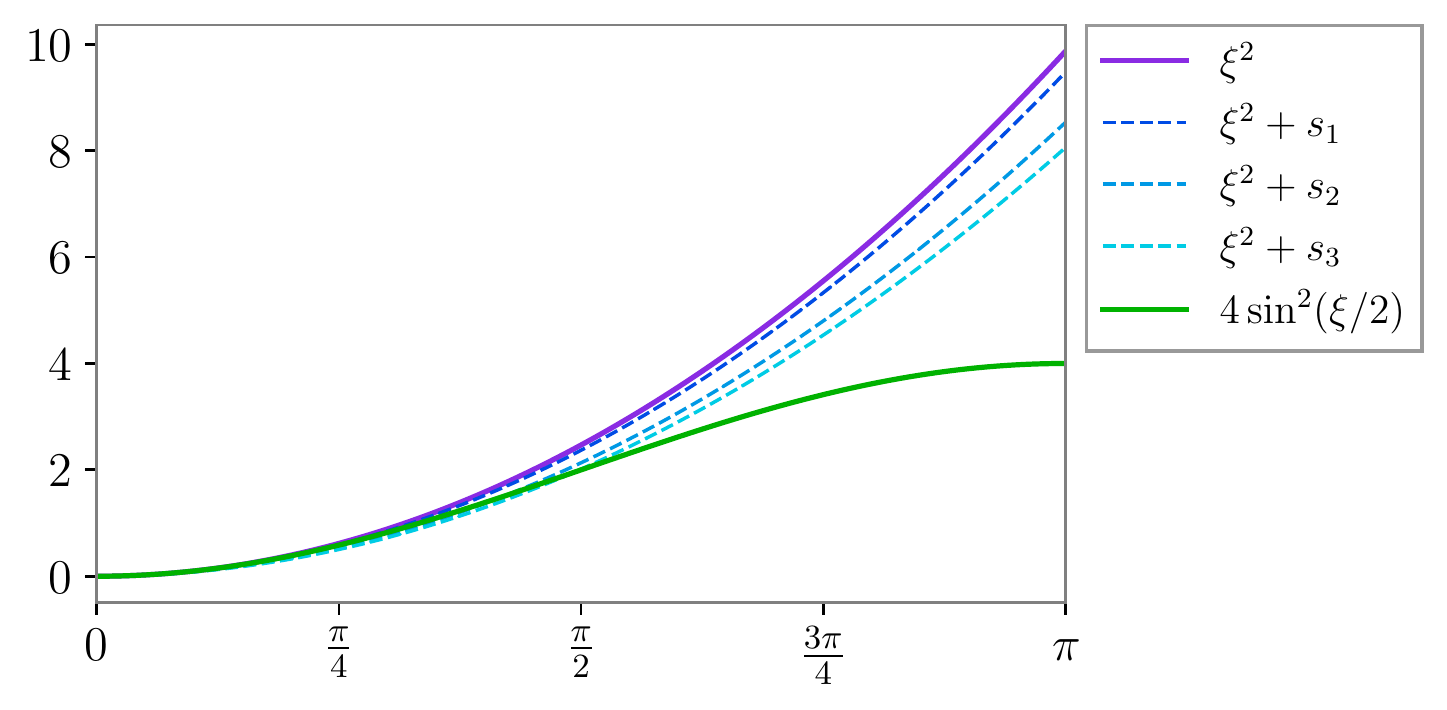}
	\caption{Function $\xi^2 + s_N$ for $\xi\in (0,\pi)$ and different values of $N$, ranging from $N=0$ up $N\to +\infty$.}\label{fig:symbol}
\end{SCfigure}

We stress that, since $\{\gamma_n\}_{n\geq 0}\in\ell^1$, there exists some $\hat c\in\RR$ such that $c_N\to \hat c$ when $N\to +\infty$. This means that, when $|\xi|\sim h^{-1}$, the trigonometric symbol \eqref{eq:FDsymbol_1D} generates high-frequency spurious solutions of the wave equation, traveling at a velocity $\hat c$. As observed in several previous works (see \cite{biccari2020propagation,macia2002lack,marica2015propagation} and the references therein), the presence of these solutions will contaminate all kind of conclusions about the properties of the finite-difference wave equation \eqref{eq:mainEqFD}, with consequences, for instance, on related inverse or control problems. 

\medskip 
\paragraph*{\textbf{Case 2:} $|\xi|\sim h^{-\frac{2n}{2n+2}}$ \textbf{for all} $n\in\NN^\ast$}

A second interesting frequency regime that deserves some further discussion is 
\begin{align*}
	|\xi|\sim h^{-\frac{2n}{2n+2}} \quad\text{ for all } n\in\NN^\ast.
\end{align*}

In particular, we can see that, in this regime, the contribution of the sum into the trigonometric symbol \eqref{eq:FDsymbol_full_1D} changes with respect to the situation of Case 1 above. To this end, let us start by rewriting
\begin{align*}
	\mathcal P_{fd,h}(\xi,\tau) = -\tau^2 + c\xi^2 + 4c\gamma_1 h^2\xi^4 + 4c\sum_{n\geq 2}\gamma_n h^{2n}\xi^{2n+2}.
\end{align*}

We can easily see that, in the regime $|\xi|\sim h^{-\frac 12}$ (that we stress corresponds to $|\xi|\sim h^{-\frac{2n}{2n+2}}$ when $n=1$), the last term of the above expression is of the order of $h$ and, therefore, negligible as $h\to 0^+$. This leads to the following approximation of the symbol $\mathcal P_{fd,h}(\xi,\tau)$:
\begin{align*}
	\mathcal P_{fd,h}(\xi,\tau) \sim -\tau^2 + c\xi^2 + \gamma_1 h^2\xi^4.
\end{align*}
Notice that this is the Fourier symbol associated with the fourth-order PDE  
\begin{align*}
	u_{tt} - c\partial^2_x u - \gamma_1 h^2 \partial^4_x u = 0.		
\end{align*}

In the same fashion, when $|\xi|\sim h^{-\frac 23}$ (corresponding to $|\xi|\sim h^{-\frac{2n}{2n+2}}$ when $n=2$), we can approximate $\mathcal P_{fd,h}(\xi,\tau)$ with 
\begin{align*}
	\mathcal P_{fd,h}(\xi,\tau) \sim -\tau^2 + c\xi^2 + \gamma_1 h^2\xi^4 + \gamma_2h^4\xi^6,	
\end{align*}
with the associated sixth-order PDE  
\begin{align*}
	u_{tt} - c\partial^2_x u - \gamma_1 h^2 \partial^4_x u - \gamma_2 h^4 \partial^6_x u = 0.		
\end{align*}

This kind of reasoning can be carried on for all $N\in\NN^\ast$, until recovering the symbol \eqref{eq:FDsymbol_1D} in the regime $|\xi|\sim h^{-1}$ when $n\to +\infty$. In particular, this heuristic discussion suggests that we can fill the gap between the pure wave equation $u_{tt}-c\partial_x^2 u = 0$ and the discrete one associated with the trigonometric symbol \eqref{eq:FDsymbol_1D} by adding a series of correcting terms of the form
\begin{align}\label{eq:corrections}
	-\gamma_nh^{2n}\partial^{2n+2}_x u, \quad n\in\NN^\ast,
\end{align}
that generate a family of solutions for the discretized wave equation that are observable only in the high-frequency regime $|\xi|\sim  h^{-\frac{2n}{2n+2}}$. 

A possible way to appreciate the impact of these solution on the propagation properties of discrete wave equation would be to develop a general GB analysis starting from the symbol \eqref{eq:FDsymbol_full_1D}. To do that, we may expect that the approach we develop in this paper (which, according to our previous discussion, covers the limit case $N\to +\infty$) is still applicable up to some modification, including and adaptation of the ansatz so to take into account the presence of the correcting terms \eqref{eq:corrections}, and a different selection of the frequency parameter $k(h)$, that we could conjecture to be 
\begin{align*}
	k(h) \sim h^{-\frac{2n}{2n+2}}, \quad\text{ for all } n\in\NN^\ast.
\end{align*}

This would provide us with a series of GB profiles, whose superposition would connect the continuous GB solutions of Theorem \ref{thm:ApproxSol} to the semi-discrete ones of Theorem \ref{thm:ApproxSolFD}.

\subsubsection{Design of the phase $\Phi$}\label{sec:PhaseFD}

As for the continuous case of Section \ref{sec:GBcontinuous}, a suitable phase for our GB construction needs to possess two main features. On the one hand, $\Phi$ should contain a term of the form
\begin{align*}
	M(t)(x-x_{fd}(t))^2
\end{align*}
with $\Im (M(t))>0$ for all $t>0$, to ensure that the ansatz \eqref{eq:ansatzFD_h} is really a Gaussian profile transported along the ray $x_{fd}$. On the other hand, $\Phi$ has to be such that
\begin{align}\label{eq:R2zero}
	\partial_h^\alpha \mathcal R_2(x_{fd}(t),t) = 0 \text{ for all } t\in\RR \text{ and } \alpha\in\{0,1,2\},
\end{align}
where $\partial_h^\alpha$ denotes a discrete derivative of order $\alpha$ on the mesh $\mathcal G^h$. 

Hence, to impose \eqref{eq:R2zero}, one has to compute discrete derivatives of $\Phi$ on $x_{fd}(t)$ which, however, may lead to cumbersome calculations. To avoid them, we replace the semi-discrete eikonal $\mathcal R_2$ in \eqref{eq:R2} with $\mathcal R_{fd}$ given in \eqref{eq:eikonalFD}, that only involves continuous derivatives. We are allowed to do that since the error is small:
\begin{align}\label{eq:RError}
	|\mathcal R_2 - \mathcal R_{fd}| &= 4c\left|\sin^2\left(\frac{\Dh\Phi_j}{2}\right)-\sin^2\left(\frac{\partial_x\Phi}{2}\right)\right|
	\\
	&= 4c\left|\sin\left(\frac{\Dh\Phi_j-\partial_x\Phi}{2}\right)\sin\left(\frac{\Dh\Phi_j+\partial_x\Phi}{2}\right)\right| \leq 2c\left|\Dh\Phi_j-\partial_x\Phi\right| = \mathcal O(h^2). \notag
\end{align}
Hence, in what follows, we will design $\Phi$ such that
\begin{align}\label{eq:Rfdzero}
	\partial_x^\alpha \mathcal R_{fd}(x_{fd}(t),t) = 0 \text{ for all } t\in\RR \text{ and } \alpha\in\{0,1,2\}. 
\end{align}

Taking inspiration from the continuous framework of Section \ref{sec:GBcontinuous}, one could then try considering a phase function $\Phi$ with the same structure as $\phi$ in \eqref{eq:phi}, i.e.
\begin{align}\label{eq:phiHybridFalse}
	\Phi(x,t) = \xi_0(x-x_{fd}(t)) + \frac 12 M(t)(x-x_{fd}(t))^2, \quad\Im (M(t))>0.
\end{align}

Nevertheless, we can easily see that this would not be a good candidate for our construction. In fact, such a function $\Phi$ does not satisfy \eqref{eq:Rfdzero}, not even at order $\alpha = 0$. Indeed, we can readily check from \eqref{eq:rayFD} and \eqref{eq:phiHybridFalse} that
\begin{align*}
	\mathcal R_{fd}(x_{fd}(t),t) &= 4c\sin^2\left(\frac{\partial_x\Phi(x_{fd}(t),t)}{2}\right) - (\partial_t\Phi(x_{fd}(t),t))^2 
	\\
	&= 4c\sin^2\left(\frac{\xi_0}{2}\right) - \xi_0^2\dot x_{fd}(t)^2 = 4c\sin^2\left(\frac{\xi_0}{2}\right) - c\xi_0^2\cos^2\left(\frac{\xi_0}{2}\right).
\end{align*}
Hence, we would have $\mathcal R_{fd}(x_{fd}(t),t) = 0$ only for $\widehat{\xi}_0\in\RR$ satisfying the trigonometric equation
\begin{align}\label{eq:xi0hat}
	\widehat{\xi}_0\cos\left(\frac{\widehat{\xi}_0}{2}\right) = \pm 2\sin\left(\frac{\widehat{\xi}_0}{2}\right).
\end{align}

But for all the values of $\xi_0\in\RR$ such that \eqref{eq:xi0hat} is not fulfilled, we would have $\mathcal R_{fd}(x_{fd}(t),t) \neq 0$. This tells us that $\Phi$ as in \eqref{eq:phiHybridFalse} is not appropriate to generate a suitable GB ansatz for \eqref{eq:mainEqFD}. 

To cope with this fact, taking inspiration from general GB constructions described for instance in \cite{liu2010recovery,liu2014gaussian,liu2013error,liu2016sobolev}, we shall introduce a correction term in the definition of the phase. In particular, we shall take $\Phi$ in the form
\begin{align}\label{eq:phiFD}
	\Phi(x,t) = \omega(t) + \xi_0(x-x_{fd}(t)) + \frac 12 M(t)(x-x_{fd}(t))^2,
\end{align}
with $\omega$ and $M$ to be determined by imposing \eqref{eq:Rfdzero}. To this end, let us first compute
\begin{align}\label{eq:eikonalFDder}
	\partial_x\mathcal R_{fd}(x,t) = 2\Big(c\sin(\partial_x\Phi)\partial_{xx}\Phi - \partial_t\Phi\partial_{tx}\Phi\Big),
\end{align}
and 
\begin{align}\label{eq:eikonalFDder2}
	\partial_{xx} \mathcal R_{fd}(x,t) = 2\Big(c\cos(\partial_x\Phi)(\partial_{xx}\Phi)^2 + c\sin(\partial_x\Phi)\partial_{xxx}\Phi - (\partial_{tx}\Phi)^2 - \partial_t\Phi\partial_{txx}\Phi\Big).
\end{align}
Moreover, from \eqref{eq:phiFD}, we get that
\begin{equation}\label{eq:phiChar}
	\begin{array}{ll}
		\partial_x\Phi(x_{fd}(t),t) = \xi_0 & \quad\quad\quad\partial_{tx}\Phi(x_{fd}(t),t) = - M(t)\dot x_{fd}(t)
		\\[5pt]
		\partial_t\Phi(x_{fd}(t),t) = \dot \omega(t) -\xi_0\dot x_{fd}(t) & \quad\quad\quad\partial_{xxx}\Phi(x_{fd}(t),t) = 0 
		\\[5pt]
		\partial_{xx}\Phi(x_{fd}(t),t) = M(t) & \quad\quad\quad\partial_{txx}\Phi(x_{fd}(t),t) = \dot M(t).
	\end{array} 
\end{equation}
Plugging this in \eqref{eq:eikonalFD}, \eqref{eq:eikonalFDder} and \eqref{eq:eikonalFDder2}, we then obtain that
\begin{align*}
	&\mathcal R_{fd}(x_{fd}(t),t) = 4c\sin^2\left(\frac{\xi_0}{2}\right) - \Big(\dot \omega(t) - \xi_0\dot x_{fd}(t)\Big)^2
	\\
	&\partial_x\mathcal R_{fd}(x_{fd}(t),t) = 2\bigg(c\sin(\xi_0) + \Big(\dot \omega(t) - \xi_0\dot x_{fd}(t)\Big)\dot x_{fd}(t)\bigg)M(t)
	\\
	&\partial_{xx}\mathcal R_{fd}(x_{fd}(t),t) = 2\bigg(\Big(c\cos(\xi_0)- \dot x_{fd}(t)^2\Big)M(t)^2 - \Big(\dot \omega(t) - \xi_0\dot x_{fd}(t)\Big)\dot M(t)\bigg).
\end{align*}

Therefore, by imposing \eqref{eq:Rfdzero}, we have that the functions $\omega$ and $M$ in \eqref{eq:phiFD} are obtained by solving the following ODE system
\begin{subequations}
	\begin{empheq}[left=\empheqlbrace]{align}
		&\Big(\dot\omega(t) - \xi_0\dot x_{fd}(t)\Big)^2 = 4c\sin^2\left(\frac{\xi_0}{2}\right) \label{eq:phiODE_FDeq1}
		\\[5pt]
		&\Big(\dot\omega(t) - \xi_0\dot x_{fd}(t)\Big)\dot x_{fd}(t) = -c\sin(\xi_0) \label{eq:phiODE_FDeq2}
		\\[7pt]
		&\Big(\dot\omega(t) - \xi_0\dot x_{fd}(t)\Big)\dot M(t) = \Big(c\cos(\xi_0)- \dot x_{fd}(t)^2\Big)M(t)^2 \label{eq:phiODE_FDeq3}
	\end{empheq}
\end{subequations}
with initial conditions $(\omega(0),M(0)) = (\omega_0,M_0)$. In what follows, without losing generality, we will always assume $\omega_0 = 0$.

\subsubsection{Solution of the ODE system \eqref{eq:phiODE_FDeq1}-\eqref{eq:phiODE_FDeq3}}

Let us start by observing that the first equation \eqref{eq:phiODE_FDeq1} is actually redundant, which is not surprising since the ODE system has only two unknowns ($\omega(t)$ and $M(t)$). In fact, by taking the square in both terms of the second equation \eqref{eq:phiODE_FDeq2}, and using the explicit expression of the finite difference bi-characteristic rays $x_{fd}(t)$ obtained in \eqref{eq:rayFD}, we have
\begin{align*}
	c^2\sin^2(\xi_0) = 4c^2\sin^2\left(\frac{\xi_0}{2}\right)\cos^2\left(\frac{\xi_0}{2}\right) = \Big(\dot\omega(t) - \xi_0\dot x_{fd}(t)\Big)^2\dot x_{fd}(t)^2 = c\Big(\dot\omega(t) - \xi_0\dot x_{fd}(t)\Big)^2\cos^2\left(\frac{\xi_0}{2}\right),
\end{align*}
so that we immediately get 
\begin{align*}
	\Big(\dot\omega(t) - \xi_0\dot x_{fd}(t)\Big)^2 = 4c\sin^2\left(\frac{\xi_0}{2}\right).
\end{align*}

In other words, a function $\omega(t)$ satisfying \eqref{eq:phiODE_FDeq2} will automatically solve also \eqref{eq:phiODE_FDeq1}. In view of this, the original ODE system reduces to
\begin{subequations}
	\begin{empheq}[left=\empheqlbrace]{align}
		&\Big(\dot\omega(t) - \xi_0\dot x_{fd}(t)\Big)\dot x_{fd}(t) = -c\sin(\xi_0) \label{eq:phiODE_FD2eq2}
		\\[7pt]
		&\Big(\dot\omega(t) - \xi_0\dot x_{fd}(t)\Big)\dot M(t) = -c\sin^2\left(\frac{\xi_0}{2}\right)M(t)^2 \label{eq:phiODE_FD2eq3}
	\end{empheq}
\end{subequations}
where we have used the fact that
\begin{align*}
	c\cos(\xi_0)- \dot x_{fd}(t)^2 = c\cos(\xi_0)- c\cos^2\left(\frac{\xi_0}{2}\right) = -c\sin^2\left(\frac{\xi_0}{2}\right).
\end{align*}
Now, replacing \eqref{eq:rayFD} into \eqref{eq:phiODE_FD2eq2}, we obtain 
\begin{align}\label{eq:phiODE_FD3eq2}
	\pm\sqrt{c}\Big(\dot\omega(t) - \xi_0\dot x_{fd}(t)\Big)\cos\left(\frac{\xi_0}{2}\right) = -c\sin(\xi_0) \quad\longrightarrow\quad \dot\omega(t) - \xi_0\dot x_{fd}(t) = \mp2\sqrt{c}\sin\left(\frac{\xi_0}{2}\right), 
\end{align}
from which we can easily compute
\begin{align}\label{eq:phiODE_FDsol_omega}
	\omega(t) = \pm\sqrt{c}\left(\xi_0\cos\left(\frac{\xi_0}{2}\right) - 2\sin\left(\frac{\xi_0}{2}\right)\right)t. 
\end{align}
Moreover, notice that, when 
\begin{align*}
	\xi_0\cos\left(\frac{\xi_0}{2}\right) = 2\sin\left(\frac{\xi_0}{2}\right),
\end{align*}
which is one of the two solutions of \eqref{eq:xi0hat}, we have $\omega(t) = 0$. This is consistent with the fact that, for the above value of $\xi_0$, the phase function $\phi$ given in \eqref{eq:phi} for the ansatz of the continuous wave equation \eqref{eq:mainEqR} satisfies the finite difference eikonal equation 
\begin{align*}
	\mathcal R_{fd}(x_{fd}(t),t) = 0
\end{align*}
and, therefore, the correction term $\omega(t)$ is not needed.

Finally, to compute $M$ and conclude the resolution of the ODE system \eqref{eq:phiODE_FDeq1}-\eqref{eq:phiODE_FDeq3}, we replace \eqref{eq:phiODE_FDsol_omega} into \eqref{eq:phiODE_FD2eq3}. In this way, we obtain
\begin{align}\label{eq:phiODE_FDsol_M}
	M(t)= \frac{M_0}{1\mp\frac{M_0\sqrt{c}}{2}\sin\left(\frac{\xi_0}{2}\right)t}.
\end{align}

Let us recall, however, that to guarantee that the phase function $\Phi$ really generates an ansatz \eqref{eq:ansatzFD_h} with a Gaussian envelop we shall have $\Im (M(t)) > 0$ for all $t>0$. This can be ensured by a proper choice of the initial datum $M_0$. In fact, a simple calculation gives us
\begin{align*}
	\Im (M(t)) = \frac{\left(1\mp\frac{\Re (M_0)\sqrt{c}}{2}\sin\left(\frac{\xi_0}{2}\right)t\right)\Im (M_0) - \left(1\mp\frac{\Im (M_0)\sqrt{c}}{2}\sin\left(\frac{\xi_0}{2}\right)t\right)\Re (M_0)}{1 + \frac{|M_0|^2 c}{4}\sin^2\left(\frac{\xi_0}{2}\right)t^2 \mp\Re (M_0)\sqrt{c}\sin\left(\frac{\xi_0}{2}\right)t}.
\end{align*}
We then immediately see that it is enough to take 
\begin{equation*}
	M_0\in\CC \quad\text{ with } \begin{array}{ll} \Re (M_0) = 0 \\ \Im (M_0) >0 \end{array} 
\end{equation*}
to obtain
\begin{align*}
	\Im (M(t)) = \frac{\Im (M_0)}{1 + \frac{|M_0|^2 c}{4}\sin^2\left(\frac{\xi_0}{2}\right)t^2} > 0 \quad\text{ for all } t>0.
\end{align*}
This concludes the construction of the phase $\Phi$.

\subsubsection{Design of the amplitude $A$}

To construct the amplitude function $A$, we start again from \eqref{eq:FDdAlambertian_h} and notice that, since the eikonal equation $\mathcal R_{fd}=0$ is solved up to the second order on the ray $x_{fd}(t)$, the last term in that identity will not contribute in the computation. On the other hand, the second term of order $h^{-\frac 14}$ will definitely contribute. In particular, $A$ shall be determined by imposing 
\begin{align*}
	\mathcal R_1 A_j(x_{fd}(t),t) = 0.
\end{align*}

As for the phase $\Phi$ before, we notice that the expression of $\mathcal R_1$ involves discrete partial derivatives with respect to the space variable, which may lead to cumbersome computations. To avoid them, we replace $\mathcal R_1$ with the following expression
\begin{align}\label{eq:R1tilde}
	\RT_1 A_j := 2\partial_t A_j\partial_t\Phi_j - 2c\partial_x A_j\sin(\partial_x\Phi_j) + A_j\Big(\partial_t^2\Phi_j - c\cos(\partial_x\Phi_j)\partial_x^2\Phi_j\Big), 
\end{align}
in which only continuous derivatives in space appear. Once again, we are allowed to do that since the error is small. Indeed, we have
\begin{align*}
	\mathcal R_1A_j - \RT_1A_j = 2c \mathcal K_1 + A_j\mathcal K_2
\end{align*}
with 
\begin{align*}
	\mathcal K_1 := \partial_xA_j\sin(\partial_x\Phi_j) -\Dh A_j\sin\left(\Dh\Phi_j\right)e^{i\frac{h}{2c}\Dhc\Phi_j}
\end{align*}
and
\begin{align*}
	\mathcal K_2 := c\cos(\partial_x\Phi_j)\partial_x^2\Phi_j - \cos\left(\Dh\Phi_j\right)\frac{4c\sin\left(\frac{h}{4c}\Dhc\Phi_j\right)}{h}e^{i\frac{h}{4c}\Dhc\Phi_j}. 
\end{align*}
Moreover, we can easily rewrite
\begin{align*}
	\mathcal K_1 =&\; \Big(\partial_xA_j - \Dh A_j\Big)\sin(\partial_x\Phi_j) + \Dh A_j\Big(\sin\left(\partial_x\Phi_j\right) - \sin\left(\Dh\Phi_j\right)\Big) + \Dh A_j\sin\left(\Dh\Phi_j\right)\Big(1-e^{i\frac{h}{2c}\Dhc\Phi_j}\Big)
	\\[10pt]
	\mathcal K_2 =&\; c\Big(\cos(\partial_x\Phi_j)-\cos(\Dh\Phi_j)\Big)\partial_x^2\Phi_j + \cos(\Dh\Phi_j)\left(c\partial_x^2\Phi_j - \frac{4c\sin\left(\frac{h}{4c}\Dhc\Phi_j\right)}{h}\right) 
	\\
	&+ \cos(\Dh\Phi_j)\frac{4c\sin\left(\frac{h}{4c}\Dhc\Phi_j\right)}{h}\left(1- e^{i\frac{h}{4c}\Dhc\Phi_j} \right)
\end{align*}
and estimate
\begin{align*}
	|\mathcal K_1| \leq& \; |\partial_xA_j - \Dh A_j| + |\Dh A_j||\sin\left(\partial_x\Phi_j\right) - \sin\left(\Dh\Phi_j\right)| + |\Dh A_j|\left|1-e^{i\frac{h}{2c}\Dhc\Phi_j}\right|
	\\
	\leq& \; |\partial_xA_j - \Dh A_j| + |\Dh A_j||\partial_x\Phi_j - \Dh\Phi_j| + \mathcal C |\Dhc\Phi_j||\Dh A_j|h \leq\mathcal C |\Dhc\Phi_j||\Dh A_j|h + \mathcal O(h^2)
	\\[10pt]
	|\mathcal K_2| \leq& \; c|\partial_x^2\Phi_j||\cos(\partial_x\Phi_j)-\cos(\Dh\Phi_j)| + \left|c\partial_x^2\Phi_j - \frac{4c\sin\left(\frac{h}{4c}\Dhc\Phi_j\right)}{h}\right| 
	\\
	&+ \left|\frac{4c\sin\!\left(\frac{h}{4c}\Dhc\Phi_j\right)}{h}\right|\left|1- e^{i\frac{h}{4c}\Dhc\Phi_j}\right|
	\\
	\leq& \; c|\partial_x^2\Phi_j||\partial_x\Phi_j-\Dh\Phi_j| + |c\partial_x^2\Phi_j - \Dhc\Phi_j| + \mathcal C|\Dhc\Phi_j|^2h \leq \mathcal C|\Dhc\Phi_j|^2h + \mathcal O(h^2).
\end{align*}
Therefore,
\begin{align}\label{eq:R1tildeError}
	|\mathcal R_1A_j - \RT_1A_j| \leq \mathcal C\Big(|\Dhc\Phi_j| + |\Dhc\Phi_j|^2\Big)h + \mathcal O(h^2)
\end{align}
and, if we rewrite
\begin{align}\label{eq:dAlambertianFD}
	\square_{c,h} u^h_{fd,j} = e^{\frac ih\Phi_j}\left[h^{\frac 34}\left(\mathcal R_0A_j + A_j \frac{\mathcal R_2-\mathcal R_{fd}}{h^2}\right) + ih^{-\frac 14}(\mathcal R_1A_j - \RT_1 A_j) + ih^{-\frac 14}\RT_1 A_j + h^{-\frac 54}A_j\mathcal R_{fd}\right], \notag
	\\	
\end{align}
we then see that it is enough to design $A_j$ such that
\begin{align*}
	\RT_1 A_j(x_{fd}(t),t) = 0.
\end{align*}
By means of \eqref{eq:rayFD}, \eqref{eq:phiChar} and \eqref{eq:R1tilde}, we can readily see that this amounts at solving the equation
\begin{align*}
	\frac{d}{dt} A(x_{fd}(t),t) = \frac{\sqrt{c}}{4}\sin\left(\frac{\xi_0}{2}\right) M(t)A(x_{fd}(t),t),
\end{align*} 
from which, taking into account the explicit expression of $M$ given in \eqref{eq:phiODE_FDsol_M}, we obtain
\begin{align*}
	A(x_{fd}(t),t) = A(x_0,0)e^{\frac{\sqrt{c}}{4}\sin\left(\frac{\xi_0}{2}\right)\int_0^t M(s)\,ds} = A(x_0,0)e^{\mp\frac 12 \ln\left(1-\frac{M_0\sqrt{c}}{2}\sin\left(\frac{\xi_0}{2}\right)t\right)}.
\end{align*} 
Choosing $A(x_0,0)=1$, we then get
\begin{align*}
	A(x_{fd}(t),t) = e^{\mp\frac 12 \ln\left(1-\frac{M_0\sqrt{c}}{2}\sin\left(\frac{\xi_0}{2}\right)t\right)}.
\end{align*} 

In the same spirit of what we did in Section \ref{sec:GBcontinuous}, this suggests to take the amplitude function in our discrete ansatz \eqref{eq:ansatzFD_h} in the form
\begin{align}\label{eq:aFD}
	A(x,t) = e^{-(x-x_{fd}(t))^2} e^{\mp\frac 12 \ln\left(1-\frac{M_0\sqrt{c}}{2}\sin\left(\frac{\xi_0}{2}\right)t\right)}.
\end{align} 
In conclusion, we finally have from \eqref{eq:ansatzFD_h}, \eqref{eq:phiFD} and \eqref{eq:aFD} that
\begin{align}\label{eq:ansatzFDexpl}
	&u^h_{fd}(x,t) = h^{\frac 34} e^{-(x-x_{fd}(t))^2} e^{\mp\frac 12 \ln\left(1-\frac{M_0\sqrt{c}}{2}\sin\left(\frac{\xi_0}{2}\right)t\right)} e^{\frac ih \Big[\omega(t) + \xi_0(x-x_{fd}(t)) + \frac 12 M(t)(x-x_{fd}(t))^2\Big]}, 
\end{align}
with $\omega$ and $M$ given by \eqref{eq:phiODE_FDsol_omega} and \eqref{eq:phiODE_FDsol_M}, respectively.

\begin{remark}
To conclude this section, let us highlight again the main differences between the finite difference GB ansatz \eqref{eq:ansatzFDexpl} and the continuous one we have introduced in Theorem \ref{thm:ApproxSol}. These differences arise at two levels.
\begin{itemize}
	\item[1.] First of all, in \eqref{eq:ansatzFDexpl}, we have a phase function 
	\begin{align}\label{eq:PhiRem}
		\Phi(x,t) = \omega(t) + \xi_0(x-x_{fd}(t)) + \frac 12 M(t)(x-x_{fd}(t)),
	\end{align}	
	with $\omega$ and $M$ given by \eqref{eq:phiODE_FDsol_omega} and \eqref{eq:phiODE_FDsol_M}, respectively. This differs from the phase 
	\begin{align}\label{eq:phiRem}
		\phi(x,t) = \xi_0(x-x_{fd}(t)) + \frac 12 M_0(x-x_{fd}(t))
	\end{align}
	of Theorem \ref{thm:ApproxSol} in two aspects. On the one hand, in \eqref{eq:PhiRem} we have a time-dependent complex-valued function $M(t)$ in the quadratic part of the phase, which replaces the constant $M_0$ in \eqref{eq:phiRem}. In addition to that, we remark the presence of the correction term $\omega$. Both modifications stem from the fact that a phase in the form \eqref{eq:phiRem} does not fulfill the condition \eqref{eq:Rfdzero} and, in particular, does not solve the finite difference eikonal equation on the semi-discrete characteristics. Therefore, it is not suitable to build a GB ansatz correctly approximating the solutions of \eqref{eq:mainEqFD}. 
	\item[2.] Secondly, also the amplitude function 
	\begin{align*}
		A(x,t) = e^{-(x-x_{fd}(t))^2} e^{\mp\frac 12 \ln\left(1-\frac{M_0\sqrt{c}}{2}\sin\left(\frac{\xi_0}{2}\right)t\right)}
	\end{align*}
	in \eqref{eq:ansatzFDexpl} presents some modification with respect to its continuous counterpart
	\begin{align*}
		a(x,t) = e^{-(x-x_{fd}(t))^2}.
	\end{align*}
	This originates directly from \eqref{eq:R1tilde}, where the presence of some trigonometric terms of the phase $\Phi$ (that do not appear at the continuous level) needs to be compensated through some small adjustments when designing $A$.	
\end{itemize} 
\end{remark}

\subsection{The multi-dimensional case} We discuss here briefly the extension of Theorem \ref{thm:ApproxSolFD} to the general multi-dimensional case $d\geq 1$. In particular, we shall point out the main changes in the construction of the GB ansatz with respect to the one-dimensional case presented in the previous sections.

When going from 1D to multi-D, the GB ansatz for the semi-discrete wave equation \eqref{eq:mainEqFD} changes at three levels. 

First of all, the scaling factor $h^{\frac 34}$ in \eqref{eq:ansatzFD_Thm} needs to be adjusted to the dimension of the problem. In particular, to remain consistent with the continuous case of Theorem \ref{thm:ApproxSol}, we shall take a scaling factor $h^{1-\frac d4}$. Our ansatz for the semi-discrete wave equation \eqref{eq:mainEqFD} in dimension $d\geq 1$ will then take the form
\begin{align}\label{eq:ansatzFD_multiD}
	u_{fd}^h(\bx,t) = h^{1-\frac d4}A(\bx,t)e^{\frac ih \Phi(\bx,t)}.
\end{align}

Secondly, also the phase and amplitude functions $\Phi(\bx,t)$ and $A(\bx,t)$ need to be adjusted, in a way that we describe below.

\subsubsection{Multi-dimensional semi-discrete phase}

The phase function $\Phi(\bx,t)$ is still of the form
\begin{align}\label{eq:PhiFD_multiD}
	\Phi(\bx,t) = \omega(t) + \bxi_0\cdot(\bx-\bx_{fd}(t)) + \frac 12 (\bx-\bx_{fd}(t))\cdot\Big[M(t)(\bx-\bx_{fd}(t))\Big],
\end{align}
but has to be designed so that it fulfills
\begin{align*}
	D_\bx^\alpha \mathcal R_{fd}(\bx(t),t) = 0 \text{ for all } t\in\RR \text{ and } \alpha\in\NN^d\text{ with }|\alpha|\in\{0,1,2\}, 
\end{align*}
with 
\begin{align*}
	\mathcal R_{fd}(\bx,t) := 4c\left|\sin\left(\frac 12\nabla\Phi\right)\right|^2-\Phi_t^2 = 4c\sum_{i=1}^d\sin^2\left(\frac{\partial_{x_i}\Phi}{2}\right)-\Phi_t^2.
\end{align*}

Following the construction we have presented in Section \ref{sec:PhaseFD} for the one-dimensional case, we then obtain that the functions $\omega(t)$ and $M(t)$ in \eqref{eq:PhiFD_multiD} have to solve the coupled ODE system
\begin{subequations}
	\begin{empheq}[left=\empheqlbrace]{align}
	&\Big(\dot\omega(t)-\bxi_0\cdot\dot\bx(t)\Big)^2 = 4c\left|\sin\left(\frac{\bxi_0}{2}\right)\right|^2 \label{eq:ODEsystem_multiDeq1}
	\\[7pt]
	&\Big(\dot\omega(t)-\bxi_0\cdot\dot\bx(t)\Big)\dot M(t) = M(t)^\top \Theta M(t) \label{eq:ODEsystem_multiDeq2}
	\end{empheq}
\end{subequations}
with initial data $(\omega(0),M(0))=(0,M_0)$ and where $\theta = \text{diag}\Big(\theta_1,\theta_2,\ldots,\theta_d\Big)\in\RR^{d\times d}$ is a real and diagonal $d\times d$ matrix with elements
\begin{align*}
	\theta_i:= c\cos(\xi_{0,i})-\dot x_{fd,i}^2 , \quad i\in\{1,\ldots,d\}.
\end{align*}
By using the explicit expression \eqref{eq:rayFD} for the ray $\bx_{fd}(t)$, we immediately obtain from \eqref{eq:ODEsystem_multiDeq1} that
\begin{align*}
	\omega(t) = \pm2\sqrt{c}\left|\sin\left(\frac{\bxi_0}{2}\right)\right|\left(\frac{\bxi_0\cdot\sin(\bxi_0)}{4\left|\sin\left(\frac{\bxi_0}{2}\right)\right|^2} \pm 1\right)t.
\end{align*}
Finally, we obtain from \eqref{eq:ODEsystem_multiDeq2} that $M(t)$ is the solution of the differential Riccati equation
\begin{align*}
	\pm 2\sqrt{c}\left|\sin\left(\frac{\bxi_0}{2}\right)\right|\dot M(t) = M(t)^\top \Theta M(t) 
\end{align*}
and we know from \cite{babich1999higher,ralston1982gaussian} that, given a symmetric matrix $M_0\in\CC^{d\times d}$ with $\Im(M_0) > 0$, there exist a global solution $M(t)$ of \eqref{eq:ODEsystem_multiDeq2} that satisfies $M(0) = M_0$, $M(t) = M(t)^\top$ and $\Im(M(t)) > 0$ for all $t$.

\subsubsection{Multi-dimensional semi-discrete amplitude}

Finally, the amplitude function $A(\bx,t)$ has to be designed such that
\begin{align}\label{eq:AmultiDeq}
	\widetilde{\mathcal R}_1A(\bx(t),t) = 0
\end{align}
with 
\begin{align*}
	\widetilde{\mathcal R}_1A := 2A_t\Phi_t-2c\nabla A\cdot\sin(\nabla\phi) + A\Big(\Phi_t^2 - c\nabla\Phi^\top \Psi(\Phi)\nabla\Phi\Big),
\end{align*}
where $\Phi$ is given by \eqref{eq:PhiFD_multiD}-\eqref{eq:ODEsystem_multiDeq1}-\eqref{eq:ODEsystem_multiDeq2} and where $\Psi(\Phi) = \text{diag}\Big(\psi_1(\Phi),\psi_2(\Phi),\ldots,\psi_d(\Phi)\Big)\in\RR^{d\times d}$ is a real and diagonal $d\times d$ matrix with elements
\begin{align*}
	\psi_i(\Phi):= \cos(\partial_{x_i}\Phi), \quad i\in\{1,\ldots,d\}.
\end{align*}
Moreover, from \eqref{eq:PhiFD_multiD} and \eqref{eq:ODEsystem_multiDeq1} we have that
\begin{align*}
	&\Phi_t(\bx_{fd}(t),t) = \dot\omega(t) - \bxi_0\cdot\dot\bx_{fd}(t) = \pm 2\sqrt{c}\left|\sin\left(\frac{\bxi_0}{2}\right)\right|
	\\	
	&\nabla\Phi(\bx_{fd}(t),t)^\top \Psi(\Phi(\bx_{fd}(t),t))\nabla\Phi(\bx_{fd}(t),t) = (\bxi_0\odot\bxi_0)\cdot\cos(\bxi_0).
\end{align*}
Using this in \eqref{eq:AmultiDeq}, we can readily see that the amplitude $A$ has to solve the equation 
\begin{align*}
	\frac{d}{dt}A(\bx_{fd}(t),t) = \C(c,\bxi_0) A(\bx_{fd}(t),t) 
\end{align*}
with 
\begin{align*}
	\C(c,\bxi_0) = \sqrt{c}\; \frac{(\bxi_0\odot\bxi_0)\cdot\cos(\bxi_0)-4\left|\sin\left(\frac{\bxi_0}{2}\right)\right|^2}{4\left|\sin\left(\frac{\bxi_0}{2}\right)\right|}. 
\end{align*}
Choosing the initial datum $A(\bx_{fd}(0),0) = A(\bx_0,0) = 1$, we then obtain that
\begin{align*}
	A(\bx_{fd}(t),t) = e^{\C(c,\bxi_0)t}.
\end{align*}

In the same spirit of what we did in the one-dimensional case, this suggests to take the amplitude function $A(\bx,t)$ in the form 
\begin{align*}
	A(\bx,t) = e^{-|\bx-\bx_{fd}(t)|^2}e^{\C(c,\bxi_0)t}.
\end{align*}

\subsubsection{Concentration along the semi discrete characteristics}
Once the ansatz has been built according to the above discussion, we can prove that it provides quasi-solution to the semi-discrete wave equation \eqref{eq:mainEqFD} whose energy is concentrated along the rays $\bx_{fd}(t)$ given by \eqref{eq:rayFD}. In particular, we have:
\begin{itemize}
	\item[1.] $\displaystyle \sup_{t\in(0,T)}\norm{\square_{c,h}u^h_{fd}(\cdot,t)}{\ell^2(h\ZZ^d)} = \mathcal O(h^{\frac 12})$ as $h\to 0^+$.	
	
	\vspace{0.25cm}	
	\item[2.] $\displaystyle \mathcal E_h[u_{fd}^h](t) = \mathcal O(1)$ as $h\to 0^+$.
	
	\vspace{0.25cm}	
	\item[3.] $\displaystyle \sup_{t\in(0,T)} \frac{h^d}{2}\sum_{j\in\ZZ^{d,\dagger}(t)} \Big(|\partial_t u_{fd,j}^h|^2 + c|\fwd u_{fd,j}^h|^2\Big)\leq \C_1(A,\Phi) \Big(1 + h + h^2\Big) e^{-\C_2(M_0)h^{-\frac 12}}$ 
	
	\noindent where $\C_1(A,\Phi)>0$ and $\C_2(M_0)>0$ are two positive constants independent of $h$ and we have denoted 
	\begin{align*}
		\ZZ^{d,\dagger}(t) := \Big\{ \bj\in\ZZ^d\,:\, |\bx_\bj-\bx_{fd,\bj}(t)|>h^{\frac 14}\Big\}.
	\end{align*}
\end{itemize}

The proofs of the above facts are totally analogous to what we have already presented in the one-dimensional case. We leave the details to the reader.
\section{Numerical simulations}\label{sec:numerics}

We present here some numerical simulations to illustrate the results of the previous sections. On the one hand, we shall provide a graphical comparison of the solution to our finite difference wave equation \eqref{eq:mainEqFD} and the ansatz $u_{fd}^h$ given by \eqref{eq:ansatzFD_h}. This will give us a visual confirmation of the correctness of our theoretical results of the previous sections. On the other hand, we shall check numerically the approximation rates obtained in Theorem \ref{thm:ApproxSolFD} for the GB ansatz $u_{fd}^h$. 

To compute numerically the solution of the finite difference wave equation \eqref{eq:mainEqFD}, we have employed a standard leapfrog scheme in time 
\begin{align}\label{eq:LeapFrog}
	\partial_t^2 u_\bj(t) \sim \frac{1}{h_t^2}\Big(u_\bj^{n+1} - 2u_\bj^n + u_\bj^{n-1}\Big), \quad u_\bj^n:=u(\bx_\bj,t_n)	
\end{align}
on a uniform mesh $\mathcal T = \{t_n\}_{n = 1}^N$ of size $h_t$, satisfying the Courant-Friedrichs-Lewy (CFL) condition (which is necessary since the method is explicit). This corresponds to taking $h_t = \mu h$, with $\mu\in (0,1)$. In our forthcoming simulations, we have always chosen $\mu = 0.1$, although other selections of $\mu\in (0,1)$ are possible without affecting the stability of the numerical scheme and the propagation properties of our numerical solution. 

To remain consistent with \eqref{eq:ansatzFD_h}, the initial data in \eqref{eq:mainEqFD} are constructed starting from the Gaussian profile $u_{fd}^h(x,0)$. In more detail, we have taken 
\begin{align*}
	u_\bj^0 = u_{fd}^h(\bx_\bj,0) \quad\text{ and }\quad u_\bj^1 = \partial_t u_{fd}^h(\bx_\bj,0).
\end{align*}

Moreover, we have considered different values of the frequency $\bxi_0$ to illustrate the different propagation properties of the rays, hence of the solution.

We start by considering the one-dimensional case and providing in Figure \ref{fig:approx} a first graphical confirmation of the accuracy of our construction. This is done by displaying the $\ell^2(h\ZZ)$ error 
\begin{align*}
	e = \norm{u_{fd}^h-u_j^n}{2}=: \norm{u_{fd}^h-u_j^n}{\ell^2(h\ZZ)}
\end{align*}
between our ansatz $u_{fd}^h$ given by \eqref{eq:ansatzFD_h} and the numerical solution $u_j^n$ obtained through \eqref{eq:LeapFrog}, computed for different values of the frequency $\xi_0$, and noticing that this error decreases as $h\to 0^+$.
\begin{SCfigure}[0.75][h]
	\centering
	\includegraphics[width=0.5\textwidth]{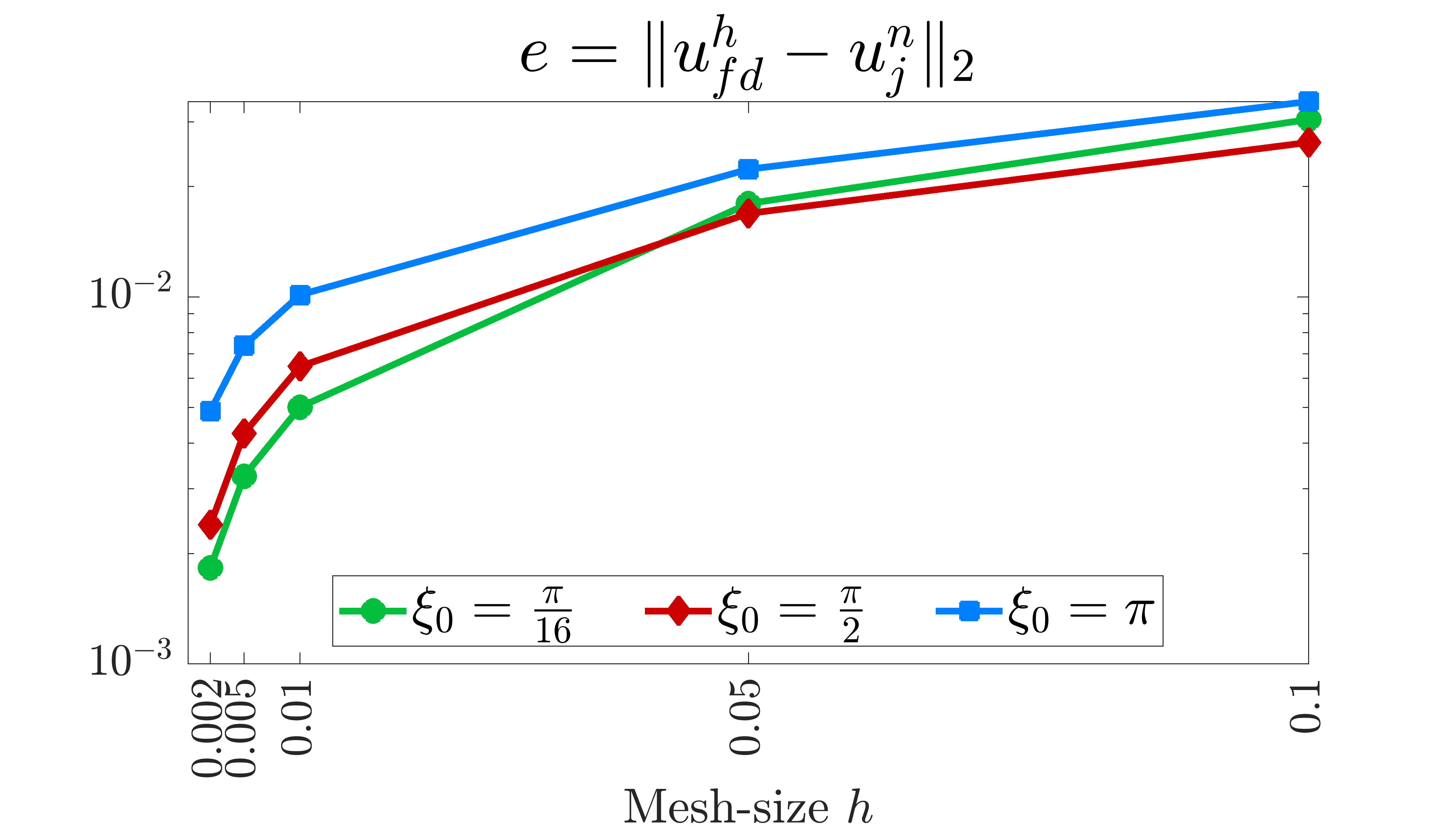}
	\caption{$\ell^2(h\ZZ)$ error between the ansatz $u_{fd}^h$ given by \eqref{eq:ansatzFD_h} and the numerical solution $u_j^n$ obtained through \eqref{eq:LeapFrog}, in space dimension $d=1$.}\label{fig:approx}
\end{SCfigure}

Moreover, in Figure \ref{fig:ansatz}, we display and compare the dynamical behavior of the ansatz $u_{fd}^h$, on the left, and the numerical solution $u_j^n$, on the right, for all the frequency previously considered.
\begin{figure}[h]
	\centering
	\includegraphics[width=\textwidth]{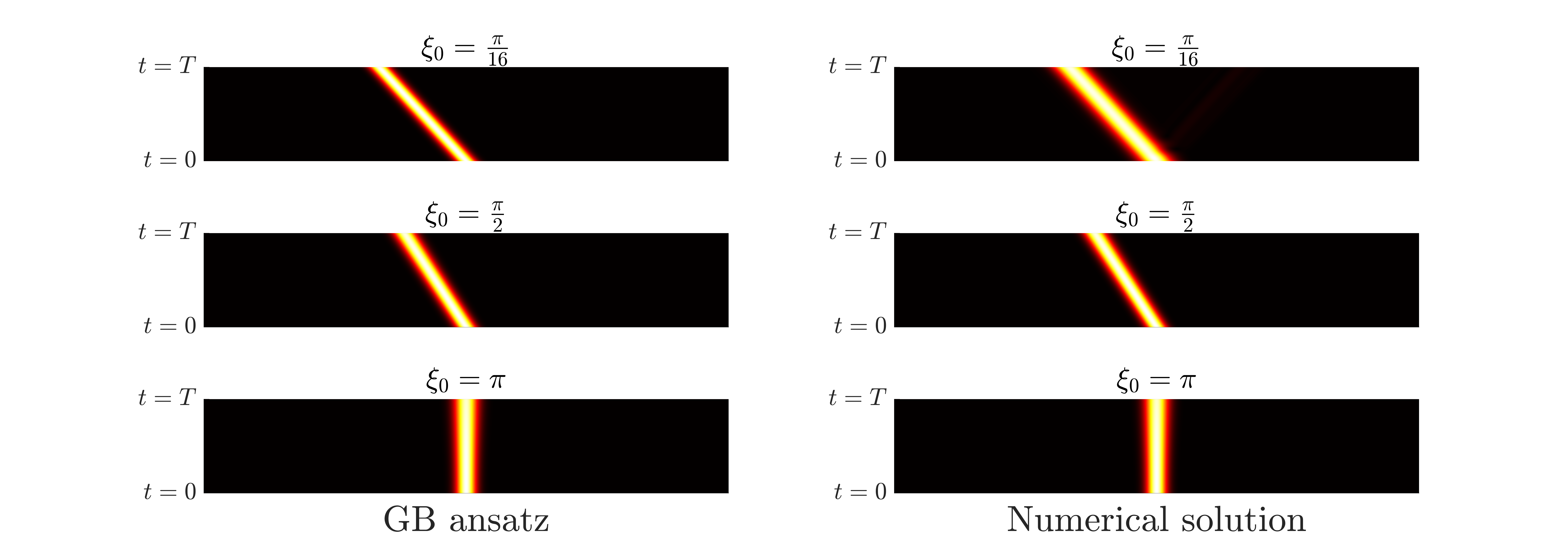}
	\caption{Comparison between the ansatz \eqref{eq:ansatzFD_h} (left) and  solution of \eqref{eq:mainEqFD} (right), in space dimension $d=1$. Both solutions are localized on the characteristic rays starting from $x_0=0$, with different frequencies $\xi_0\in (0,\pi]$.}\label{fig:ansatz}
\end{figure}

We can appreciate that both the solution of \eqref{eq:mainEqFD} computed through the scheme \eqref{eq:LeapFrog} and the GB ansatz \eqref{eq:ansatzFD_h} show the same dynamical behavior, remaining concentrated along the bi-characteristic ray $x_{fd}(t)$. In addition to that, we can clearly see how, as the frequency $\xi_0$ increases, the propagation properties of these solutions change up to the pathological case $\xi_0 = \pi$ in which we appreciate a lack of propagation in space. This is consistent with the equations for the semi-discrete bi-characteristic rays $x_{fd}(t)$ given in \eqref{eq:rayFD}. 

An analogous behavior can be appreciated also in the two-dimensional case in Figure \ref{fig:ansatz2}, we display and compare again the ansatz $u_{fd}^h$, on the left, and the numerical solution $u_\bj^n$, on the right, for different frequencies $\bxi_0$.

\begin{figure}[h]
	\centering
	\includegraphics[width=\textwidth]{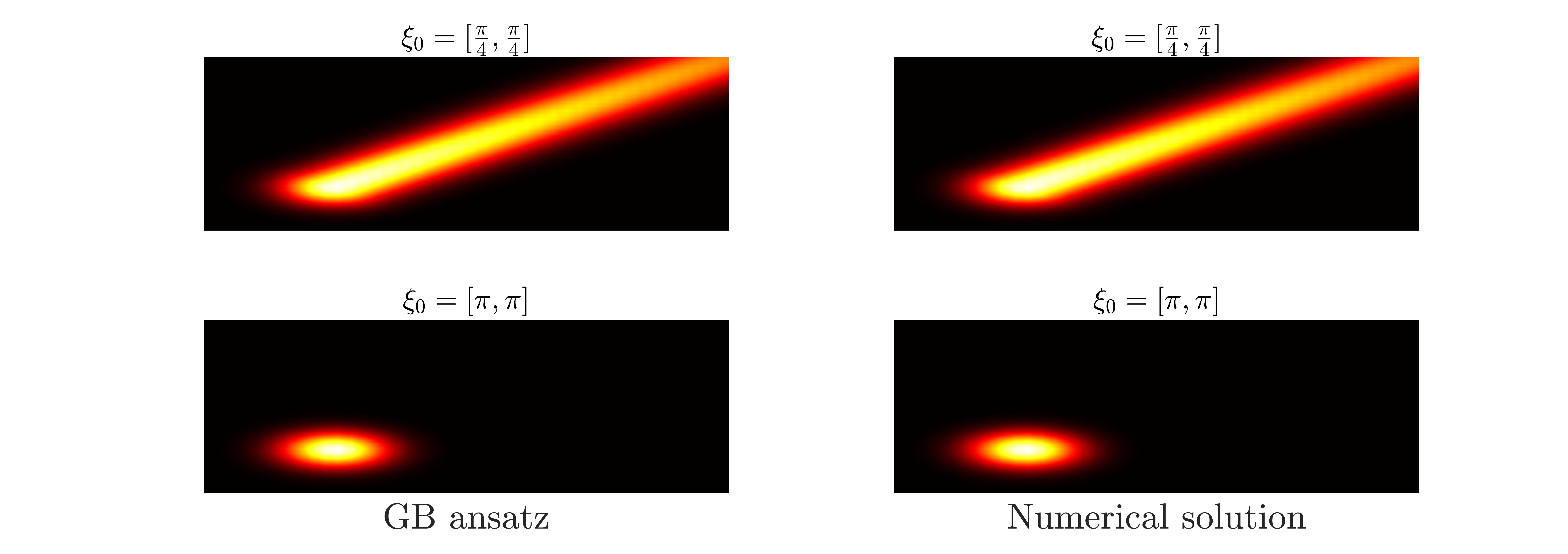}
	\caption{Comparison between the the ansatz \eqref{eq:ansatzFD_h} (left) and  solution of \eqref{eq:mainEqFD} (right) in space dimension $d=2$. Both solutions are localized on the characteristic rays starting from $\bx_0=[-1/4,-1/4]$ ,with different frequencies $\bxi_0\in (0,\pi]\times(0,\pi]$.}\label{fig:ansatz2}
\end{figure}

We can appreciate again that both the solution of \eqref{eq:mainEqFD} computed through the scheme \eqref{eq:LeapFrog} and the GB ansatz \eqref{eq:ansatzFD_h} show the same dynamical behavior, remaining concentrated along the bi-characteristic ray $\bx_{fd}(t)$. In addition to that, we can clearly see how the frequency $\bxi_0=[\pi,\pi]$ is pathological, making the velocity of propagation of the rays \eqref{eq:rayFD} vanish both in the $x$ and in the $y$ direction. This results in a wave that, for all $t$, is trapped at the initial point $\bx_0=\bx_{fd}(0)$, as we can clearly see in the plots. 

Moreover, to confirm the accuracy of our approximation even in this two-dimensional case, we display in Figure \ref{fig:approx2} the $\ell^2(h\ZZ^2)$ error 
\begin{align*}
	e = \norm{u_{fd}^h-u_\bj^n}{2}=: \norm{u_{fd}^h-u_\bj^n}{\ell^2(h\ZZ^2)}
\end{align*}
between our ansatz $u_{fd}^h$ given by \eqref{eq:ansatzFD_h} and the numerical solution $u_\bj^n$ obtained through \eqref{eq:LeapFrog}, computed for different values of the frequency $\xi_0$. 
\begin{SCfigure}[0.75][h]
	\centering
	\includegraphics[width=0.5\textwidth]{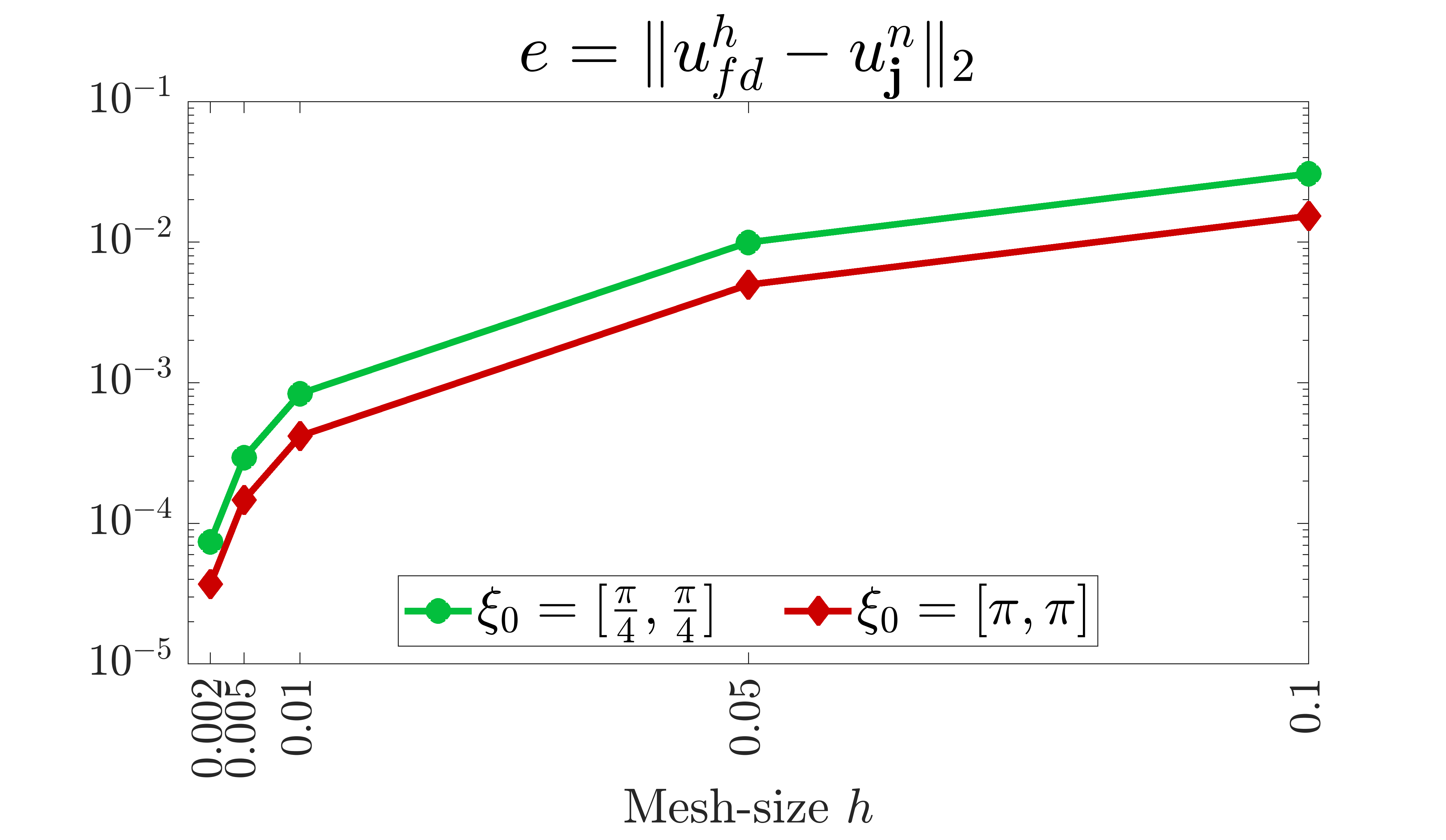}
	\caption{$\ell^2(h\ZZ^2)$ error between the ansatz $u_{fd}^h$ given by \eqref{eq:ansatzFD_h} and the numerical solution $u_\bj^n$ obtained through \eqref{eq:LeapFrog}, in space dimension $d=2$.}\label{fig:approx2}
\end{SCfigure}

Finally, in Table \ref{tab:compare} and Figure \ref{fig:comprare} we collect and display the behavior with respect to the mesh parameter $h$ of the quantities $\mathcal S[u_{fd}^h]$ and $\mathcal E[u_{fd}^h]$ introduced in \eqref{eq:approxSolFD} and \eqref{eq:approxEnergyFD}, once again in the one-dimensional case. For simplicity, we have considered there only solutions with initial frequency $\xi_0=\pi/16$, although other values of $\xi_0$ could have been employed obtaining analogous results.

\begin{table}[h]
	\begin{center}
		\begin{tabular}{|c|c|c|c|c|c|}
			\hline $h$ & $0.1$ & $0.05$ & $0.01$ & $0.005$ & $0.002$ 
			\\
			\hline \multirow{2}{*}{$\mathcal E_h[u_{fd}^h]$} & \multirow{2}{*}{$0.3197$} & \multirow{2}{*}{$0.2633$} & \multirow{2}{*}{$0.2747$} & \multirow{2}{*}{$0.2562$} & \multirow{2}{*}{$0.2599$} 
			\\
			& & & & &
			\\
			\hline\rowcolor{Gray} & & & & &
			\\
			\hline \multirow{2}{*}{$\displaystyle\mathcal S_h[u_{fd}^h]=\sup_{t\in(0,T)}\norm{\square_{c,h}u_{fd}^h(\cdot,t)}{\ell^2(h\ZZ)}$} & \multirow{2}{*}{$0.1084$} & \multirow{2}{*}{$0.0769$} & \multirow{2}{*}{$0.0452$} & \multirow{2}{*}{$0.0311$} & \multirow{2}{*}{$0.0212$} 
			\\
			& & & & &
			\\
			\hline \multirow{2}{*}{$h^{\frac 12}$} & \multirow{2}{*}{$0.3162$} & \multirow{2}{*}{$0.2236$} & \multirow{2}{*}{$0.1$} & \multirow{2}{*}{$0.0707$} & \multirow{2}{*}{$0.0316$} 
			\\
			& & & & &
			\\
			\hline
		\end{tabular}
	\end{center}
	\caption{Behavior with respect to the mesh parameter $h$ of the quantities $\mathcal E[u_{fd}^h]$ and $\mathcal S[u_{fd}^h]$ introduced in Theorem \ref{thm:ApproxSolFD}.}\label{tab:compare}
\end{table}

\begin{SCfigure}[0.7][h]
	\centering
	\includegraphics[width=0.5\textwidth]{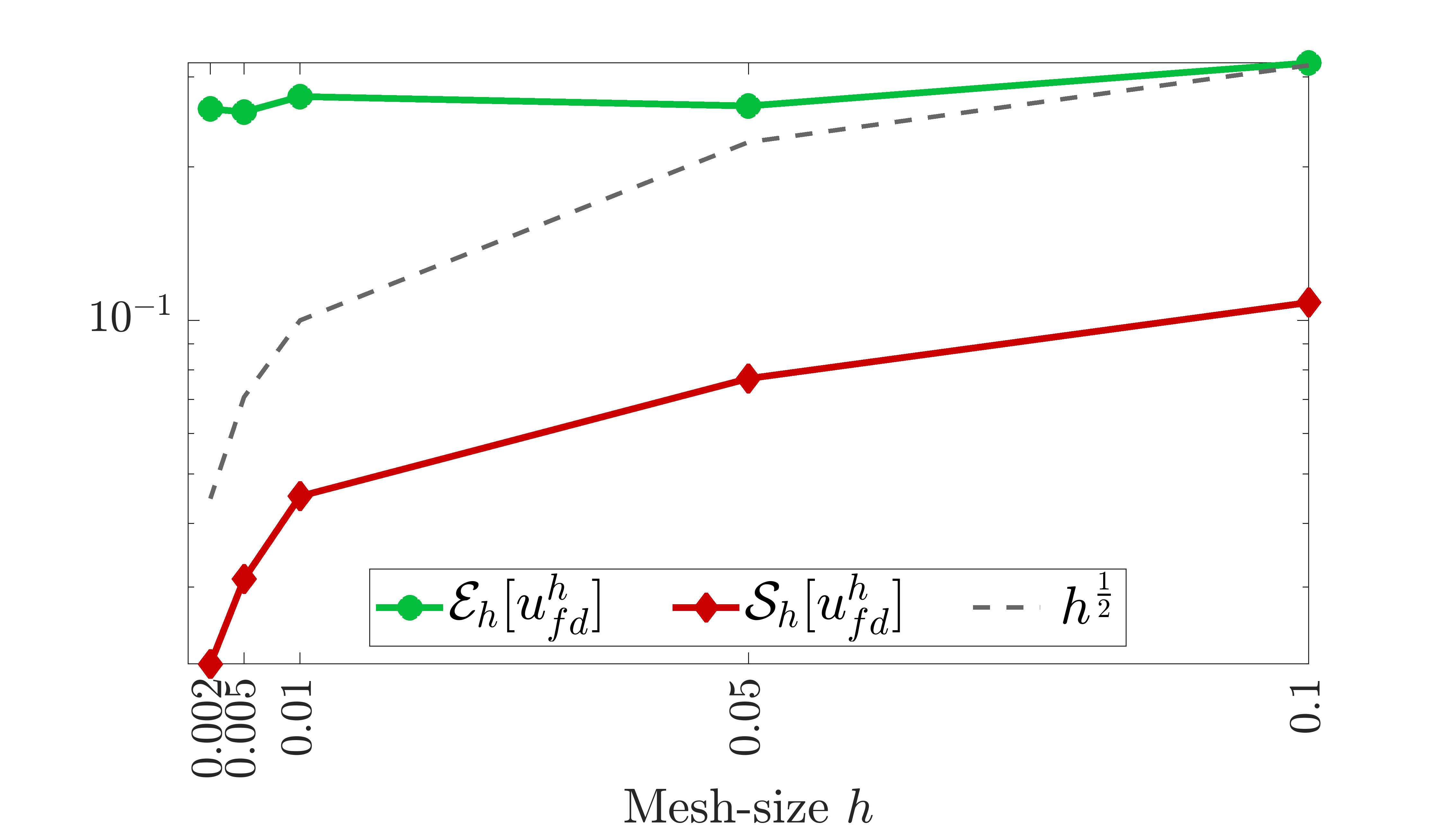}
	\caption{Behavior with respect to the mesh parameter $h$ of the quantities $\mathcal S[u_{fd}^h]$ and $\mathcal E[u_{fd}^h]$ introduced in Theorem \ref{thm:ApproxSolFD}. Case $\xi_0 = \pi/16$.}\label{fig:comprare}
\end{SCfigure}

In both cases, we can clearly appreciate how these quantities of interest behave as anticipated by Theorem \ref{thm:ApproxSolFD}. In particular, while we can see how $\mathcal S_h[u^h_{fd}]$ decreases at a rate $h^{\frac 12}$ as $h\to 0^+$, the energy $\mathcal E_h[u_{fd}^h]$ remains essentially constant with respect to $h$. This provides further confirmation of the accuracy of our theoretical results.

\section{Conclusions and open problems}\label{sec:conclusions}
In this paper, we have discussed the construction of GB solutions for the numerical approximation of wave equations, semi-discretized in space by finite difference schemes. In particular we have focused on the case of constant coefficient wave equations defined on the entire Euclidean space, and we have shown how to accurately build a GB ansatz to describe the solutions' propagation properties. Due to the well-known fact that the solutions of finite difference wave equations may exhibit pathological behaviors such as lack of space propagation at high-frequencies, classical GB constructions developed for the corresponding continuous models cannot be immediately applied. In turn, some adjustments in the GB ansatz are required, in order to compensate this lack of propagation and generate a family of quasi solutions that correctly approximates the finite difference wave dynamics. As we have showed in our main result Theorem \ref{thm:ApproxSolFD}, this can be done by introducing a correction term $\omega(t)$ in the ansatz's phase $\Phi$, to cope with the vanishing velocity of propagation of high-frequency discrete waves. The main contribution of this paper has therefore been to show that our proposed GB construction indeed produces approximated solutions for the finite difference wave equation \eqref{eq:mainEqFD}, whose energy propagates along the corresponding bi-characteristic rays, with specific approximation rates given in terms of the mesh parameter $h$. Furthermore, the numerical experiments of Section \ref{sec:numerics} have provided a confirmation of the validity and accuracy of our construction. Nevertheless, some key interesting issues have remained excluded by our analysis, and will be object of future research works.

\begin{itemize}
	\item[1.] \textbf{GB for semi-discrete wave equations on bounded domains} The constructions of Theorems \ref{thm:ApproxSol} and \ref{thm:ApproxSolFD} can be adapted to obtain highly localized solutions for the corresponding finite difference approximation of the Dirichlet problem
	\begin{align}\label{eq:mainEqDomain}
		\begin{cases}
			u_{tt}(\bx,t) - c\Delta u(\bx,t) = 0, & (\bx,t)\in \Omega\times (0,T)
			\\
			u(\bx,t) = 0, & (\bx,t)\in \partial\Omega\times (0,T)
			\\
			u(\bx,0) = u_0(\bx), \;\; u_t(\bx,0) = u_1(\bx), & \bx \in\Omega 
		\end{cases}
	\end{align}
	$\Omega\subset\RR^d$ being a bounded and regular domain. Obviously, since $\Omega$ is bounded, there may exist rays that exit this domain in finite time. So for an arbitrary $T > 0$ a GB beam will not satisfy in general the Dirichlet boundary condition. In order to overcome this difficulty, one has to superpose two GB beams - concentrated one on the positive branch $\bx_{fd}^+(t)$ of the bi-characteristics and the other on the negative branch $\bx_{fd}^-(t)$ - and include in the construction the Snell's law to handle the reflections at the boundary. An overview of this construction for the continuous wave equation \eqref{eq:mainEqDomain} can be found, for instance, in \cite[Proposition 8]{macia2002lack}. The extension of this constructions to the case of finite difference approximations of \eqref{eq:mainEqDomain} is, to the best of our knowledge, still missing and it would be an interesting complement to our analysis.
	
	\item[2.] \textbf{GB for semi-discrete wave equations on non-uniform meshes} Our GB construction has focused on the employment of a uniform mesh 
	\begin{align*}
		\mathcal G^h:= \Big\{\bx_\bj := \bj h,\,\bj\in\ZZ^d\,\Big\}
	\end{align*}
	for the space discretization. A natural extension of our work would then be to address the case of non-uniform meshes  
	\begin{align*}
		\mathcal G^h_g:= \Big\{g_\bj := g(\bx_\bj),\,x_\bj\in\mathcal G^h\,\Big\},
	\end{align*}
	obtained by transforming $\mathcal G_h$ through some suitable diffeomorphism $g: \RR^d\to\RR$. The introduction of a non-uniform mesh for the finite difference space semi-discretization of wave equations is a delicate issue that needs to be handled with a particular care. On the one hand, it has been shown in \cite{ervedoza2016numerical} that there are suitable choices of $\mathcal G_h$ that may help fixing the high-frequency pathologies of semi-discretized waves. On the other hand, a refinement of the mesh may introduce further pathological dynamics, with a series of unexpected propagation properties at high frequencies, that are not observed when employing a uniform mesh. For instance, in space dimension $d=1$, one can generate spurious solutions presenting the so-called \textit{internal reflection} phenomenon, meaning that the waves change direction as if they were hitting some fictitious boundary (see Figure \ref{fig:waveNonunif}).
	\begin{SCfigure}[1.5][h]
		\includegraphics[scale=0.075]{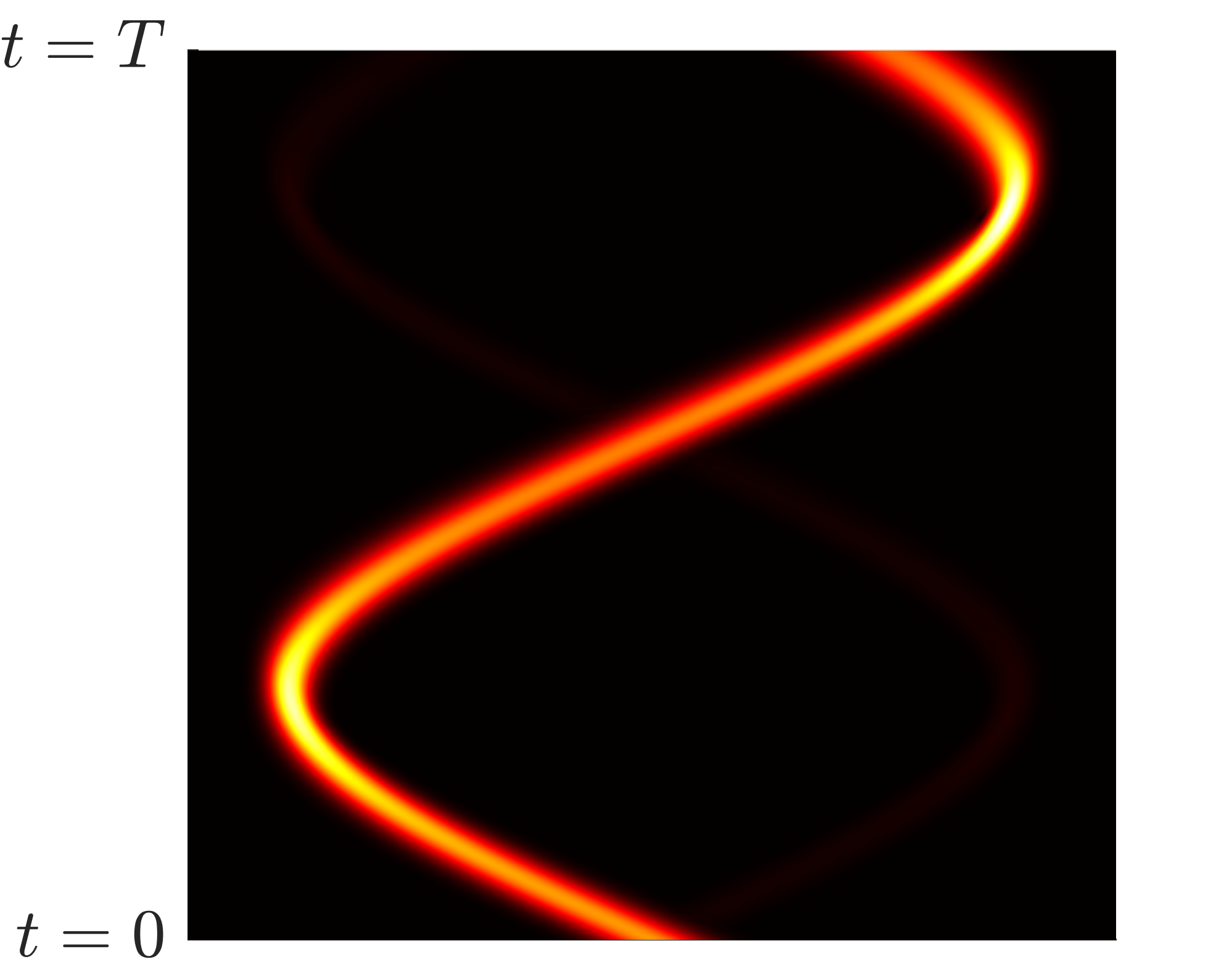}
		\caption{Numerical solution corresponding to the mesh $\mathcal G^h_g$ with $g(x) = \tan(\pi x/4)$ and initial frequency $\xi_0 = 7\pi/15$. The mesh is finer in the interior of the domain and coarser near the boundary, and generates the \textbf{internal reflection phenomenon}, which makes the wave changing direction even though no boundary is present in the discretization domain.}\label{fig:waveNonunif}
	\end{SCfigure}
	These effects are enhanced in the multi-dimensional case where the interaction and combination of such behaviors in the various space directions may produce, for instance, the \textit{rodeo effect}, i.e., waves that are trapped by the numerical grid in closed loops (see Figure \ref{fig:rodeo}).
	\begin{SCfigure}[1.5][h]
		\includegraphics[scale=0.15]{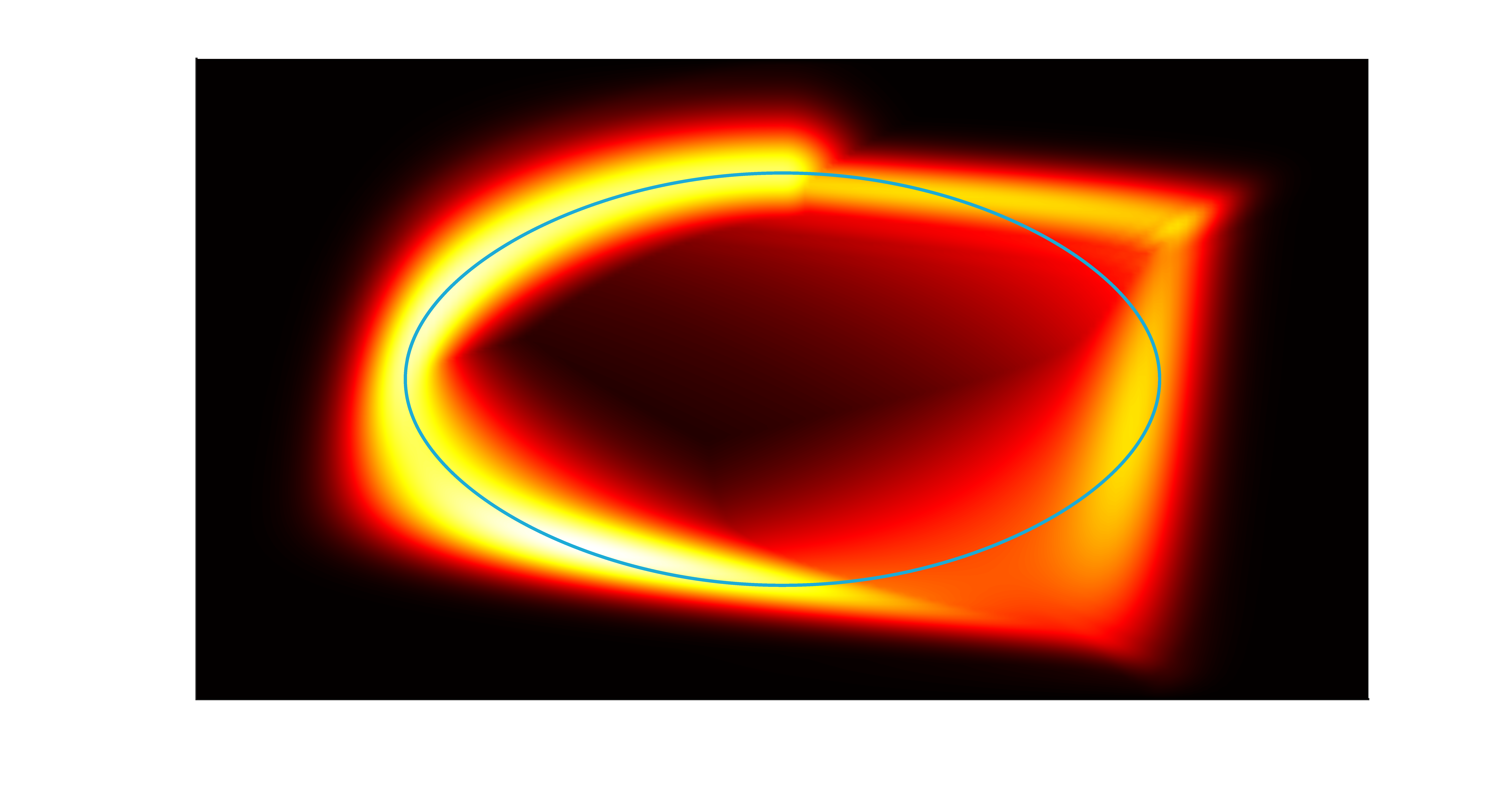}
		\caption{The rodeo effect in two-dimensional wave equations, semi-discretized in space on a non-uniform mesh. The blue line describes the characteristics, that propagate in time on a closed loop.}\label{fig:rodeo}
	\end{SCfigure}
	As usual, these new phenomena are due to changes in the Hamiltonian system providing the equations for the bi-characteristic rays, which in this case is generated by the principal symbol
	\begin{align*}
		\widetilde{\mathcal P}(\bx,t,\bxi,\tau) = -\tau^2 + \frac{4c}{g'(g^{-1}(\bx))}\sum_{i=1}^d\sin^2\left(\frac{\xi_i}{2}\right).
	\end{align*} 
	The GB ansatz should then be designed by adapting our construction in Section \ref{sec:GBdiscrete}, including in it the new equations of the bi-characteristic rays. Notice that this may require the introduction of some further correction term in the definition of the phase, in order to deal with those added pathological behaviors that do not appear when discretizing on a uniform mesh.
	
	\item[3.] \textbf{Alternative numerical schemes}  In this paper, we have considered only the case of finite difference approximations for the wave equation. On the other hand, it would be interesting to consider also other numerical schemes such as mixed finite elements, which have been introduced, for instance, in \cite{castro2006boundary,ervedoza2010observability,glowinski1989mixed} with control purposes. According to our analysis, the first step would be to identify the principal symbol associated with the FE discretization of the wave equation. From there, the study of the corresponding Hamiltonian system would again allow one to understand the propagation properties of the numerical solutions and, consequently, properly design a GB ansatz. 
	
	\item[4.] \textbf{GB for fully-discrete wave equations} Our GB construction of Section \ref{sec:GBdiscrete} has focused on the finite difference semi-discretization of the wave equation \eqref{eq:mainEqFD}, defined through the semi-discrete scheme 
	\begin{align*}
		\partial_t^2 u_j(t) = \frac{c}{h^2} \Big(u_{j+1} - 2u_j + u_{j-1}\Big).
	\end{align*}
	It would be worth to extend this analysis to the case of fully-discrete wave equations, approximated through 
	\begin{align}\label{eq:fullyDiscr}
		u_j^{n+1} - 2u_j^n + u_j^{n-1} = c\frac{h_t^2}{h^2} \Big(u_{j+1}^n - 2u_j^n + u_{j-1}^n\Big).
	\end{align}	
	When working at this fully discrete level, the analysis of wave propagation is quite more involved than the semi-discrete case, since now the bi-characteristic rays are generated by a principal symbol in the form 
	\begin{align*}
		\widehat{\mathcal P}(\bx,t,\bxi,\tau) = -\sin^2\left(\frac \tau2\right) + c\sum_{i=1}^d\sin^2\left(\frac{\xi_i}{2}\right),
	\end{align*} 
	with a trigonometric structure also in the time-frequency $\tau$. Some preliminary study of the dynamical behavior of these fully-discrete waves has been conducted in \cite{vichnevetsky1987wave}. Nevertheless, to the best of our knowledge, a GB construction for \eqref{eq:fullyDiscr} is still am open (and very interesting) problem that definitely deserves further investigation. 
\end{itemize}

\appendix{
\section{Proof of Theorem \ref{thm:ApproxSol}}\label{appendixCont}	

We give here the proof of Theorem \ref{thm:ApproxSol}, concerning the construction of a GB ansatz for the wave equation \eqref{eq:mainEqR}. To this end, we shall first need the following technical result.

\begin{proposition}\label{prop:Integral}
Let $\bx_0\in\RR^d$, $N\in\NN$ and $f\in L^\infty(\RR^d)$ be a function satisfying 
\begin{align}\label{eq:fCond}
	|\bx-\bx_0|^{-N} f(\bx)\in L^\infty(\RR^d).
\end{align}
Then, for any positive constant $0<\beta\in\RR$, we have
\begin{align}\label{eq:Integral}
	\int_{\RR^d} \left|f(\bx)e^{-k\beta |\bx-\bx_0|^2}\right|^2\,d\bx \leq \C k^{-\frac d2-N}
\end{align}
for some $\C = \C(d,N,\beta) > 0$ that does not depend on $k$.
\end{proposition}

\begin{proof}
Using \eqref{eq:fCond}, we have that there exists a function $g\in L^\infty(\RR^d)$ such that
\begin{align*}
	f(\bx) = |\bx-\bx_0|^N g(\bx).
\end{align*}
In view of this, we can apply the H\"older inequality to estimate
\begin{align*}
	\int_{\RR^d} \left|f(\bx)e^{-k\beta |\bx-\bx_0|^2}\right|^2\,d\bx & = \int_{\RR^d} \left||\bx-\bx_0|^Ng(\bx)e^{-k\beta |\bx-\bx_0|^2}\right|^2\,d\bx \leq \C(d) \int_{\RR^d} |\bx-\bx_0|^{2N} e^{-2k\beta |\bx-\bx_0|^2}\,d\bx. 
\end{align*}
Now, by means of the change of variable $\BF{y} = \sqrt{2k\beta}(\bx-\bx_0)$, we obtain that
\begin{align*}
	\int_{\RR^d} |\bx-\bx_0|^{2N} e^{-2k\beta |\bx-\bx_0|^2}\,d\bx = (2k\beta)^{-\frac d2-N} \int_{\RR^d} |\BF{y}|^{2N} e^{-|\BF{y}|^2}\,d\BF{y}.
\end{align*}
Finally, by employing polar coordinates, we can compute
\begin{align*}
	\int_{\RR^d} |\BF{y}|^{2N} e^{-|\BF{y}|^2}\,d\BF{y} &= \int_{\mathbb{S}^{d-1}} \int_0^{+\infty} r^{2N+d-1} e^{-r^2}\,dr d\sigma = \frac 12|\mathbb{S}^{d-1}| \int_0^{+\infty} \rho^{N+\frac d2-1} e^{-\rho}\,d\rho 
	\\
	&= \frac{d\pi^{\frac d2}}{2\Gamma\left(\frac d2+1\right)} \int_0^{+\infty} \rho^{N+\frac d2-1} e^{-\rho}\,d\rho = \frac{d\pi^{\frac d2}}{2}\frac{\Gamma\left(\frac d2+N\right)}{\Gamma\left(\frac d2+1\right)},
\end{align*}
where $\Gamma$ denotes the Euler gamma function. Putting everything together, we finally obtain that 
\begin{align*}
	\int_{\RR^d} \left|f(\bx)e^{-k\beta |\bx-\bx_0|^2}\right|^2\,d\bx \leq \C(d) \left[(2\beta)^{-\frac d2-N} \frac{d\pi^{\frac d2}}{2}\frac{\Gamma\left(\frac d2+N\right)}{\Gamma\left(\frac d2+1\right)}\right]k^{-\frac d2-N}.
\end{align*}

\end{proof}

\begin{proof}[Proof of Theorem \ref{thm:ApproxSol}]

We organize the proof in three steps, one for each statement of the theorem.

\medskip 
\noindent\textbf{Step 1: proof of \eqref{eq:approxSol}.} Starting from \eqref{eq:ansatz3}, we have that	
\begin{align*}
	\norm{\square_c u^k}{L^2(\RR^d)}^2 = \int_{\RR^d} |\square_c u^k|^2\,d\bx \leq k^{\frac d2-2}\int_{\RR^d} \left|e^{ik\phi}r_0\right|^2\,d\bx + k^{\frac d2} \int_{\RR^d} \left|e^{ik\phi}r_1\right|^2\,d\bx + k^{\frac d2+2}\int_{\RR^d}\left|e^{ik\phi}r_2\right|^2\,d\bx,
\end{align*}
where, we recall
\begin{align*}
	&r_0 = \square_c a
	\\
	&r_1 = a\square_c\phi + 2 a_t\phi_t - 2c \nabla a\cdot\nabla\phi
	\\
	&r_2 = \Big(c|\nabla\phi|^2-\phi_t^2\Big)a.  
\end{align*}

Since $a,\phi\in C^\infty(\RR^d\times\RR)$, we clearly have that also $r_0,r_1,r_2 \in C^\infty(\RR^d\times\RR)$. Moreover, by construction, $r_1$ and $r_2$ vanish on $\bx=\bx(t)$ up to the order $0$ and $2$. In view of that, we have that $r_0,r_1,r_2$ satisfy \eqref{eq:fCond} with $N = 0$, $N = 1$ and $N = 3$, respectively. Then, applying Proposition \ref{prop:Integral} with $\bx_0 = \bx(t)$ and the previous values of $N\in\NN$, we get
\begin{align*}
	k^{\frac d2-1}\int_{\RR^d} \left|e^{ik\phi}r_0\right|^2\,d\bx \leq \C(a,\phi) k^{-2}
	\\
	k^{\frac d2} \int_{\RR^d} \left|e^{ik\phi}r_1\right|^2\,d\bx \leq \C(a,\phi) k^{-1}
	\\
	k^{\frac d2+1}\int_\RR\left|e^{ik\phi}r_2\right|^2\,d\bx \leq \C(a,\phi) k^{-1}.
\end{align*}
Putting everything together, since $k\geq 1$, we finally obtain that 
\begin{align*}
	\norm{\square_c u^k}{L^2(\RR^d)}^2 \leq \C\Big(k^{-2} +  k^{-1} +  k^{-1}\Big) \leq \C(a,\phi) k^{-1},
\end{align*}
that is,
\begin{align*}
	\norm{\square_c u^k}{L^2(\RR^d)} \leq \C(a,\phi) k^{-\frac 12}.
\end{align*}

\medskip 
\noindent\textbf{Step 2: proof of \eqref{eq:approxEnergy}.} Starting from \eqref{eq:energy}, we have that	
\begin{align*}
	E_c(u^k(\cdot,t)) =&\, \frac 12 \int_{\RR^d} \Big(|u^k_t(\cdot,t)|^2 + c|\nabla u^k(\cdot,t)|^2\Big)\,d\bx = \Xi_0^k(t) + \Xi_1^k(t) + \Xi_2^k(t),
\end{align*}
where we have denoted 
\begin{align*}
	& \Xi_0^k(t) := \frac{k^{\frac d2}}{2} \int_{\RR^d} |a|^2\Big(|\phi_t|^2 + c|\nabla\phi|^2\Big)e^{-k(\bx-\bx(t))\cdot\big[\Im (M_0)(\bx-\bx(t))\big]}\,d\bx  
	\\
	& \Xi_1^k(t) := k^{\frac d2-1} \int_{\RR^d} a\Big(a_t\phi_t + c\nabla a\cdot\nabla\phi\Big)e^{-k(\bx-\bx(t))\cdot\big[\Im (M_0)(\bx-\bx(t))\big]}\,d\bx
	\\
	& \Xi_2^k(t) := \frac{k^{\frac d2-2}}{2} \int_{\RR^d} \Big(|a_t|^2 + c|\nabla a|^2\Big)e^{-k(\bx-\bx(t))\cdot\big[\Im (M_0)(\bx-\bx(t))\big]}\,d\bx.
\end{align*}
Notice that, since $a,\phi \in C^\infty(\RR^d\times\RR)$ and $\Im(M_0)>0$, we have that for all $t\in (0,T)$
\begin{align*}
	&|\Xi_1^k(t)| \leq \C(a,\phi)k^{\frac d2-1} \int_{\RR^d} e^{-k(\bx-\bx(t))\cdot\big[\Im (M_0)(\bx-\bx(t))\big]}\,d\bx = \C(a,\phi)\frac{d\Gamma\left(\frac{d+1}{2}\right)}{2\Gamma\left(\frac d2 +1\right)}\left(\frac{\pi}{\text{det}(\Im(M_0))}\right)^{\frac d2} k^{-1}
	\\
	&|\Xi_2^k(t)| \leq \C(a,\phi)k^{\frac d2-2} \int_{\RR^d} e^{-k(\bx-\bx(t))\cdot\big[\Im (M_0)(\bx-\bx(t))\big]}\,d\bx = \C(a,\phi)\frac{d\Gamma\left(\frac{d+1}{2}\right)}{2\Gamma\left(\frac d2 +1\right)}\left(\frac{\pi}{\text{det}(\Im(M_0))}\right)^{\frac d2} k^{-2}
\end{align*}
Hence,
\begin{align}\label{eq:Xi23lim}
	\sup_{t\in(0,T)} \Big(|\Xi_1^k(t)| + |\Xi_2^k(t)|\Big) \to 0, \quad\text{ as } k\to +\infty.	
\end{align}

As for the term $\Xi_0^k(t)$, replacing in it the explicit expression \eqref{eq:aThm} of the amplitude function $a$, we get that 
\begin{align}\label{eq:Xi1}
	\Xi_0^k(t) &= \frac{k^{\frac d2}}{2} \int_{\RR^d} \Big(|\phi_t|^2 + c|\nabla\phi|^2\Big)e^{-(\bx-\bx(t))\cdot\Big[\Big(2 I_d + k\Im (M_0)\Big)(\bx-\bx(t))\Big]}\,d\bx,  
\end{align}
where $I_d$ denotes the identity matrix in dimension $d\times d$. Moreover, using \eqref{eq:ray} and the explicit expression \eqref{eq:phiThm} of the phase function $a$, we can compute
\begin{align*}
	&|\phi_t|^2 = \frac{c}{|\bxi_0|^2}\bigg[|\bxi_0|^2 + \bxi_0\cdot\Big(M_0(\bx-\bx(t))\Big)\bigg]^2
	\\
	&c|\nabla\phi|^2 = c\,\Big|\bxi_0 + M_0(\bx-\bx(t))\Big|^2
\end{align*}
and we obtain from \eqref{eq:Xi1} that
\begin{align*}
	\Xi_0^k(t) =&\; \frac{c}{2|\bxi_0|^2} k^{\frac d2} \int_{\RR^d} \bigg[|\bxi_0|^2 + \bxi_0\cdot\Big(M_0(\bx-\bx(t))\Big)\bigg]^2	
	e^{-(\bx-\bx(t))\cdot\Big[\Big(2 I_d + k\Im (M_0)\Big)(\bx-\bx(t))\Big]}\,d\bx  
	\\
	&+ \frac c2 k^{\frac d2} \int_{\RR^d} \Big|\bxi_0 + M_0(\bx-\bx(t))\Big|^2e^{-(\bx-\bx(t))\cdot\Big[\Big(2 I_d + k\Im (M_0)\Big)(\bx-\bx(t))\Big]}\,d\bx.
\end{align*}
Now, since $2 I_d + k\Im (M_0)>0$, we can apply the change of variables
\begin{align*}
	\Big(2 I_d + k\Im (M_0)\Big)^{\frac 12} (\bx-\bx(t)) = \BF{y}
\end{align*}
and we get 
\begin{align*}
	\Xi_0^k(t) =&\; \frac{c}{2|\bxi_0|^2} \frac{k^{\frac d2}}{\text{det}\Big(2 I_d + k\Im (M_0)\Big)^{\frac d2}} \int_{\RR^d} \Bigg[|\bxi_0|^2 + \bxi_0\cdot\bigg(\Big(2 I_d + k\Im (M_0)\Big)^{-\frac 12}M_0\BF{y}\bigg)\Bigg]^2 e^{-|\BF{y}|^2}\,d\BF{y}  
	\\
	&+ \frac c2 \frac{k^{\frac d2}}{\text{det}\Big(2 I_d + k\Im (M_0)\Big)^{\frac d2}} \int_{\RR^d} \Big|\bxi_0 + \Big(2 I_d + k\Im (M_0)\Big)^{-\frac 12}M_0\BF{y}\Big|^2e^{-|\BF{y}|^2}\,d\BF{y}.
\end{align*}

Finally, the two integrals in the expression above can be computed by employing polar coordinates. In this way, we obtain that
\begin{align*}
	\Xi_0^k(t) = \C\Big(d,\bxi_0,M_0\Big)\pi^{\frac d2}\frac{k^{\frac d2}\Big(1+k^{-1}\Big)}{\text{det}\Big(2 I_d + k\Im (M_0)\Big)^{\frac d2}} = \frac{\C\Big(d,\bxi_0,M_0\Big)\pi^{\frac d2}}{\text{det}\Big(2 k^{-1} I_d + \Im (M_0)\Big)^{\frac d2}}\Big(1+k^{-1}\Big).
\end{align*}
Hence,
\begin{align}\label{eq:Xi1lim}
	\lim_{k\to +\infty} \Xi_0^k(t) = \C\Big(d,\bxi_0,M_0\Big)\left(\frac{\pi}{\text{det}\big(\Im (M_0)\big)}\right)^{\frac d2}.
\end{align}
From \eqref{eq:Xi23lim} and \eqref{eq:Xi1lim}, we immediately get \eqref{eq:approxEnergy}.

\medskip 
\noindent\textbf{Step 3: proof of \eqref{eq:approxRay}.} Since $a,\phi \in C^\infty(\RR^d\times\RR)$ and $k\geq 1$, we have that	
\begin{align*}
	\int_{\RR^d\setminus B_k(t)} \Big(|u^k_t(\cdot,t)|^2 + c|\nabla u^k(\cdot,t)|^2\Big)\,d\bx & \leq \C\left(k^{\frac d2} + k^{\frac d2-1} + k^{\frac d2-2}\right) \int_{\RR^d\setminus B_k(t)} e^{-k(\bx-\bx(t))\cdot\big[\Im (M_0)(\bx-\bx(t))\big]}\,d\bx
	\\
	& \leq \C k^{\frac d2} \int_{\RR^d\setminus B_k(t)} e^{-k(\bx-\bx(t))\cdot\big[\Im (M_0)(\bx-\bx(t))\big]}\,d\bx
	\\
	&= \C k^{\frac d2} \int_{\RR^d\setminus B\left(0,k^{-\frac 14}\right)} e^{-k\BF{y}\cdot\big[\Im (M_0)\BF{y}\big]}\,d\BF{y},
\end{align*}
with $\C = \C(a,\phi)>0$ a positive constant not depending on $k$. Moreover, by employing the change of variable $\BF{y} = k^{-\frac 12}\BF{z}$, we have that
\begin{align*}
	\C k^{\frac d2} \int_{\RR^d\setminus B\left(0,k^{-\frac 14}\right)} e^{-k\BF{y}\cdot\big[\Im (M_0)\BF{y}\big]}\,d\BF{y} &= \C\int_{\RR^d\setminus B\left(0,k^{\frac 14}\right)} e^{-\BF{z}\cdot\big[\Im (M_0)\BF{z}\big]}\,d\BF{z} 
	\\
	&= \C\int_{\RR^d\setminus B\left(0,k^{\frac 14}\right)} e^{-\frac 12\BF{z}\cdot\big[\Im (M_0)\BF{z}\big]}e^{-\frac 12\BF{z}\cdot\big[\Im (M_0)\BF{z}\big]}\,d\BF{z}	
	\\
	&\leq \sup_{\RR^d\setminus B\left(0,k^{\frac 14}\right)}\left(e^{-\frac 12\BF{z}\cdot\big[\Im (M_0)\BF{z}\big]}\right) \int_{\RR^d\setminus B\left(0,k^{\frac 14}\right)} e^{-\frac 12\BF{z}\cdot\big[\Im (M_0)\BF{z}\big]}\,d\BF{z}
	\\
	&= e^{-\frac 12\text{det}\big(\Im (M_0)\big)k^{\frac 12}} \int_{\RR^d\setminus B\left(0,k^{\frac 14}\right)} e^{-\frac 12\BF{z}\cdot\big[\Im (M_0)\BF{z}\big]}\,d\BF{z}
	\\
	&\leq e^{-\frac 12\text{det}\big(\Im (M_0)\big)k^{\frac 12}} \int_{\RR^d} e^{-\frac 12\BF{z}\cdot\big[\Im (M_0)\BF{z}\big]}\,d\BF{z} 
	\\
	&= \frac{d\Gamma\left(\frac{d+1}{2}\right)}{2\Gamma\left(\frac d2 +1\right)}\left(\frac{2\pi}{\text{det}(\Im(M_0))}\right)^{\frac d2}e^{-\frac 12\text{det}\big(\Im (M_0)\big)k^{\frac 12}}.
\end{align*}
From this last estimate, \eqref{eq:approxRay} follows immediately.
\end{proof}

\section{Proof of Theorem \ref{thm:ApproxSolFD}}\label{appendixFD}	

\subsection{Concentration of solutions}

We give here the proof of Theorem \ref{thm:ApproxSolFD}, concerning the construction of a GB ansatz for the semi-discrete wave equation \eqref{eq:mainEqFD}. 

\begin{proof}[Proof of Theorem \ref{thm:ApproxSolFD}]
	
We are going to split the proof into three steps, one for each point in the statement of the theorem. Moreover, in what follows, we will denote by $\C>0$ a generic positive constant independent of $h$. This constant may change even from line to line.
	
\medskip
\noindent\textbf{Step 1: proof of}\eqref{eq:approxSolFD}. Starting from \eqref{eq:dAlambertianFD}, we have that	
\begin{align*}
	\norm{\square_{c,h} u^h_{fd}(\cdot,t)}{\ell^2(h\ZZ)}^2 =&\; h\sum_{j\in\ZZ} \left|\square_{c,h} u^h_{fd,j}(t)\right|^2
	\\
	\leq&\; h^{\frac 52}\sum_{j\in\ZZ} \left|e^{\frac ih\Phi_j}\left(\mathcal R_0A_j + A_j \frac{\mathcal R_2-\mathcal R_{fd}}{h^2}\right)\right|^2 
	\\
	&+ h^{\frac 12}\sum_{j\in\ZZ}\left(\left|e^{\frac ih\Phi_j}(\mathcal R_1A_j - \RT_1 A_j)\right|^2 + \left|e^{\frac ih\Phi_j}\RT_1 A_j\right|^2\right)
	\\
	&+ h^{-\frac 32}\sum_{j\in\ZZ} \left|e^{\frac ih\Phi_j}A_j\mathcal R_{fd}\right|^2.
\end{align*}
Moreover, using \eqref{eq:phiFD}, we see that
\begin{align*}
	e^{\frac ih\Phi_j} = e^{\frac ih \Re(\Phi_j)}e^{-\frac 1h \Im(\Phi_j)} = e^{\frac ih \big(\omega(t)+\xi_0(x_j-x_{fd,j}(t)) + \frac 12\Re (M(t))(x_j-x_{fd,j}(t))^2\big)}e^{-\frac{1}{2h}\Im (M(t))(x_j-x_{fd,j}(t))^2}
\end{align*}
and, therefore,
\begin{align}\label{eq:dicreteGaussian}
	\left|e^{\frac ih\Phi_j}\right|^2 = e^{-\frac{1}{2h}\Im (M(t))(x_j-x_{fd,j}(t))^2}.
\end{align}
In view of this, we get
\begin{align*}
	\norm{\square_{c,h} u^h_{fd}(\cdot,t)}{\ell^2(h\ZZ)}^2 \leq& \; h^{\frac 52}\sum_{j\in\ZZ} e^{-\frac 1h \Im(M(t))(x_j-x_{fd,j}(t))^2}\left|\mathcal R_0A_j + A_j \frac{\mathcal R_2-\mathcal R_{fd}}{h^2}\right|^2 
	\\
	&+ h^{\frac 12}\sum_{j\in\ZZ} e^{-\frac 1h \Im(M(t))(x_j-x_{fd,j}(t))^2}\left(\left|\mathcal R_1A_j - \RT_1 A_j\right|^2 + \left|\RT_1 A_j\right|^2 \right)
	\\
	&+ h^{-\frac 32}\sum_{j\in\ZZ} e^{-\frac 1h \Im(M(t))(x_j-x_{fd,j}(t))^2}\left|A_j\mathcal R_{fd}\right|^2.
\end{align*}
Now, by means of \eqref{eq:RError} and \eqref{eq:R1tildeError}, we have that 
\begin{align*}
	h^{\frac 52} & \sum_{j\in\ZZ} e^{-\frac 1h \Im(M(t))(x_j-x_{fd,j}(t))^2}\left|\mathcal R_0A_j + A_j \frac{\mathcal R-\mathcal R_{fd}}{h^2}\right|^2 
	\\
	&\leq h^{\frac 52}\sum_{j\in\ZZ} e^{-\frac 1h \Im(M(t))(x_j-x_{fd,j}(t))^2}\Big(\left|\mathcal R_0A_j\right|^2 + \C \left|A_j\right|^2\Big) 
\end{align*}
and
\begin{align*}
	h^{\frac 12} &\sum_{j\in\ZZ} e^{-\frac 1h \Im(M(t))(x_j-x_{fd,j}(t))^2}\left|\mathcal R_1A_j - \RT_1 A_j\right|^2 
	\\
	&\leq \C h^{\frac 52}\sum_{j\in\ZZ} e^{-\frac 1h \Im(M(t))(x_j-x_{fd,j}(t))^2}\Big(|\Dhc\Phi_j|^2 + |\Dhc\Phi_j|^4\Big) +\C h^{\frac 92}\sum_{j\in\ZZ} e^{-\frac 1h \Im(M(t))(x_j-x_{fd,j}(t))^2}.
\end{align*}
Hence, 
\begin{align*}
	\norm{\square_{c,h} u^h_{fd}(\cdot,t)}{\ell^2(h\ZZ)}^2 \leq& \; \C h^{\frac 92}\sum_{j\in\ZZ} e^{-\frac 1h \Im(M(t))(x_j-x_{fd,j}(t))^2} +\C h^{\frac 52}\sum_{j\in\ZZ} e^{-\frac 1h \Im(M(t))(x_j-x_{fd,j}(t))^2}
	\\
	&+ h^{\frac 12}\sum_{j\in\ZZ} e^{-\frac 1h \Im(M(t))(x_j-x_{fd,j}(t))^2}\left|\RT_1 A_j\right|^2 
	\\
	&+ h^{-\frac 32}\sum_{j\in\ZZ} e^{-\frac 1h \Im(M(t))(x_j-x_{fd,j}(t))^2}\left|A_j\mathcal R_{fd}\right|^2.
\end{align*}
The first two terms on the right-hand side of the above inequality can be estimated by observing that
\begin{align}\label{eq:GaussianSum1}
	\sum_{j\in\ZZ} e^{-\frac 1h \Im(M(t))(x_j-x_{fd,j}(t))^2} = \mathcal O(h^{-\frac 12}).
\end{align}
This can be easily seen by considering the Riemann sum approximating on the mesh $\mathcal G_h$ the integral 
\begin{align*}
	\int_\RR e^{-\frac 1h \Im(M(t))(x-x_{fd}(t))^2}\,dx = \C h^{\frac 12}.
\end{align*}
Hence, using \eqref{eq:GaussianSum1}, we obtain that 
\begin{align*}
	\Big(h^{\frac 92} + h^{\frac 52}\Big) \sum_{j\in\ZZ} e^{-\frac 1h \Im(M(t))(x_j-x_{fd,j}(t))^2} = \mathcal O(h^4) + \mathcal O(h^2) = \mathcal O(h^2)		
\end{align*}
and, therefore,
\begin{align*}
	\norm{\square_{c,h} u^h_{fd}(\cdot,t)}{\ell^2(h\ZZ)}^2 \leq& \; h^{\frac 12}\sum_{j\in\ZZ} e^{-\frac 1h \Im(M(t))(x_j-x_{fd,j}(t))^2}\left|\RT_1 A_j\right|^2 
	\\
	&+ h^{-\frac 32}\sum_{j\in\ZZ} e^{-\frac 1h \Im(M(t))(x_j-x_{fd,j}(t))^2}\left|A_j\mathcal R_{fd}\right|^2 + \mathcal O(h^2).
\end{align*}
	
Finally, since by construction $\RT_1 A_j$ and $\mathcal R_{fd}$ vanish on the discrete characteristics $x_{fd}(t)$ up to the order $0$ and $2$ respectively, we have that 
\begin{align*}
	&\sum_{j\in\ZZ} e^{-\frac 1h \Im(M(t))(x_j-x_{fd,j}(t))^2}\left|\RT_1 A_j\right|^2 = \mathcal O(h^{\frac 12})
	\\
	&\sum_{j\in\ZZ} e^{-\frac 1h \Im(M(t))(x_j-x_{fd,j}(t))^2}\left|A_j\mathcal R_{fd}\right|^2 = \mathcal O(h^{\frac 52}).
\end{align*}
	
This is the discrete version of Proposition \ref{prop:Integral}, which can be proven once again by applying Riemann sum to approximate the corresponding integrals. In view of this, we obtain  
\begin{align*}
	h^{\frac 12}\sum_{j\in\ZZ} e^{-\frac 1h \Im(M(t))(x_j-x_{fd,j}(t))^2}\left|\RT_1 A_j\right|^2 + h^{-\frac 32}\sum_{j\in\ZZ} e^{-\frac 1h \Im(M(t))(x_j-x_{fd,j}(t))^2}\left|A_j\mathcal R_{fd}\right|^2 = \mathcal O(h).
\end{align*}
and we can finally conclude that
\begin{align*}
	\norm{\square_{c,h} u^h_{fd}(\cdot,t)}{\ell^2(h\ZZ)}^2 = \mathcal O(h^2) + \mathcal O(h) = \mathcal O(h), \quad\text{ as } h\to 0^+.
\end{align*}
	
\medskip
\noindent\textbf{Step 2: proof of}\eqref{eq:approxEnergyFD}. First of all, starting from \eqref{eq:energyFD}, \eqref{eq:discreteSob} and \eqref{eq:ansatzFD_h}, we can write 
\begin{align*}
	\mathcal E_h[u_{fd}^h] = \frac h2\sum_{j\in\ZZ} \Big(|\partial_t u_{fd,j}^h|^2 + c|\fwd u_{fd,j}^h|^2\Big) = \frac{h^{\frac 52}}{2}\sum_{j\in\ZZ} \left(\left|\partial_t \left(A_j e^{\frac ih \Phi_j}\right)\right|^2 + c\left|\fwd \left(A_j e^{\frac ih \Phi_j}\right)\right|^2\right). 
\end{align*}
Moreover, we can easily compute
\begin{align*}
	&\partial_t \left(A_j e^{\frac ih \Phi_j}\right) = \frac 1h e^{\frac ih \Phi_j}\Big(h\partial_t A_j + iA_j\partial_t\Phi_j\Big) 
	\\[10pt]
	&\fwd \left(A_j e^{\frac ih \Phi_j}\right) = \frac 1h e^{\frac ih \Phi_j}\Big(h\fwd A_je^{i\fwd\Phi_j} + \big(e^{i\fwd\Phi_j}-1\big)A_j\Big) 
\end{align*}
Using these identities and \eqref{eq:dicreteGaussian}, we then obtain
\begin{align*}
	\mathcal E_h[u_{fd}^h] =&\; \frac{h^{\frac 52}}{2}\sum_{j\in\ZZ} e^{-\frac 1h \Im(M(t))(x_j-x_{fd,j}(t))^2} \mathcal S_1 + h^{\frac 32}\sum_{j\in\ZZ} e^{-\frac 1h \Im(M(t))(x_j-x_{fd,j}(t))^2} \mathcal S_2 
	\\
	&+ \frac{h^{\frac 12}}{2}\sum_{j\in\ZZ} e^{-\frac 1h \Im(M(t))(x_j-x_{fd,j}(t))^2} \mathcal S_3,
\end{align*}
where we have denoted
\begin{subequations}
	\begin{align}
		&\mathcal S_1 := \Big|\partial_t A_j\Big|^2 + c\Big|\fwd A_je^{i\fwd\Phi_j}\Big|^2 \label{eq:S1}
		\\
		&\mathcal S_2 := \Big|iA_j\partial_t A_j\partial_t\Phi_j\Big| + c\Big|\fwd A_je^{i\fwd\Phi_j}\big(e^{i\fwd\Phi_j}-1\big)A_j\Big| \label{eq:S2}
		\\	
		&\mathcal S_3 := \Big|iA_j\partial_t\Phi_j\Big|^2 + c\Big|\big(e^{i\fwd\Phi_j}-1\big)A_j\Big|^2. \label{eq:S3}
	\end{align}
\end{subequations}
	
Moreover, by construction of $A$ and $\Phi$, we have that $|\mathcal S_i|\leq \C(A,\phi)$ for $i\in\{1,2,3\}$. Using this and \eqref{eq:GaussianSum1}, we finally obtain
\begin{align*}
	\mathcal E_h[u_{fd}^h] = \mathcal O(h^2) + \mathcal O(h) + \mathcal O(1) = \mathcal O(1), \quad\text{ as } h\to 0^+.
\end{align*}
	
\medskip
\noindent\textbf{Step 3: proof of}\eqref{eq:approxRayFD}. Repeating the computations of Step 2, we have that
\begin{align*}
	\frac h2\sum_{j\in\ZZ^\dagger(t)} \Big(|\partial_t u_{fd,j}^h|^2 + c|\fwd u_{fd,j}^h|^2\Big) =&\; \frac{h^{\frac 52}}{2}\sum_{j\in\ZZ^\dagger(t)} e^{-\frac 1h \Im(M(t))(x_j-x_{fd,j}(t))^2} \mathcal S_1
	\\
	&+ h^{\frac 32}\sum_{j\in\ZZ^\dagger(t)} e^{-\frac 1h \Im(M(t))(x_j-x_{fd,j}(t))^2} \mathcal S_2
	\\
	&+ \frac{h^{\frac 12}}{2}\sum_{j\in\ZZ^\dagger(t)} e^{-\frac 1h \Im(M(t))(x_j-x_{fd,j}(t))^2} \mathcal S_3,
\end{align*}
with $\mathcal S_1$, $\mathcal S_2$ and $\mathcal S_3$ defined in \eqref{eq:S1}, \eqref{eq:S2} and \eqref{eq:S3}, respectively. Hence, recalling that $|\mathcal S_i|\leq \C(A,\Phi)$ for $i\in\{1,2,3\}$, we obtain
\begin{align}\label{eq:estPrel}
	\frac h2\sum_{j\in\ZZ^\dagger(t)} \Big(|\partial_t u_{fd,j}^h|^2 + c|\fwd u_{fd,j}^h|^2\Big) \leq \C \Big(h^{\frac 52} + h^{\frac 32} + h^{\frac 12}\Big)\sum_{j\in\ZZ^\dagger(t)} e^{-\frac 1h \Im(M(t))(x_j-x_{fd,j}(t))^2}.
\end{align}
Now, employing the transformation 
\begin{align*}
	x_j-x_{fd,j} = y_j,
\end{align*}
and denoting 
\begin{align*}
	\ZZ^\ddagger(t) := \Big\{ j\in\ZZ\,:\, |y_j|>h^{\frac 14}\Big\},
\end{align*}
we have that 
\begin{align*}
	\sum_{j\in\ZZ^\dagger(t)} e^{-\frac 1h \Im(M(t))(x_j-x_{fd,j}(t))^2} &= \sum_{j\in\ZZ^\ddagger(t)} e^{-\frac 1h \Im(M(t))y_j^2} \leq e^{-\frac 12 \Im(M(t))h^{-\frac 12}} \sum_{j\in\ZZ^\ddagger(t)} e^{-\frac{1}{2h} \Im(M(t))y_j^2}
	\\
	&\leq e^{-\frac 12 \Im(M(t))h^{-\frac 12}} \sum_{j\in\ZZ} e^{-\frac{1}{2h} \Im(M(t))y_j^2} = \C h^{-\frac 12} e^{-\frac 12 \Im(M(t))h^{-\frac 12}}.
\end{align*}
Substituting this in \eqref{eq:estPrel}, we finally conclude that
\begin{align*}
	\frac h2\sum_{j\in\ZZ^\dagger(t)} \Big(|\partial_t u_{fd,j}^h|^2 + c|\fwd u_{fd,j}^h|^2\Big) \leq \C \Big(1 + h + h^2\Big) e^{-\frac 12 \Im(M(t))h^{-\frac 12}}.
\end{align*}
	
\end{proof}

}

\section*{Acknowledgments} The authors wish to acknowledge Dr. Konstantin Zerulla (Friedrich-Alexander-Universit\"at Erlangen-N\"urnberg, Germany) for his careful revision and precious comments on early versions of this work.

\bibliographystyle{acm}
\bibliography{biblio}

\begin{thebibliography}{10}

\bibitem{arnaud1973hamiltonian}
{\sc Arnaud, J.~A.}
\newblock Hamiltonian theory of beam mode propagation.
\newblock In {\em Progress in Optics}, vol.~11. Elsevier, 1973, pp.~247--304.

\bibitem{babich1999higher}
{\sc Babich, V.}
\newblock The higher-dimensional {WKB} method or ray method. its analogues and
  generalizations.
\newblock In {\em Partial Differential Equations V}. Springer, 1999,
  pp.~91--131.

\bibitem{babich1972asymptotic}
{\sc Babich, V.~M., and Buldyrev, V.~S.}
\newblock {\em Short-wavelength diffraction theory: asymptotic methods}.
\newblock Nauka, Moscow, 1972.

\bibitem{bardos1992sharp}
{\sc Bardos, C., Lebeau, G., and Rauch, J.}
\newblock Sharp sufficient conditions for the observation, control, and
  stabilization of waves from the boundary.
\newblock {\em SIAM J. Control Optim. 30}, 5 (1992), 1024--1065.

\bibitem{biccari2020propagation}
{\sc Biccari, U., Marica, A., and Zuazua, E.}
\newblock Propagation of one- and two-dimensional discrete waves under finite
  difference approximation.
\newblock {\em Found. Comput. Math. 20}, 6 (2020), 1401--1438.

\bibitem{brillouin1926mecanique}
{\sc Brillouin, L.}
\newblock La m{\'e}canique ondulatoire de {S}chr{\"o}dinger; une m{\'e}thode
  g{\'e}n{\'e}rale de r{\'e}solution par approximations successives.
\newblock {\em Compt. Rend. Acad. Sci 183}, 11 (1926), 24--26.

\bibitem{burq1997condition}
{\sc Burq, N., and G{\'e}rard, P.}
\newblock Condition n{\'e}cessaire et suffisante pour la contr{\^o}labilit{\'e}
  exacte des ondes.
\newblock {\em C.R. Acad. Sci. Paris S{\'e}r. I 325}, 7 (1997), 749--752.

\bibitem{burq1998controle}
{\sc Burq, N., and Schlenker, J.-M.}
\newblock Contr{\^o}le de l'{\'e}quation des ondes dans des ouverts comportant
  des coins.
\newblock {\em Bull. Soc. Math. France 126}, 4 (1998), 601.

\bibitem{castro2006boundary}
{\sc Castro, C., and Micu, S.}
\newblock Boundary controllability of a linear semi-discrete 1d wave equation
  derived from a mixed finite element method.
\newblock {\em Numer. Math. 102}, 3 (2006), 413--462.

\bibitem{cerveny1982computation}
{\sc {\v{C}}erven{\`y}, V., Popov, M.~M., and P{\v{s}}en{\v{c}}{\'\i}k, I.}
\newblock Computation of wave fields in inhomogeneous media - gaussian beam
  approach.
\newblock {\em Geophys. J. Int. 70}, 1 (1982), 109--128.

\bibitem{engquist2003computational}
{\sc Engquist, B., and Runborg, O.}
\newblock Computational high frequency wave propagation.
\newblock {\em Acta Num. 12\/} (2003), 181--266.

\bibitem{ervedoza2010observability}
{\sc Ervedoza, S.}
\newblock Observability properties of a semi-discrete 1d wave equation derived
  from a mixed finite element method on nonuniform meshes.
\newblock {\em ESAIM: Control Optim. Calc. Var. 16}, 2 (2010), 298--326.

\bibitem{ervedoza2016numerical}
{\sc Ervedoza, S., Marica, A., and Zuazua, E.}
\newblock Numerical meshes ensuring uniform observability of one-dimensional
  waves: construction and analysis.
\newblock {\em IMA J. Numer. Anal. 36}, 2 (2016), 503--542.

\bibitem{gerard1991microlocal}
{\sc G{\'e}rard, P.}
\newblock Microlocal defect measures.
\newblock {\em Comm. Partial Differential Equations 16}, 11 (1991), 1761--1794.

\bibitem{glowinski1989mixed}
{\sc Glowinski, R., Kinton, W., and Wheeler, M.~F.}
\newblock A mixed finite element formulation for the boundary controllability
  of the wave equation.
\newblock {\em Internat. J.Numer. Methods Engrg 27}, 3 (1989), 623--635.

\bibitem{hill2001prestack}
{\sc Hill, N.~R.}
\newblock Prestack {G}aussian-beam depth migration.
\newblock {\em Geophys. 66}, 4 (2001), 1240--1250.

\bibitem{hormander1971existence}
{\sc H{\"o}rmander, L.~V.}
\newblock On the existence and the regularity of solutions of linear
  pseudodifferential equations.
\newblock {\em Enseign. Math. 17\/} (1971), 99--163.

\bibitem{jin2008gaussian}
{\sc Jin, S., Wu, H., and Yang, X.}
\newblock Gaussian beam methods for the {S}chr\"odinger equation in the
  semi-classical regime: {L}agrangian and {E}ulerian formulations.
\newblock {\em Commun. Math. Sci. 6}, 4 (2008), 995--1020.

\bibitem{keller1962geometrical}
{\sc Keller, J.~B.}
\newblock Geometrical theory of diffraction.
\newblock {\em J. Opt. Soc. Amer. 52}, 2 (1962), 116--130.

\bibitem{leung2009eulerian}
{\sc Leung, S., and Qian, J.}
\newblock Eulerian {G}aussian beams for {S}chr{\"o}dinger equations in the
  semi-classical regime.
\newblock {\em J. Comput. Phys. 228}, 8 (2009), 2951--2977.

\bibitem{lions1993mesures}
{\sc Lions, P.-L., and Paul, T.}
\newblock Sur les mesures de {W}igner.
\newblock {\em Rev. Mat. Iberoamericana 9}, 3 (1993), 553--618.

\bibitem{liu2010recovery}
{\sc Liu, H., and Ralston, J.}
\newblock Recovery of high frequency wave fields from phase space-based
  measurements.
\newblock {\em Multiscale Model. Simul. 8}, 2 (2010), 622--644.

\bibitem{liu2014gaussian}
{\sc Liu, H., Ralston, J., Runborg, O., and Tanushev, N.~M.}
\newblock Gaussian beam methods for the {H}elmholtz equation.
\newblock {\em SIAM J. Appl. Math. 74}, 3 (2014), 771--793.

\bibitem{liu2013error}
{\sc Liu, H., Runborg, O., and Tanushev, N.}
\newblock Error estimates for {G}aussian beam superpositions.
\newblock {\em Math. Comput. 82}, 282 (2013), 919--952.

\bibitem{liu2016sobolev}
{\sc Liu, H., Runborg, O., and Tanushev, N.~M.}
\newblock Sobolev and max norm error estimates for {G}aussian beam
  superpositions.
\newblock {\em Commun. Math. Sci. 14}, 7 (2016), 2037--2072.

\bibitem{macia2002propagacion}
{\sc Maci{\`a}, F.}
\newblock {\em Propagaci{\'o}n y control de vibraciones en medios discretos y
  continuos}.
\newblock PhD thesis, Universidad Complutense de Madrid, 2002.

\bibitem{macia2002lack}
{\sc Maci\`a, F., and Zuazua, E.}
\newblock On the lack of observability for wave equations: a {G}aussian beam
  approach.
\newblock {\em Asympt. Anal. 32}, 1 (2002), 1--26.

\bibitem{marica2015propagation}
{\sc Marica, A., and Zuazua, E.}
\newblock Propagation of 1d waves in regular discrete heterogeneous media: a
  {W}igner measure approach.
\newblock {\em Found. Comp. Math. 15}, 6 (2015), 1571--1636.

\bibitem{markowich1994wigner}
{\sc Markowich, P., Mauser, N., and Poupaud, F.}
\newblock A {W}igner-function approach to (semi) classical limits: Electrons in
  a periodic potential.
\newblock {\em J. Math. Phys. 35}, 3 (1994), 1066--1094.

\bibitem{motamed2010taylor}
{\sc Motamed, M., and Runborg, O.}
\newblock Taylor expansion and discretization errors in {G}aussian beam
  superposition.
\newblock {\em Wave motion 47}, 7 (2010), 421--439.

\bibitem{ralston1982gaussian}
{\sc Ralston, J.}
\newblock Gaussian beams and the propagation of singularities.
\newblock {\em Studies in partial differential equations 23}, 206 (1982), C248.

\bibitem{runborg2007mathematical}
{\sc Runborg, O.}
\newblock Mathematical models and numerical methods for high frequency waves.
\newblock {\em Commun. Comput. Phys 2}, 5 (2007), 827--880.

\bibitem{tanushev2008superpositions}
{\sc Tanushev, N.~M.}
\newblock Superpositions and higher order {G}aussian beams.
\newblock {\em Commun. Math. Sci. 6}, 2 (2008), 449--475.

\bibitem{tanushev2007mountain}
{\sc Tanushev, N.~M., Qian, J., and Ralston, J.~V.}
\newblock Mountain waves and {G}aussian beams.
\newblock {\em Multiscale Model. Simul. 6}, 2 (2007), 688--709.

\bibitem{tartar1990h}
{\sc Tartar, L.}
\newblock H-measures, a new approach for studying homogenisation, oscillations
  and concentration effects in partial differential equations.
\newblock {\em Proc. Roy. Soc. Edinburgh Sec. A 115}, 3-4 (1990), 193--230.

\bibitem{trefethen1982group}
{\sc Trefethen, L.~N.}
\newblock Group velocity in finite difference schemes.
\newblock {\em SIAM rev. 24}, 2 (1982), 113--136.

\bibitem{trefethen1982wave}
{\sc Trefethen, L.~N.}
\newblock {\em Wave propagation and stability for finite difference schemes}.
\newblock PhD thesis, Stanford University, 1982.

\bibitem{vichnevetsky1980propagation}
{\sc Vichnevetsky, R.}
\newblock Propagation properties of semi-discretizations of hyperbolic
  equations.
\newblock {\em Math. Comp. Simul. 22}, 2 (1980), 98--102.

\bibitem{vichnevetsky1981energy}
{\sc Vichnevetsky, R.}
\newblock Energy and group velocity in semi discretizations of hyperbolic
  equations.
\newblock {\em Math. Comp. Simul. 23}, 4 (1981), 333--343.

\bibitem{vichnevetsky1981propagation}
{\sc Vichnevetsky, R.}
\newblock Propagation through numerical mesh refinement for hyperbolic
  equations.
\newblock {\em Math. Comp. Simul. 23}, 4 (1981), 344--353.

\bibitem{vichnevetsky1987wave}
{\sc Vichnevetsky, R.}
\newblock Wave propagation and reflection in irregular grids for hyperbolic
  equations.
\newblock {\em Appl. Numer. Math. 3\/} (1987), 133--166.

\bibitem{vichnevetsky1982fourier}
{\sc Vichnevetsky, R., and Bowles, J.~B.}
\newblock {\em Fourier analysis of numerical approximations of hyperbolic
  equations}, vol.~5.
\newblock Siam, 1982.

\bibitem{weiss1985reflection}
{\sc Weiss, W.~R., and Hagedorn, G.~A.}
\newblock Reflection and transmission of high frequency pulses at an interface.
\newblock {\em Transp. Theory Stat. Phys. 14}, 5 (1985), 539--565.

\bibitem{wigner1932quantum}
{\sc Wigner, E.}
\newblock On the quantum correction for thermodynamic equilibrium.
\newblock {\em Phys. Rev. 40}, 5 (1932), 749.

\bibitem{zuazua2005propagation}
{\sc Zuazua, E.}
\newblock Propagation, observation, control and numerical approximation of
  waves.
\newblock {\em SIAM Rev. 47}, 2 (2005), 197--243.

\end{thebibliography}

\end{document}